\documentclass[10pt,a4paper,twoside]{article}

\usepackage{stmaryrd}
\usepackage{silence}
\usepackage{upgreek}
\usepackage{adjustbox}
\usepackage{amssymb}
\usepackage{faktor}
\usepackage[all,cmtip]{xy}
\usepackage{amsmath,amscd}
\usepackage{mathrsfs}
\usepackage[table]{xcolor}
\usepackage{tikz}
\usepackage{tikz-cd}
\usepackage{enumitem}
\usepackage{mathtools}
\usepackage[mathcal]{euscript}
\usepackage{graphicx}
\usepackage{hyperref}
\usepackage[T1]{fontenc}
\usepackage{ragged2e}
\usepackage{subfigure}
\usepackage{amsthm}

\usetikzlibrary{shapes.geometric}
\hypersetup{colorlinks=true,linkcolor=blue}

\newcommand{\gaps}[1]{\mathcal{G}_{#1}}
\newcommand{\sepgaps}[1]{\widetilde{\mathcal{G}}_{#1}}

\newcommand{\sk}[2]{\mathrm{sk}_{#1}(#2)}
\newcommand{\localstates}[1]{S({#1})}
\newcommand{\reducedlocalstates}[1]{\widetilde{S}({#1})}
\newcommand{\intrinsic}[1]{S(#1)}
\newcommand{\reducedintrinsic}[2]{S^{#1}(#2)}
\newcommand{\fullyreduced}[1]{\widetilde{S}(#1)}
\newcommand{\tailreduced}[1]{S^\partial(#1)}
\newcommand{\starreduced}[1]{S^{\Yup}(#1)}

\newcommand{\bimod}[1]{\mathcal{B}\mathrm{imod}_{#1}}

\newcommand{\isotopyinvariant}[1]{\mathcal{I}\mathrm{nvt}_{#1}^\mathrm{iso}}

\newcommand{\DeltaZero}{\mathbf{1}}
\newcommand{\Cx}{\mathcal{C}\mathrm{x}}
\newcommand{\SD}{\mathrm{SD}}

\newcommand{\sing}{\mathrm{sing}}

\DeclareMathOperator{\forgettodomain}{forget}

\newcommand{\Flw}{\mathcal{F}\mathrm{lw}}

\newcommand{\XX}{\mathcal{X}}
\newcommand{\YY}{\mathcal{Y}}
\newcommand{\Gph}{\mathcal{G}\mathrm{ph}}

\DeclareMathOperator{\Lan}{Lan}
\DeclareMathOperator{\hoLan}{hoLan} 
\newcommand{\Conf}{\mathrm{Conf}}
\newcommand{\Ch}{\mathcal{C}\mathrm{h}}

\newcommand{\Emb}{\mathcal{E}\mathrm{mb}}
\newcommand{\Top}{\mathcal{T}\mathrm{op}}

\newcommand{\Fun}{\mathrm{Fun}}
\newcommand{\id}{\mathrm{id}}
\newcommand{\Uu}{\mathcal{U}}
\newcommand{\Ind}{\mathrm{Ind}}

\newcommand{\subd}{\mathcal{P}}
\newcommand{\D}{\mathcal{D}}

\newcommand{\E}{\mathcal{E}}

\newcommand{\F}{\mathcal{F}}

\newcommand{\Cc}{\mathcal{C}}
\newcommand{\X}{\mathcal{X}}
\newcommand{\Y}{\mathcal{Y}}

\DeclareMathOperator*{\im}{\mathrm{im}}

\newcommand{\pt}{\mathrm{pt}}

\DeclareMathOperator*{\hocolim}{\mathrm{hocolim}}

\DeclareMathOperator*{\colim}{\mathrm{colim}}

\definecolor{coloryellow}{RGB}{240,228,66}
\definecolor{colorskyblue}{RGB}{86,180,233}
\definecolor{colorvermillion}{RGB}{213,94,0}

\DeclareMathOperator{\coker}{coker}
\DeclareMathOperator{\Hom}{Hom}
\DeclareMathOperator{\rk}{rk}

\newcommand{\netgraph}{\graphfont{N}}
\newcommand{\graphfont}{\mathsf}

\newcommand{\thetagraph}[1]{\graphfont{\Theta}_{#1}}
\newcommand{\completegraph}[1]{\graphfont{K}_{#1}}
\newcommand{\intervalgraph}{\graphfont{I}}

\newcommand{\stargraph}[1]{\graphfont{S}_{#1}}
\newcommand{\graf}{\graphfont{\Gamma}}
\newcommand{\Xigraph}{\graphfont{\Xi}}
\newcommand{\treegraph}{\graphfont{T}}
\newcommand{\lollipopgraph}[1]{\graphfont{L}_{#1}}
\newcommand{\figureeightgraph}{\graphfont{8}}
\newcommand{\cyclegraph}[1]{\graphfont{C}_{#1}}

\DeclareMathOperator{\tor}{Tor}

\newcommand{\boundary}{\partial}

\newtheorem*{decomposition}{Decomposition Theorem}
\newtheorem*{comparison}{Comparison Theorem}

\theoremstyle{definition}
\newtheorem{definition}{Definition}[section]
\newtheorem{notation}[definition]{Notation}

\newtheorem{example}[definition]{Example}

\newtheorem{construction}[definition]{Construction}

\theoremstyle{plain}
\newtheorem{proposition}[definition]{Proposition}
\newtheorem{lemma}[definition]{Lemma}
\newtheorem{corollary}[definition]{Corollary}
\newtheorem{theorem}[definition]{Theorem}

\theoremstyle{remark}
\newtheorem{remark}[definition]{Remark}

\makeatletter
\@addtoreset{definition}{section}
\makeatother

\def\YEAR{\year}\newcount\VOL\VOL=\YEAR\advance\VOL by-1995
\def\firstpage{1}\def\lastpage{1000}
\def\received{}\def\revised{}
\def\communicated{}

\makeatletter
\def\magnification{\afterassignment\m@g\count@}
\def\m@g{\mag=\count@\hsize6.5truein\vsize8.9truein\dimen\footins8truein}
\makeatother

\font\eightrm=cmr8
\font\caps=cmcsc10                    
\font\Caps=cmcsc10 scaled \magstep1   
\font\scaps=cmcsc8

\makeatletter
\@specialpagefalse\headheight=8.5pt
\def\DocMath{{\def\th{\thinspace}\scaps Documenta Math.}}
\renewcommand{\@oddfoot}{\hfill\scaps Documenta Mathematica 
    \number\VOL\  (\number\YEAR) \number\firstpage--\lastpage\hfill}
\renewcommand{\@evenfoot}{\ifnum\thepage>\lastpage\hfill\scaps
    Documenta Mathematica \number\VOL\  (\number\YEAR)\hfill\else\@oddfoot\fi}%

\renewcommand{\@evenhead}{%
    \ifnum\thepage>\lastpage\rlap{\thepage}\hfill%
    \else\rlap{\thepage}\slshape\leftmark\hfill{\caps\SAuthor}\hfill\fi}%
\renewcommand{\@oddhead}{%
    \ifnum\thepage=\firstpage{\DocMath\hfill\llap{\thepage}}%
    \else{\slshape\rightmark}\hfill{\caps\STitle}\hfill\llap{\thepage}\fi}%
\makeatother

\def\TSkip{\bigskip}
\newbox\TheTitle{\obeylines\gdef\GetTitle #1
\ShortTitle  #2
\SubTitle    #3
\Author      #4
\ShortAuthor #5
\EndTitle
{\setbox\TheTitle=\vbox{\baselineskip=20pt\let\par=\cr\obeylines%
\halign{\centerline{\Caps##}\cr\noalign{\medskip}\cr#1\cr}}%
    \copy\TheTitle\TSkip\TSkip%
\def\next{#2}\ifx\next\empty\gdef\STitle{#1}\else\gdef\STitle{#2}\fi%
\def\next{#3}\ifx\next\empty%
    \else\setbox\TheTitle=\vbox{\baselineskip=20pt\let\par=\cr\obeylines%
    \halign{\centerline{\caps##} #3\cr}}\copy\TheTitle\TSkip\TSkip\fi%
\centerline{\caps #4}\TSkip\TSkip%
\def\next{#5}\ifx\next\empty\gdef\SAuthor{#4}\else\gdef\SAuthor{#5}\fi%
\ifx\received\empty\relax
    \else\centerline{\eightrm Received: \received}\fi%
\ifx\revised\empty\TSkip%
    \else\centerline{\eightrm Revised: \revised}\TSkip\fi%
\ifx\communicated\empty\relax
    \else\centerline{\eightrm Communicated by \communicated}\fi\TSkip\TSkip%
\catcode'015=5}}\def\Title{\obeylines\GetTitle}
\def\Abstract{\begingroup\narrower
    \parskip=\medskipamount\parindent=0pt{\caps Abstract. }}
\def\EndAbstract{\par\endgroup\TSkip}

\long\def\MSC#1\EndMSC{\def\arg{#1}\ifx\arg\empty\relax\else
     {\par\narrower\noindent%
     2010 Mathematics Subject Classification: #1\par}\fi}

\long\def\KEY#1\EndKEY{\def\arg{#1}\ifx\arg\empty\relax\else
    {\par\narrower\noindent Keywords and Phrases: #1\par}\fi\TSkip}

\newbox\TheAdd\def\Addresses{\vfill\copy\TheAdd\vfill
    \ifodd\number\lastpage\vfill\eject\phantom{.}\vfill\eject\fi}
{\obeylines\gdef\GetAddress #1
\Address #2 
\Address #3
\Address #4
\EndAddress
{\def\xs{4.3truecm}\parindent=0pt
\setbox0=\vtop{{\obeylines\hsize=\xs#1\par}}\def\next{#2}
\ifx\next\empty 
     \setbox\TheAdd=\hbox to\hsize{\hfill\copy0\hfill}
\else\setbox1=\vtop{{\obeylines\hsize=\xs#2\par}}\def\next{#3}
\ifx\next\empty 
     \setbox\TheAdd=\hbox to\hsize{\hfill\copy0\hfill\copy1\hfill}
\else\setbox2=\vtop{{\obeylines\hsize=\xs#3\par}}\def\next{#4}
\ifx\next\empty\ 
     \setbox\TheAdd=\vtop{\hbox to\hsize{\hfill\copy0\hfill\copy1\hfill}
                \vskip20pt\hbox to\hsize{\hfill\copy2\hfill}}
\else\setbox3=\vtop{{\obeylines\hsize=\xs#4\par}}
     \setbox\TheAdd=\vtop{\hbox to\hsize{\hfill\copy0\hfill\copy1\hfill}
            \vskip20pt\hbox to\hsize{\hfill\copy2\hfill\copy3\hfill}}
\fi\fi\fi\catcode'015=5}}\gdef\Address{\obeylines\GetAddress}

\def\LOCAL{\jobname.files}

\begin{document}
\Title Subdivisional spaces and graph braid groups
\ShortTitle 
\SubTitle   
\Author Byung Hee An, Gabriel C. Drummond-Cole, and Ben Knudsen
\ShortAuthor An, Drummond-Cole, and Knudsen
\EndTitle
\Abstract We study the problem of computing the homology of the configuration spaces of a finite cell complex $X$. We proceed by viewing $X$, together with its subdivisions, as a \emph{subdivisional space}---a kind of diagram object in a category of cell complexes. After developing a version of Morse theory for subdivisional spaces, we decompose $X$ and show that the homology of the configuration spaces of $X$ is computed by the derived tensor product of the Morse complexes of the pieces of the decomposition, an analogue of the monoidal excision property of factorization homology. 

Applying this theory to the configuration spaces of a graph, we recover a cellular chain model due to \'{S}wi\k{a}tkowski. Our method of deriving this model enhances it with various convenient functorialities, exact sequences, and module structures, which we exploit in numerous computations, old and new.

\EndAbstract
\MSC primary: 20F36, 55R80, 55U05; secondary: 05C10
\EndMSC
\KEY configuration spaces, cell complexes, subdivisional spaces, graphs, braid groups
\EndKEY
\Address Byung Hee An\\Center for Geometry\\ and Physics\\Institute for Basic Science\\Pohang 37673\\Republic of Korea
\Address Gabriel C. Drummond-Cole\\Center for Geometry\\ and Physics\\Institute for Basic Science\\Pohang 37673\\Republic of Korea
\Address Ben Knudsen\\Department of Mathematics\\Harvard University\\Cambridge 02138\\USA
\Address
\EndAddress

\setcounter{tocdepth}{1}
\tableofcontents
\section{Introduction}

Consider the following problem: given a cell structure on a space $X$, compute the homology of the configuration space \[B_k(X)\coloneqq{}\{(x_1,\ldots, x_k)\in X^k: x_i\neq x_j\text{ if }i\neq j\}/\Sigma_k,\]where $\Sigma_k$ is the symmetric group on $\{1,\ldots, k\}$. 
In this work, we provide new tools to address this problem by combining ideas from homotopy theory and robotics. We apply these tools to study the homology of configuration spaces of graphs.

\subsection{Configuration spaces and gluing} First introduced in \cite{FadellNeuwirth:CS}, configuration spaces of manifolds have a long and rich history in algebraic topology---for a taste, we direct the reader to \cite{Arnold:CRGDB}, \cite{Segal:CSILS}, \cite{CohenLadaMay:HILS}, \cite{Totaro:CSAV}, and \cite{LongoniSalvatore:CSANHI}. One powerful idea in this area has been the idea that configuration spaces behave predictably under collar gluings; that is, if we decompose a manifold $M$ into two open submanifolds $M_0$ and $M_1$ glued along the thickening of an embedded codimension one submanifold $N$, then the configuration spaces of $M$ are in some sense determined by the configuration spaces of these smaller manifolds. This decomposition technique has borne considerable fruit---see \cite{McDuff:CSPNP}, \cite{Bodigheimer:SSMS}, \cite{BodigheimerCohen:RCCSS}, \cite{BodigheimerCohenTaylor:OHCS}, \cite{FelixThomas:RBNCS}, and~\cite{Knudsen:BNSCSVFH}. 

One articulation of this gluing property comes from the theory of factoriztion homology, which provides a quasi-isomorphism \[
C_*(B(M))\simeq C_*(B(M_0))\bigotimes^\mathbb{L}_{C_*(B(N\times\mathbb{R}))}C_*(B(M_1)),
\] where $B(M):=\coprod_{k\geq0} B_k(M)$ \cite{AyalaFrancis:FHTM}. Unfortunately, these chain complexes are only algebras and modules up to structured homotopy, so this quasi-isomorphism may be difficult to use in computations. To address this issue, we will combine this gluing property with ideas drawn from a different school of thought.

\subsection{Configuration spaces and cell structures} Since their introduction in \cite{GhristKoditschek:SCRPVDOG} and \cite{Ghrist:CSBGGR}, configuration spaces of graphs have been studied intensively---see \cite{Abrams:CSBGG}, \cite{Farber:TCMP}, and \cite{Farber:CFMPG} and the references therein. According to a theorem of \cite{Ghrist:CSBGGR}, the configuration spaces of a graph $\graf$ are all classifying spaces for their fundamental groups, the so-called \emph{graph braid groups} of $\graf$.

The key to approaching these spaces is to notice that the cell structure of a graph $\graf$ yields a cellular approximation $B_k^\Box(\graf)\to B_k(\graf)$, which becomes a homotopy equivalence after finite subdivision \cite{Abrams:CSBGG}. Thus, one may apply the Morse theory for cell complexes developed in \cite{Forman:MTCC}. This approach has led to many computations---see \cite{FarleySabalka:DMTGBG}, \cite{Farley:HTBG}, \cite{FarleySabalka:OCRTBG}, \cite{FarleySabalka:PGBG}, \cite{KimKoPark:GBGRAAG}, \cite{Sabalka:ORIPTBG}, and \cite{KoPark:CGBG}.

There are two immediate obstacles to extending the success of cellular methods in the case of a graph to higher dimensional complexes. First, the geometry in higher dimensions may be too difficult to handle directly. To address this problem, we will use the cut-and-paste idea employed in the manifold case, working with locally defined Morse data. Second, the approximation $B_k^\Box$ may fail to capture the homotopy type of the configuration space in higher dimensions, even after finite subdivision. 

\subsection{Subdivisional spaces and decomposition} Fortunately, the cellular approximation does improve with subdivision, and, by applying the method of \cite{Abrams:CSBGG}, we prove in Theorem~\ref{thm:discrete approximation} that, for any suitably \emph{convergent} set $\subd$ of subdivisions of $X$, the map \[\colim_{X'\in \subd}B_k^\Box(X')\to B_k(X)\] is a weak homotopy equivalence. Thus, in the presence of such a convergent \emph{subdivisional structure} $\subd$, the study of the configuration spaces of $X$ is equivalent to the study of the diagram of configuration complexes determined by $\subd$.

These diagrammatics are formalized within the framework of \emph{subdivisional spaces,} which behave rigidly in some ways and like continuous objects in others. In particular, cellular chains and an abstract form of Morse theory lift effortlessly to the subdivisional context, whereas the \emph{subdivisional configuration space} $B^\SD(X)$ captures the homotopy type of $B(X)$. Our first main result is a gluing theorem in this context.

\begin{decomposition}[Theorem~\ref{thm:decomposition}]
Given a decomposition $X\cong X_0\amalg_{A\times I} X_1$ as subdivisional spaces, together with suitable locally-defined Morse data, the homology of $B(X)$ is computed by the derived tensor product of Morse complexes \[
I(B^\SD(X_0))\bigotimes^\mathbb{L}_{I(B^\SD(A\times I))}I(B^\SD(X_1)).
\]
\end{decomposition}

\subsection{Swiatkowski complexes} In the second half of the paper, we apply this theory to the semi-classical case of graphs. What results is a family of chain complexes computing the homology of graph braid groups, which first appeared in the work of \'{S}wi\k{a}tkowski \cite{Swiatkowski:EHDCSG}. Let $\graf$ be a graph with vertices $V$ and edges $E$. We set 
\[\intrinsic{\graf}\coloneqq{}\mathbb{Z}[E]\otimes \bigotimes_{v\in V}\localstates{v},\] 
where $\localstates{v}$ is the free Abelian group generated by $\{\varnothing, v\}\amalg H(v)$. Here $H(v)$ is the set of half-edges at $v$. The \emph{\'{S}wi\k{a}tkowski complex} of $\graf$ is the (bigraded differential) $\mathbb{Z}[E]$-module $\intrinsic{\graf}$---see \S\ref{section:intrinsic complex} for details. Our second main result is the following:

\begin{comparison}[Theorem~\ref{thm:comparison}]
There is a natural isomorphism of bigraded Abelian groups\[H_*(B(\graf))\cong H_*(\intrinsic{\graf}).\] 
\end{comparison}

To derive Theorem~\ref{thm:comparison} from the decomposition theorem, we fragment $\graf$ completely. We take $\graf_0$ to be a disjoint union over $v\in V$ of the star graphs $\stargraph{d(v)}$, where $d(v)$ is the number of half-edges incident on $v$. We take $\graf_1$ to be a disjoint union of intervals, one for each edge in $E$. We obtain $\graf$ by gluing these pieces along $2|E|$ disjoint intervals. We define local Morse data by putting the vertex of each star graph ``at the top,'' so that configurations flow down the legs (see \S\ref{section:interval and star flows})---and the resulting Morse complex is a reduced version of the \'{S}wi\k{a}tkowski complex for the star graph. The decomposition theorem gives the isomorphism \[H_*(B(\graf))\cong H_*\left(\bigotimes_{v\in V}\mathbb{Z}[H(v)]\otimes \localstates{v}\bigotimes_{\mathbb{Z}[E]^{\otimes 2}}\mathbb{Z}[E],\,\partial\right),\] and the righthand side is isomorphic to $\intrinsic{\graf}$ by inspection.

\subsection{Homology computations} The \'{S}wi\k{a}tkowski complex has many desirable features. It is finite dimensional in each bidegree and finitely generated as a $\mathbb{Z}[E]$-module. It connects configuration spaces of different cardinalities by the action of $\mathbb{Z}[E]$. It depends only on intrinsic graph theoretic data and requires no choice of subdivision. It decomposes geometrically, assigning a short exact sequence to the removal of a vertex. It is functorial for embeddings among graphs, so relations at the level of atomic subgraphs impose global constraints. Some of these properties are evident or implicit already in~\cite{Swiatkowski:EHDCSG}; several are new. These features amount to a robust computational toolkit, which we exploit extensively in \S\ref{section:tools} and Appendix~\ref{appendix:degree one}. 

\subsection{How to read this paper}The reader concerned mainly with graph braid groups may wish to start with just enough of \S\ref{section:graph braid groups} to see our conventions on graphs and the definition of the \'{S}wi\k{a}tkowski complex before skipping directly to the computations of \S\ref{section:tools}, returning to the theory later. Starting from the beginning is recommended for the reader interested in configuration spaces in general, higher dimensional applications, or variations on the ideas of factorization homology.

\subsection{Relation to previous work}\label{section:relation to previous work} This paper grew out of the desire to combine the local-to-global approach to configuration spaces of graphs promised by the stratified factorization homology developed in \cite{AyalaFrancisTanaka:FHSS} with the combinatorial character and computational ease of the discrete Morse theoretic model of \cite{FarleySabalka:DMTGBG}, following \cite{Abrams:CSBGG}. Although we do not directly employ the results of any of this work, its ideas permeate the theory developed in \S\ref{section:configuration complexes}--\ref{section:graph braid groups}.

The \'Swi\k{a}tkowksi complex first appeared in \cite{Swiatkowski:EHDCSG} (see also \cite{Luetgehetmann:CSG}). There \'{S}wi\k{a}tkowski constructed a cubical complex lying inside $B_k(\graf)$ as a deformation retract to study the fundamental group. The cellular chain complex of this cubical complex is isomorphic to the weight $k$ subcomplex of $S(\graf)$. This observation implies a weaker version of Theorem~\ref{thm:comparison} which contains more direct geometric content. 

A similar edge stabilization mechanism, in a different complex and for trees only, was studied by Ramos \cite{Ramos:SPHTBG}.

The generators and some of the relations that we describe for the first homology groups appear in~\cite{HarrisonKeatingRobbinsSawicki:NPQSG}. That work uses these generators to perform many computations motivated by physics, and it contains another alternative proof of the theorem of \cite{KoPark:CGBG} dealt with in our Appendix~\ref{appendix:degree one}.

\subsection{Future directions} We defer pursuit of the following ideas to future work. 
\begin{enumerate}
\item \emph{Edge stabilization}. In the sequel to this paper \cite{AnDrummond-ColeKnudsen:ESHGBG}, we show that the $\mathbb{Z}[E]$-action is geometric, arising from a new family of stabilization maps at the level of the configuration spaces themselves, and we carry out a detailed investigation of its properties. 
\item \emph{Destabilization}. Dual to the process of adding points is that of splitting configurations apart, which may be phrased as a cocommutative coalgebra structure for which $\mathbb{Z}[E]$ acts by coderivations.
\item \emph{Higher dimensions}. Little research has been done on configuration spaces of higher dimensional cell complexes in general---see \cite{Gal:ECCSC}, \cite{AnPark:OSBGC}, and \cite{Wiltshire-Gordon:MCSSC} for rare examples. Replicating the computational success of the \'{S}wi\k{a}tkowski complex in higher dimensions amounts to identifying tractable local Morse data.
\item \emph{Cup products}. The diagonal is not a cellular embedding, but it is an embedding of subdivisional spaces, so our methods may shed light on the cohomology rings of configuration spaces. This is already very interesting for graphs---see \cite{Sabalka:ORIPTBG}.
\item \emph{Ordered configurations}. Our program translates with minor modifications to the context of ordered configurations, and we expect to recover an enhanced version of the cellular chain complex of the cubical model constructed in~\cite{Luetgehetmann:CSG}.
\end{enumerate}

\subsection{Conventions}
Graded objects are concentrated in non-negative degrees. This restriction is only used in Proposition \ref{prop:dold-kan hocolim}. We write $\Ch_\mathbb{Z}$ for the category of chain complexes. 

Bigradings of modules are by \emph{degree} and \emph{weight}. The braiding isomorphism for a tensor product of modules has a sign which depends on degree and not on weight: if $x$ and $y$ have degree $i$ and $j$, the braiding isomorphism takes $x\otimes y$ to $(-1)^{ij}y\otimes x$. We write $[m]$ for the degree shift functor by $m$ and $\{n\}$ for the weight shift functor by $n$ so that the degree $i$ and weight $j$ component of $M[m]\{n\}$ is the degree $i-m$ and weight $j-n$ component of $M$.

Symmetric monoidal functors are strong monoidal. We use the phrases ``(natural) weak equivalence'' to refer to a (natural) isomorphism in the relevant homotopy category  and the phrases ``(natural) quasi-isomorphism'' to refer to a (natural) chain map which induces an isomorphism on homology groups. 

We write $C^\sing(X)$ for the singular chain complex of the topological space $X$. If $X$ is a CW complex, we denote the cellular chain complex of $X$ by $C(X)$.

\subsection{Acknowledgments}
The second author thanks Felix Boes for sharing computational tools. The third author thanks Daniel L\"{u}tgehetmann for noticing the connection to the work of \'{S}wi\k{a}tkowski \cite{Swiatkowski:EHDCSG}. The authors thank Tomasz Maci\k{a}\.{z}ek for pointing out several related papers.

\section{Subdivisional spaces}\label{section:configuration complexes}\label{section:subdivisional spaces}

Following the ideas of \cite{Abrams:CSBGG}, we approximate the configuration spaces of a cell complex by cell complexes, and we show, in Theorem~\ref{thm:discrete approximation}, that the homotopy types of these approximations often ``converge'' to the homotopy type of the true configuration space under transfinite subdivision. We then introduce the framework of subdivisional spaces, which sets a complex $X$ equipped with a set of subdivisions $\subd$ on equal footing with the corresponding collection of configuration complexes $\{B_k^\Box(X')\}_{X'\in\subd}$. We identify a natural theory of homology for these objects, the complex of subdivisional chains, and we show that it is homotopically well-behaved.

 \subsection{Complexes and subdivision}
If $X$ is a CW complex and $c\subseteq X$ is an $n$-cell, we write $\partial c$ for the image of $\partial D^n$ under a characteristic map for $c$, and we set $\mathring{c}\coloneqq{}c\setminus \partial c$. The $n$-skeleton of $X$, which is to say the union of the cells of $X$ of dimension at most $n$, is denoted $\sk{n}{X}$. A cellular map is a map preserving skeleta.

From now on, a \emph{complex} will be a finite CW complex. We choose to restrict our attention to the finite case for the sake of convenience only.

\begin{definition}
Let $f:X\to Y$ be a cellular map between complexes. We say that $f$ is \begin{enumerate}
\item \emph{regular} if $f$ preserves both closed and open cells;
\item an \emph{isomorphism} if $f$ is regular and bijective;
\item an \emph{embedding} if $f$ is regular and injective;
\item a \emph{subdvision} if $f$ is bijective and preserves subcomplexes;
\item a \emph{subdivisional embedding} if $f$ is injective and preserves subcomplexes.
\end{enumerate}
\end{definition}

Thus, a subdivisional embedding factors via its image into a subdivision followed by an embedding.

Given a complex $X$, we write $\SD(X)$ for the category whose objects are subdivisions of $X$ and whose morphisms are commuting triangles of subdivisions. Note that, since a subdivision is in particular a homeomorphism, there can be at most one morphism in $\SD(X)$ with fixed source and target.

\begin{remark}
Some authors consider a subdivision to be the inverse to what we have defined to be a subdivision. We choose this convention because it matches the direction of the functoriality that will arise naturally in the examples of interest. Modulo this issue of direction, our notion of subdivision is equivalent to that of \cite[Def.~II.6.2]{LundellWeingram:TCWC}.
\end{remark}

\subsection{Configuration complexes}

We now introduce the main object of study.

\begin{definition}
Let $X$ be a topological space. \begin{enumerate}
\item The $k$th ordered \emph{configuration space} of $X$ is \[\Conf_k(X)=\{(x_1,\ldots, x_k):x_i\neq x_j\text{ if }i\neq j\}.\] \item The $k$th \emph{unordered} configuration space of $X$ is the quotient \[B_k(X)=\Conf_k(X)/\Sigma_k.\]
\end{enumerate}
\end{definition}

Unfortunately, a cell structure on $X$ does not induce an obvious cell structure on $\Conf_k(X)$; however, following \cite{Abrams:CSBGG}, there is a cellular approximation.

\begin{definition}
Let $X$ be a complex. \begin{enumerate}
\item The $k$th ordered \emph{configuration complex} of $X$ is the largest subcomplex $\Conf^\Box_k(X)\subseteq X^k$ contained in $\Conf_k(X)$. 
\item The $k$th unordered configuration complex of $X$ is the quotient \[B_k^\Box(X)=\Conf_k^\Box(X)/\Sigma_k.\]
\end{enumerate}
\end{definition}

In other words, a cell $(c_1,\ldots, c_k)$ of $X^k$ lies in $\Conf^\Box_k(X)$ if and only if $c_i\cap c_j=\varnothing=\varnothing$ for $i\ne j$.

The quotient $B_k^\Box(X)$ is again a complex, which we view as an approximation to $B_k(X)$. These approximations enjoy a certain functoriality.

\begin{lemma}\label{lem:conf subdivision}
Let $s:X\to X'$ be a subdivision. The restriction of $s^k$ to $\Conf_k^\Box(X)$ factors $\Sigma_k$-equivariantly through $\Conf_k^\Box(X')$ as a subdivisional embedding.
\end{lemma}

Our next result, which will serve as a replacement for \cite[Thm.~2.1]{Abrams:CSBGG}, asserts that, after perhaps transfinite subdivision, this combinatorial approximation is faithful. It will be useful to state this result in some generality, and we introduce the following notion in order to do so.

\begin{definition}\label{defi:convergent}
Let $X$ be a complex. A \emph{subdivisional structure} on $X$ is a subcategory $\subd\subseteq \SD(X)$ that is non-empty, full, and filtered. We say that the subdivisional structure $\subd$ is \emph{convergent} if there is a metric on $X$ such that \[\lim_{X'\in\subd}\max_{c\subseteq X'}\mathrm{diam}(c)=0.\]
\end{definition}

\begin{example}
If $X$ is regular, then the barycentric subdivisions of $X$ form a convergent subdivisional structure---see \cite[Thm.~III.1.7]{LundellWeingram:TCWC}.
\end{example}

\begin{theorem}[Convergence theorem]\label{thm:discrete approximation}
For any complex $X$, any convergent $\subd\subseteq \SD(X)$, and any $k\geq0$, the natural map \[\colim_{X'\in\subd}\Conf_k^\Box(X')\longrightarrow \Conf_k(X)\] is a homeomorphism and, in particular, a weak homotopy equivalence.
\end{theorem}
\begin{proof}
Convergence implies that any point of $\Conf_k(X)$ lies in $\Conf_k^\Box(X')$ for some subdivision $X'\in \subd$, i.e., the collection $\{\Conf_k^\Box(X'):X'\in \subd\}$ forms a closed cover of the configuration space.
\end{proof}

Since taking homology commutes with the filtered colimit in question, this observation provides a way of understanding the homology of the configuration spaces of $X$ in terms of the homology of its configuration complexes.

\begin{remark}
The content of the convergence theorem is that configuration complexes are useful in the study of configuration spaces whenever $X$ admits a convergent subdivisional structure, and it would be interesting to know for which $X$ this is the case.
\end{remark}

\begin{remark}
It is natural to wonder whether finite subdivision suffices to recover the correct homotopy type. We do not address this question here, noting only that the method of proof completely breaks down; indeed, the homotopy groups of $\Conf_2(D^3)\simeq S^2$ are all non-zero above degree one \cite{IvanovMikhailovWu:ONCHGS}.
\end{remark}

We close by noting that the configuration complexes interact predictably with disjoint unions.

\begin{lemma}\label{lem:box monoidal}
There is a natural commuting diagram 
\[\begin{tikzcd}
\displaystyle\coprod_{i+j=k}B_i^\Box(X)\times B_j^\Box(Y)\ar[start anchor={[yshift=2.5ex]},d]\ar{r}{\simeq}& B_k^\Box(X\amalg Y)\ar[d]\\
\displaystyle\coprod_{i+j=k}B_i(X)\times B_j(Y)\ar{r}{\simeq}& B_k(X\amalg Y)
\end{tikzcd}
\] in which the bottom arrow is a homeomorphism and the top a cellular isomorphism.
\end{lemma}

\subsection{Subdivisional spaces} 

Complexes and subdivisional embeddings form a category which we denote by $\Cx^\SD$.

\begin{definition}\label{defi:subdivisional space}
A \emph{subdivisional space} is a functor $\X:\subd\to \Cx^\SD$ with $\subd$ a filtered category.
\end{definition}

We write $\Emb^\SD$ for the category whose objects are subdivisional spaces and whose morphisms are given by \[\Hom_{\Emb^\SD}(\X, \Y)=\lim_{p\in\subd}\colim_{q\in\mathcal{Q}}\Hom_{\Cx^\SD}(\XX(p), \YY(q)).\] In other words, $\Emb^\SD$ is the category of \emph{ind-objects} $\Ind(\Cx^\SD).$ We shall make very little use of the general theory of ind-objects, but the reader looking for further information on the subject may consult \cite[Ch.~6]{KashiwaraSchapira:CS}.

\begin{remark}
The formula for hom sets in $\Ind(\Cc)$ is derived from the following intuitions: \begin{enumerate}
\item an object of $\Cc$ should determine an ind-object of $\Cc$;
\item a general ind-object of $\Cc$ should be a filtered colimit of objects of $\Cc$; and
\item objects of $\Cc$ should be compact as ind-objects.
\end{enumerate}
\end{remark}

We say that $\X:\subd\to \Cx^\SD$ is \emph{indexed} on $\subd$. Subdivisional spaces indexed on different categories may be isomorphic.

\begin{example}\label{example:subdivisional space from a complex}
A subdivisional structure $\subd\subseteq \SD(X)$ determines a subdivisional space $\X:\subd\to \Cx^\SD$ sending $X\to X'$ to $X'$. 
\end{example}

Most complexes admit many subdivisional structures and many non-isomorphic realizations as subdivisional spaces. Roughly, we imagine that a subdivisional structure $\subd$ forces $X$ to be isomorphic to each of the subdivisions contained in $\subd$.

\begin{example}
Since the product of filtered categories is filtered, the levelwise Cartesian product and disjoint union of two subdivisional spaces is again a subdivisional space. Note that the former is not the categorical product, nor is the latter the categorical coproduct; indeed, since we work with embeddings, this failure is already present in $\Cx^\SD$.
More explicitly, there is in general no embedding from a disjoint union extending embeddings on its components because the images may intersect nontrivially.
On the other hand, the projections of an embedding into a product are not necessarily themselves embeddings. 
\end{example}

\begin{definition}
The \emph{spatial realization} functor $|-|$ is the composite \[\Emb^\SD=\Ind(\Cx^\SD)\to \Ind(\Top)\xrightarrow{\colim}\Top.\] 
\end{definition}

\begin{example}
If $\X$ is any of the subdivisional spaces of Example~\ref{example:subdivisional space from a complex}, there is a canonical homeomorphism $|\X|\cong X$.
\end{example}

According to Lemma~\ref{lem:conf subdivision}, configuration complexes are functorial for subdivisional embeddings, so we may make the following definition.

\begin{definition}\label{defi: subdivisional configuration} Let $\X:\subd\to \Cx^\SD$ be a subdivisional space. The $k$th ordered \emph{subdivisional configuration space} of $\X$ is the subdivisional space
\[\subd\xrightarrow{\X}\Cx^\SD\xrightarrow{\Conf_k^\Box}\Cx^\SD.\]
\end{definition}

Similarly, we have the $k$th unordered subdivisional configuration space $B_k^\SD(\X)$. We will be particularly interested in this construction when $\X$ comes from a complex $X$ with a convergent subdivisional structure, for in this case Theorem~\ref{thm:discrete approximation} gives a weak homotopy equivalence \[|B_k^\SD(\X)|\xrightarrow{\sim}B_k(X).\]

It will often be convenient to consider configuration spaces of all finite cardinalities simultaneously, and we set \[B^\SD(\X)\coloneqq{}\coprod_{k\geq0} B_k^\SD(\X)\] (note that the indicated disjoint union does in fact exist in $\Emb^\SD$, since $\Emb^\SD$ admits filtered colimits). Thus, we have a functor \[B^\SD:\Emb^\SD\to \Emb^\SD.\] Using Lemma~\ref{lem:box monoidal}, we see that $B^\SD$ naturally carries the structure of a symmetric monoidal functor, where $\Emb^\SD$ is symmetric monoidal under disjoint union in the domain and Cartesian product in the codomain.

\subsection{Subdivisional chains}

Since the complex of cellular chains is functorial for cellular maps between complexes, and, in particular, for subdivisional embeddings, we may make the following definition.

\begin{definition}
The functor $C^\SD$ of \emph{subdivisional chains} is the composite \[\Emb^\SD=\Ind(\Cx^\SD)\xrightarrow{\Ind(C)} \Ind(\Ch_\mathbb{Z})\xrightarrow{\colim}\Ch_\mathbb{Z}.\]
\end{definition}

Viewing $\Emb^\SD$ and $\Ch_\mathbb{Z}$ as symmetric monoidal under Cartesian product and tensor product, respectively, $C^\SD$ naturally carries the structure of a symmetric monoidal functor (here we use that cellular chains sends products to tensor products and that the tensor product distributes over colimits).

\begin{remark}
In contrast to the cellular chain complex, $C^\SD(\X)$ is typically a very large object. For example, if $\X$ is obtained by equipping the interval $I$ with the subdivisional structure $\SD(I)$, then $C^\SD(\X)$ is uncountably generated.
\end{remark}

The fundamental fact about the functor of subdivisional chains is the following.

\begin{proposition}\label{prop:computes homology}
There is a natural weak equivalence
\[C^\SD(-)\simeq C^\sing(|-|)\] of functors from $\Emb^\SD$ to chain complexes.
\end{proposition}

In the proof, we use the following intermediary.

\begin{construction}
Let $X$ be a CW complex. We define $S^\mathrm{cell}(X)$ to be the simplicial set given in simplicial degree $n$ by the set of \emph{cellular} maps $\sigma:\Delta^n\to X$. This construction is functorial for cellular maps. There is an inclusion $S^\mathrm{cell}(X)\to S(X)$, where $S$ is the standard functor of singular simplices, and this map is a weak homotopy equivalence of simplicial sets by the cellular approximation theorem. 
Since both the induced map on geometric realizations and the composite $|S^\mathrm{cell}(X)|\to |S(X)|\to X$ are cellular by construction, we obtain the zig-zag 
\[C(X)\xleftarrow{\sim} C(|S^\mathrm{cell}(X)|)\xrightarrow{\sim} C(|S(X)|) \cong C^\mathrm{sing}(|X|)
\]
\end{construction}

\begin{proof}[Proof of Proposition~\ref{prop:computes homology}] 
Consider the following diagram:
\[\begin{tikzcd}
\Ind(\Cx^{\SD})\ar[swap, drrr,near start,"\Ind(C)",bend right]\ar{drrr}[description]{\Ind(C(|S^{\mathrm{cell}}|))}\ar{rrr}&&&\Ind(\Top)\ar{d}{\Ind(C^{\mathrm{sing}})}\ar{r}{\colim}&\Top\ar{d}{C^{\mathrm{sing}}}
\\
&&&\Ind(\Ch_{\mathbb{Z}})\ar{r}{\colim}&\Ch_{\mathbb{Z}}.
\end{tikzcd}\]
The composition along the top and right is $C^{\mathrm{sing}}(|- |)$. The composition along the bottom is $C^{\SD}$. The right square commutes up to natural isomorphism because the colimit is filtered. The triangle and bigon on the left commute up to natural objectwise quasi-isomorphism. Since filtered colimits preserve quasi-isomorphisms, the conclusion follows.
\end{proof}

Thus there is a homotopically well-behaved natural isomorphism 
\[H_*(C^\SD(\X))\cong H_*(|\X|).\] 
To make this statement precise, we use \emph{homotopy colimits} for diagrams of chain complexes---reminders and references are in Appendix~\ref{section:hocolim appendix}. The following corollary will be a key ingredient in our proof of Theorem~\ref{thm:decomposition} below.

\begin{corollary}\label{cor:hocolim}
Let $\X$ be a subdivisional space and $F:\D\to \Emb^\SD$ a functor equipped with a natural transformation to the constant functor at $\X$. If the induced map $\hocolim_\D|F|\to |\X|$ is a weak homotopy equivalence, then the induced map \[\hocolim_\D C^\SD(F)\rightarrow C^\SD(\X)\] is a quasi-isomorphism.
\end{corollary}
\begin{proof}
Applying homology to the natural weak equivalence of Proposition~\ref{prop:computes homology}, and using the fact that a levelwise weak homotopy equivalence of functors induces a weak homotopy equivalence on homotopy colimits (Lemma~\ref{lemma: hocolim and weak equivalence}), we obtain the isomorphisms in the following commutative square:
\[
\begin{tikzcd}
H_*\left(\hocolim_\D C^\sing(|F|)\right)\ar[r]\ar{d}[rotate=90, above, xshift=1pt
]{\backsimeq}
 & H_*(C^\sing(|\X|))\ar{d}[rotate=-90, above, xshift=-1pt
]{\simeq}\\
H_*\left(\hocolim_\D C^\SD(F)\right)\ar[r]&H_*(C^\SD(\X)).
\end{tikzcd}
\] From our assumption and Proposition~\ref{prop:sing commutes with hocolim}, the top map is an isomorphism, and the claim follows.
\end{proof}

\section{Decomposition}

We prove our first main result, the decomposition theorem, stated below as Theorem~\ref{thm:decomposition}. A careful formulation of this result requires that we supply a certain amount of definitional groundwork, and this task will occupy our attention in \S\ref{section:decomposition}--\ref{subsec:subdivisional morse theory}. The theorem and its proof appear in \S\ref{subsec:decomposition theorem}.

The proof of the decomposition theorem is premised on various manipulations of homotopy colimits. For the convenience of the reader less familiar with categorical homotopy theory, we have included a brief review of the relevant terminology and results in Appendix~\ref{section:hocolim appendix}, as well as a number of references.

\subsection{Decompositions and gaps}\label{section:decomposition}

We first make precise the data involved in the type of decomposition that we wish to consider. Before doing so, we remind the reader of two operations on subdivisions. First, by restricting in the source and target, a subdivision yields a subdivision on any subcomplex. Second, the Cartesian product of two subdivisions is a subdivision of the Cartesian product. Both of these constructions respect further subdivision, so a subdivisional structure on a complex yields a subdivisional structure on any subcomplex by restriction, and any subdivisional structures on two complexes yield a subdivisional structure on their product.

\begin{definition}\label{defi:decomposition}
An ($r$-fold) \emph{decomposition} of the complex $X$ is the data of \begin{enumerate}
\item a collection of complexes $\{\widetilde X_0,\ldots, \widetilde X_r, A_1, \ldots, A_r\}$;
\item for each $0\leq j\leq r$, a pair of embeddings $A_{j}\to \widetilde X_j\leftarrow A_{j+1}$ with disjoint images, where $A_0=A_{r+1}=\varnothing$ by convention;
\item an isomorphism \[X\cong \widetilde X_0\coprod_{A_1\times\{0\}} (A_1\times I)\coprod_{A_1\times\{1\}} \cdots\coprod_{A_{r}\times\{0\}}(A_r\times I)\coprod_{A_r\times\{1\}} \widetilde X_r;\]
\item a subdivisional structure $\subd_X\subseteq \SD(X)$ restricting to a product of subdivisional structures $\subd_{A_j}\times \SD(I)$ on $A_j\times I$ for each $1\leq j\leq r$.
\end{enumerate}
We say that the decomposition is \emph{convergent} if $\subd_X$ is so.
\end{definition}

Given a decomposition, we set \[X_j\coloneqq{}(A_j\times I)\coprod_{A_j\times\{1\}}\widetilde X_j\coprod_{A_{j+1}\times\{0\}} (A_{j+1}\times I).\] These complexes are called the \emph{components} of the decomposition, and the complexes $A_j\times I$ are the \emph{bridges.}

\begin{remark}
Essentially all of our results hold if $\SD(I)$ is replaced with a convergent subdivisional structure on $I$.
\end{remark}

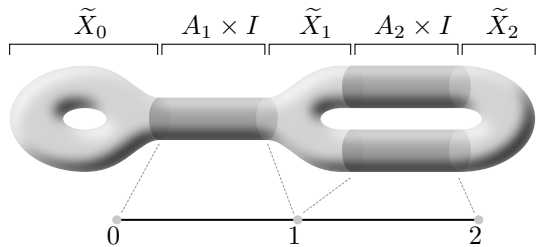
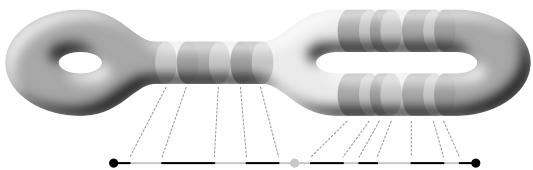
\begin{figure}[ht]{
\begin{center}
\subfigure[A $2$-fold decomposition $\E$ and the map $\pi_\E$]{
\begingroup%
  \makeatletter%
  \ifx\svgwidth\undefined%
    \setlength{\unitlength}{227.9999994bp}%
    \ifx\svgscale\undefined%
      \relax%
    \else%
      \setlength{\unitlength}{\unitlength * \real{\svgscale}}%
    \fi%
  \else%
    \setlength{\unitlength}{\svgwidth}%
  \fi%
  \global\let\svgwidth\undefined%
  \global\let\svgscale\undefined%
  \makeatother%
  \begin{picture}(1,0.35557315)%
    \put(0,0){\includegraphics[width=\unitlength,page=1]{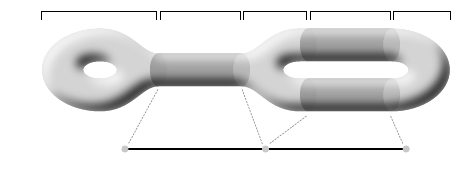}}%
    \put(0.19298246,0.34890854){\color[rgb]{0,0,0}\makebox(0,0)[lb]{\smash{$\widetilde X_0$}}}%
    \put(0.36842105,0.34890854){\color[rgb]{0,0,0}\makebox(0,0)[lb]{\smash{$A_1\times I$}}}%
    \put(0.56140351,0.34890854){\color[rgb]{0,0,0}\makebox(0,0)[lb]{\smash{$\widetilde X_1$}}}%
    \put(0.68421053,0.34890854){\color[rgb]{0,0,0}\makebox(0,0)[lb]{\smash{$A_2\times I$}}}%
    \put(0.86666667,0.34890854){\color[rgb]{0,0,0}\makebox(0,0)[lb]{\smash{$\widetilde X_2$}}}%
    \put(0.24671053,0.00241731){\color[rgb]{0,0,0}\makebox(0,0)[lb]{\smash{$0$}}}%
    \put(0.54276316,0.00241731){\color[rgb]{0,0,0}\makebox(0,0)[lb]{\smash{$1$}}}%
    \put(0.83881579,0.00241731){\color[rgb]{0,0,0}\makebox(0,0)[lb]{\smash{$2$}}}%
  \end{picture}%
\endgroup
}
\subfigure[A gap and the complement of its preimage under $\pi_\E$]{\includegraphics{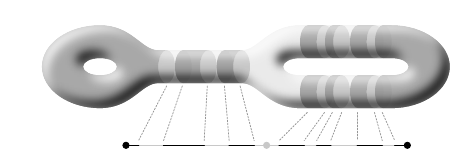}}
\end{center}
}
\vspace{-1em}
\caption{A decomposition and a gap}
\label{figure: decomposition and basics}
\end{figure}

We typically abbreviate the data of a decomposition to the letter $\E$. Note that, by restriction, a decomposition determines subdivisional structures on all of the complexes involved, and every inclusion between two such lifts to a morphism of subdivisional spaces.

\begin{definition} Let $\E$ and $\F$ be $r$-fold decompositions of $X$ and $Y$, respectively. A \emph{map of decompositions} from $\E$ to $\F$ is a map $f:X\to Y$ of subdivisional spaces whose restrictions fit into a commuting diagram of subdivisional spaces
\[\begin{tikzcd}
\widetilde X_0\ar[d]&A_1\ar[l]\ar[d]\ar[r]&\cdots&A_r\ar[d]\ar[l]\ar[r]&\widetilde X_r\ar[d]\\
\widetilde Y_0& B_1\ar[r]\ar[l]&\cdots&B_r\ar[l]\ar[r]&\widetilde Y_r,
\end{tikzcd}\] and such that $f|_{A_j\times I}=f|_{A_j}\times \id_I$ for each $1\leq j\leq r$.
\end{definition}

A decomposition gives rise to a combinatorial relationship to a certain poset.

\begin{definition}
The category of ($r$-fold) \emph{gaps} is the partially ordered set $\gaps{r}$ of nonempty open subsets $A\subseteq [0,r]$ such that \begin{enumerate}
\item the complement of $A$ is a (possibly empty) finite union of closed intervals of positive length;
\item if $i\in \{0,\ldots, r\}$ lies in the closure of $A$, then $i\in A$; and
\item for every $1\leq j\leq r$, $A\cap [j-1,j]\neq \varnothing$.
\end{enumerate}
\end{definition}

\begin{construction}
Given a decompositon $\E$ of $X$, there is a continuous map $\pi_\E:X\to [0,r]$ specified by requiring that \begin{enumerate}
\item $\pi_\E(\widetilde X_j)=\{j\}$, and
\item $\pi_\E|_{A_j\times I}$ is the projection onto $I\cong [j-1,j]$.
\end{enumerate} If $A\subseteq [0,r]$ is a gap, condition (4) of Definition \ref{defi:decomposition} provides the inverse image of $[0,r]\setminus A$ with a canonical subdivisional structure, and we obtain this way a functor
\[
\gamma_\E:\gaps{r}^{op}\to \Emb^\SD.
\] If $f:\E\to \F$ is a map of decompositions, then $f(\pi_\E^{-1}([0,r]\setminus A))\subseteq \pi_\F^{-1}([0,r]\setminus A)$ for every $A\in \gaps{r}$, so we may interpret $f$ as a natural transformation from $\gamma_\E$ to $\gamma_\F$.
\end{construction}

Each $\gamma_\E(A)$ is a (possibly empty) union of some number of components of the form $(A_j\times [a,1])\cup\widetilde X_j\cup(A_{j+1}\times[0,b])$ with $0<a,b<1$ or $U=A_j\times[c,d]$ with $0<c<d<1$. We refer to such a component as a \emph{basic}. We say that the former type of basic is of \emph{component type} and the latter of \emph{bridge type}. The term ``basic'' is borrowed from \cite{AyalaFrancisTanaka:FHSS}, whose ideas heavily influence our approach to Theorem \ref{thm:decomposition}.

Note that $X$ itself typically does not lie in the image of $\gamma_\E$, since $\varnothing\notin\gaps{r}$.

\subsection{Local invariants}\label{section:invariants} In this paper, our main interest in the decomposition theorem stated below will be as a tool to study configuration spaces, but the proof of the theorem will only make use of a few key features of these spaces.

\begin{definition}\label{defi:local invariant}
Let $\E$ be a decomposition of $X$. An $\E$-\emph{local invariant} is a symmetric monoidal functor $F:(\Emb^\SD,\amalg)\to (\Emb^\SD,\times)$ such that the natural map \[\hocolim_{\gaps{r}^{op}}|F(\gamma_\E)|\to |F(X)|\] is a weak equivalence. A map of local invariants (possibly for different decompositions) is a symmetric monoidal natural transformation.
\end{definition}

\begin{remark}
At the cost of greater verbal overhead, it is possible to work with invariants that are only defined locally relative to a given $\E$. All of our results carry over into this more general context.
\end{remark}

We now check that this condition is satisfied in the example of greatest interest to us.

\begin{proposition}
The symmetric monoidal functor $B^\SD$ is an $\E$-local invariant for any convergent decomposition $\E$.
\end{proposition}
\begin{proof}
The claim will follow by two-out-of-three after verifying that each of the numbered arrows in the commuting diagram
\[\begin{tikzcd}\displaystyle\hocolim_{A\in \gaps{r}^{op}}B\left(\pi_\E^{-1}\left([0,r]\setminus \overline{A}\right)\right)\ar{d}[swap]{(2)}\ar{rr}{(1)}&&
B(X)
\ar[no head, d, shift left=.97pt]
\ar[no head, d, shift right=.97pt]
\\
\displaystyle\hocolim_{A\in \gaps{r}^{op}}B\left(|\gamma_\E(A)|\right)\ar{rr}{(3)}&&B(X)\\
\displaystyle\hocolim_{\gaps{r}^{op}}{|B^\SD(\gamma_\E)|}\ar{u}{(4)}\ar[rr]&& {|B^{\SD}(\X)|}, \ar{u}[swap]{(5)}
\end{tikzcd}\]  
is a weak homotopy equivalence, where $\X$ is the subdivisional space determined by $X$ and the subdivisional structure $\subd_X\subseteq \SD(X)$ of $\E$. The second equivalence follows from the fact that $B$ preserves homotopies through injective maps between injective maps, and the third follows from (1) and (2) by two-out-of-three. Theorem~\ref{thm:discrete approximation} gives the fifth and (together with Lemma~\ref{lemma: hocolim and weak equivalence}) the fourth, using our assumption on $\E$ (see the discussion following Definition~\ref{defi: subdivisional configuration}). Thus, it remains to verify the first equivalence. 

Consider the collection $\Uu\coloneqq{}\{B(\pi_\E^{-1}\left([0,r]\setminus \overline{A}\right)): A\in \gaps{r}\}$, which is an open cover of $B(X)$. In fact, $\Uu$ is a \emph{complete} cover in the sense of Definition~\ref{def:complete cover}; to see this, we note that for $S$ finite, \begin{align*}\bigcap_S B(\pi_\E^{-1}\left([0,r]\setminus \overline{A_s}\right))&=B\left(\bigcap_S \pi_\E^{-1}\left([0,r]\setminus \overline{A_s}\right)\right)\\
&=B\left(\pi_\E^{-1}\left(\bigcap_S [0,r]\setminus \overline{A_s}\right)\right)\\
&=B\left(\pi_\E^{-1}\left([0,r]\setminus \bigcup_S\overline{A_s}\right)\right),
\end{align*} and that a finite union of closures of gaps is again the closure of a gap. With this observation, the desired equivalence follows from Theorem~\ref{thm:dugger-isaksen}.
\end{proof}

\begin{remark}
The local invariant $B^\SD$, from its definition as a disjoint union over finite cardinalities, is naturally a \emph{graded} subdivisional space. Furthermore, every map induced by an inclusion $U\subseteq V$ automatically preserves this grading, as do the isomorphisms $B^\SD(U_1\amalg U_2)\cong B^\SD(U_1)\times B^\SD(U_2)$. It is very useful to keep track of this grading in studying configuration spaces (see Theorem~\ref{thm:comparison}, for example), and we will often do so implicitly. 
\end{remark}

\subsection{Flows and bimodules}\label{subsec:subdivisional morse theory} 

The decomposition theorem allows us to study configuration spaces and other local invariants in terms of local information. In order to get a combinatorial handle on the local information, we axiomatize what it means to simplify a local invariant coherently.

\begin{definition}
The \emph{category of abstract flows}, $\Flw_{\mathbb{Z}}$, has objects $(C,\pi)$, where $C$ is a chain complex and $\pi:C\xrightarrow{\sim} C$ is an idempotent quasi-isomorphism. A morphism $(C,\pi)\to (C',\pi')$ is a chain map $f$ such that $\pi'\circ f = \pi'\circ f\circ \pi$. Such a map is called \emph{flow compatible}.

We will use two functors $\Flw_{\mathbb{Z}}\to \Ch_{\mathbb{Z}}$. The \emph{forgetful functor} takes $(C,\pi)$ to $C$ and is the identity on morphisms. The \emph{Morse complex} $I$ takes $(C,\pi)$ to $\pi(C)$ and $f$ to $\pi\circ f$. The subcategory of objects of $\Flw_\mathbb{Z}$ with underlying chain complexes flat in each degree inherits a symmetric monoidal structure for which each of these functors is symmetric monoidal.
\end{definition}

We think of $\pi$ as the limit of a flow on $C$ and the elements of the associated Morse complex as the critical points of that flow.

\begin{remark}
A discrete flow in the sense of \cite{Forman:MTCC} gives rise to an abstract flow; indeed, our definitions of abstract flow and flow compatible map were motivated by the desire to work functorially with discrete Morse data. 
\end{remark}

\begin{definition}
Let $\X$ be a subdivisional space indexed on $\subd$. 
\begin{enumerate}
\item A \emph{subdivisional flow} on $\X$ is the data of the dashed lift in the diagram 
\[\begin{tikzcd}&&&\Flw_\mathbb{Z}\ar[d, description,"\text{forget}"] 
\\
\subd\ar[dashed,urrr]\ar{rr}{\X}&&\Cx^\SD\ar{r}{C}&\Ch_\mathbb{Z}.
\end{tikzcd}\] 
\item A map $f$ of subdivisional spaces equipped with subdivisional flows is \emph{flow compatible} if $C(f)$ lies in the image of the forgetful functor $\Ind(\Flw_\mathbb{Z})\to \Ind(\Ch_\mathbb{Z})$.
\end{enumerate}
\end{definition}

\begin{definition}
Let $\X:\subd\to \Cx^\SD$ be a subdivisional space equipped with a subdivisional flow $g=\{g_p:C(\X(p))\to C(\X(p))\}_{p\in\subd}$. The associated \emph{Morse complex} is the chain complex \[I(C^\SD(\X),g)=\colim_{\subd} I(C(\X(p)),g_p).\]
\end{definition}

The Morse complex is functorial for flow compatible maps and comes equipped with a natural quasi-isomorphism \[C^\SD(\X)\xrightarrow{\sim} I(C^\SD(\X),g),\] since a filtered colimit of quasi-isomorphisms is a quasi-isomorphism. In particular, the Morse complex computes the homology of $|\X|$. Typically, when the choice of subdivisional flow $g$ is clear from context, we abbreviate the Morse complex to $I(C^\SD(\X))$ or simply $I(\X)$.

If $\X$ and $\Y$ are equipped with subdivisional flows, then, since the functor of cellular chains is symmetric monoidal and takes values in levelwise flat chain complexes, the product $\X\times\Y$ inherits a canonical subdivisional flow, which we refer to as the \emph{product flow.} Note that, in this case, there are canonical isomorphisms $I(\X\times \Y)\cong I(\X)\otimes I(\Y)$ satisfying obvious associativity and commutativity relations.

\begin{definition}\label{def:flow on invariant}
Let $\E$ be a decomposition of $X$ and $F$ an $\E$-local invariant. {A \emph{local flow}} on $F$ is the data of a subdivisional flow on $F(\gamma_\E(A))$ for each $A\in \gaps{r}$, subject to the following conditions:
\begin{enumerate}
\item\label{item: subset flow compatibility} for $A\subseteq B$, the induced map $F(B)\to F(A)$ is flow compatible; and \item\label{item: disjoint union flow compatibility} the isomorphism
$F(\gamma_\E(A))\times F(\gamma_\E(B))\cong F(\gamma_\E(A)\amalg \gamma_\E(B))$ and its inverse are each flow compatible, where the lefthand side carries the product flow, whenever $\gamma_\E(A)$ and $\gamma_\E(B)$ are disjoint in $X$.
\end{enumerate}
A map between local invariants equipped with local flows is \emph{flow compatible} if each of its components is so.\end{definition}

A local flow in particular determines the dashed lift in the diagram 

\[\begin{tikzcd}
&&&\Ind(\Flw_\mathbb{Z})\ar{d}{\colim}\\
\gaps{r}^{op}\ar{r}\ar[dashed,urrr]&\Emb^\SD\ar{r}{F}&\Emb^\SD\ar{r}{C^\SD}&\Ch_\mathbb{Z}.
\end{tikzcd}\]

Since the homotopy colimit of the bottom composite computes the homology of $|F(X)|$ by Corollary~\ref{cor:hocolim} and Definition~\ref{defi:local invariant}, we can hope to understand this homology by means of the Morse complexes associated to these subdivisional flows. We now introduce a further condition guaranteeing that the Morse theory is sufficiently rigid.

\begin{definition}\label{def:isotopy invariance}
Let $\E$ be a decomposition of $X$ and $F$ an $\E$-local invariant. A local flow on $F$ is \emph{isotopy invariant} if every inclusion between two basics of bridge type or two basics of component type induces an isomorphism on Morse complexes.
\end{definition}

Note that a local flow does not assign a subdivisional flow to $F(X)$. We also refrain from assigning flows to the components or the bridges; instead, we set $I(F(A_j\times I)):=\colim I(F(U))$, where the colimit is taken over all basics contained in $A_j\times I$, and similarly for $I(F(X_j))$. Isotopy invariance guarantees each of these ``Morse complexes'' is canonically isomorphic to the Morse complex of any corresponding basic.

In a more homotopical context, it would be sensible to require a weaker condition. The strictness of isotopy invariance is motivated by the computational nature of our goals, and it has the following important consequence, which is drawn from the theory of factorization algebras---see \cite{Ginot:NFAFHA} and the references therein.

\begin{construction}\label{construction:algebra and module structure}
Let $\E$ be a decomposition of $X$ and $F$ an $\E$-local invariant equipped with an isotopy invariant local flow. For each $1\leq j \leq r$, the Morse complex $I(F(A_j\times I))$ of the bridge carries a natural associative algebra structure, for which the Morse complexes $I(F(X_{j-1}))$ and $I(F(X_j))$ of the relevant components are right and left modules, respectively.

We indicate how the algebra structure arises (the module structures are similar). Set $R_j\coloneqq{}I(F(A_j\times I))$; then the unit map $\eta_j:\mathbb{Z}\to R_j$ is induced by the inclusion $\varnothing\to A_j\times I$. As for the multiplication, we take any two disjoint subbasics $U_1$ and $U_2$ of bridge type contained in $A_j\times I$ and define $\mu_j$ to be the composite
\begin{align*}
R_j\otimes R_j&\xrightarrow{\simeq} I(F(U_1))\otimes I(F(U_2))&\text{(\ref{def:isotopy invariance})}\\
&\xrightarrow{\simeq}I(F(U_1)\times F(U_2))\\
&\xrightarrow{\simeq}I(F(U_1\amalg U_2)) &{\text{(\ref{def:flow on invariant}(\ref{item: disjoint union flow compatibility}))}}\\
&\xrightarrow{\hphantom{\sim}} I(F(A_j\times I))= R_j,&{\text{(\ref{def:flow on invariant}(\ref{item: subset flow compatibility}))}}
\end{align*}where in the second line we have used that the Morse complex of a product subdivisional flow is the tensor product of the Morse complexes.

Any two choices for $U_1$ and $U_2$ may be connected by a zig-zag of inclusions of disjoint pairs of basics of bridge type, so, by naturality of the maps in question, $\mu_j$ is independent of this choice. Tracing through the construction shows that $\eta_j$ is a unit for $\mu_j$. For associativity, we consider a configuration of five bridge type basics with containments as in 
Figure~\ref{figure:associativity}. We use $U_1$ and $U_{23}$ to define the outer multiplication and $U_2$ and $U_3$ for the inner multiplication in the expression $\mu_j\circ (\id\otimes \mu_j)$, and we use $U_{12}$ and $U_3$ to define the outer multiplication and $U_1$ and $U_2$ to define the inner multiplication in the expression $\mu_j\circ(\mu_j\otimes \id)$. Then both expressions are given in terms of maps $I(F(U_1))\otimes I(F(U_2))\otimes I(F(U_3))\to R_j$, which coincide because of the associativity of the structure morphisms for local invariants.
\end{construction}

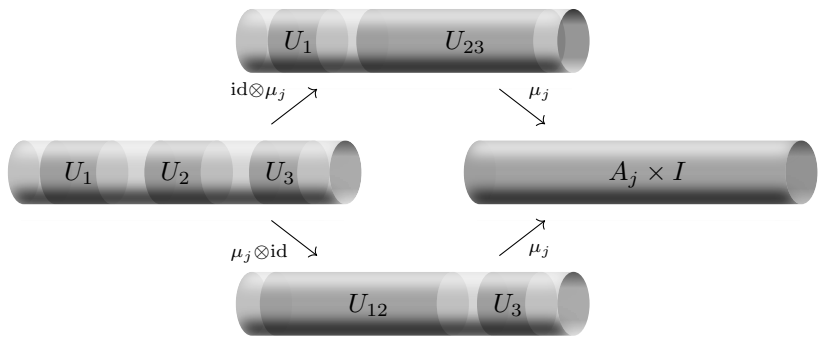
\begin{figure}[ht]
\[
\begin{tikzcd}[column sep=-25mm, row sep=2mm]
&{\vcenter{\hbox{\def\svgscale{0.75}
\begingroup%
  \makeatletter%
  \ifx\svgwidth\undefined%
    \setlength{\unitlength}{192bp}%
    \ifx\svgscale\undefined%
      \relax%
    \else%
      \setlength{\unitlength}{\unitlength * \real{\svgscale}}%
    \fi%
  \else%
    \setlength{\unitlength}{\svgwidth}%
  \fi%
  \global\let\svgwidth\undefined%
  \global\let\svgscale\undefined%
  \makeatother%
  \begin{picture}(1,0.25306194)%
    \put(0,0){\includegraphics[width=\unitlength,page=1]{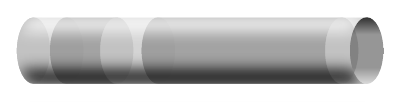}}%
    \put(0.16666667,0.10569773){\color[rgb]{0,0,0}\makebox(0,0)[lb]{\smash{$U_1$}}}%
    \put(0.58333333,0.10569773){\color[rgb]{0,0,0}\makebox(0,0)[lb]{\smash{$U_{23}$}}}%
  \end{picture}%
\endgroup%
}}}
\ar[dr,"\mu_j",start anchor={[yshift=4pt, xshift=-4pt]},end anchor = {[yshift = -4pt, xshift = 4pt]}]\\
{\vcenter{\hbox{\def\svgscale{0.75}
\begingroup%
  \makeatletter%
  \ifx\svgwidth\undefined%
    \setlength{\unitlength}{192bp}%
    \ifx\svgscale\undefined%
      \relax%
    \else%
      \setlength{\unitlength}{\unitlength * \real{\svgscale}}%
    \fi%
  \else%
    \setlength{\unitlength}{\svgwidth}%
  \fi%
  \global\let\svgwidth\undefined%
  \global\let\svgscale\undefined%
  \makeatother%
  \begin{picture}(1,0.25306194)%
    \put(0,0){\includegraphics[width=\unitlength,page=1]{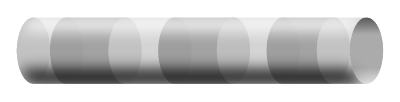}}%
    \put(0.1875,0.10569773){\color[rgb]{0,0,0}\makebox(0,0)[lb]{\smash{$U_1$}}}%
    \put(0.4375,0.10569773){\color[rgb]{0,0,0}\makebox(0,0)[lb]{\smash{$U_2$}}}%
    \put(0.70833333,0.10569773){\color[rgb]{0,0,0}\makebox(0,0)[lb]{\smash{$U_3$}}}%
  \end{picture}%
\endgroup%
}}}
\ar[ur,"\id\otimes\mu_j",start anchor={[yshift=-4pt, xshift=-4pt]},end anchor = {[yshift = 4pt, xshift = 4pt]}]
\ar[dr,swap,"\mu_j\otimes\id",start anchor={[yshift=4pt, xshift=-4pt]},end anchor = {[yshift = -4pt, xshift = 4pt]}]
& &
{\vcenter{\hbox{\def\svgscale{0.75}
\begingroup%
  \makeatletter%
  \ifx\svgwidth\undefined%
    \setlength{\unitlength}{192bp}%
    \ifx\svgscale\undefined%
      \relax%
    \else%
      \setlength{\unitlength}{\unitlength * \real{\svgscale}}%
    \fi%
  \else%
    \setlength{\unitlength}{\svgwidth}%
  \fi%
  \global\let\svgwidth\undefined%
  \global\let\svgscale\undefined%
  \makeatother%
  \begin{picture}(1,0.25306194)%
    \put(0,0){\includegraphics[width=\unitlength,page=1]{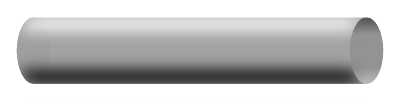}}%
    \put(0.41666667,0.10569773){\color[rgb]{0,0,0}\makebox(0,0)[lb]{\smash{$A_j\times I$}}}%
  \end{picture}%
\endgroup%
}}}\\
&{\vcenter{\hbox{\def\svgscale{0.75}
\begingroup%
  \makeatletter%
  \ifx\svgwidth\undefined%
    \setlength{\unitlength}{192bp}%
    \ifx\svgscale\undefined%
      \relax%
    \else%
      \setlength{\unitlength}{\unitlength * \real{\svgscale}}%
    \fi%
  \else%
    \setlength{\unitlength}{\svgwidth}%
  \fi%
  \global\let\svgwidth\undefined%
  \global\let\svgscale\undefined%
  \makeatother%
  \begin{picture}(1,0.25306194)%
    \put(0,0){\includegraphics[width=\unitlength,page=1]{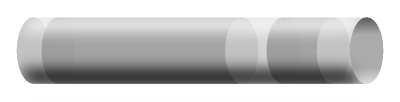}}%
    \put(0.33333333,0.10569773){\color[rgb]{0,0,0}\makebox(0,0)[lb]{\smash{$U_{12}$}}}%
    \put(0.70833333,0.10569773){\color[rgb]{0,0,0}\makebox(0,0)[lb]{\smash{$U_3$}}}%
  \end{picture}%
\endgroup%
}}}
\ar[ur,swap,"\mu_j",start anchor={[yshift=-4pt, xshift=-4pt]},end anchor = {[yshift = 4pt, xshift = 4pt]}]
\end{tikzcd}
\]
\caption{Basics showing the associativity of the product.}\label{figure:associativity}
\end{figure}

It will be convenient to have terminology for this emergent algebraic structure.
\begin{definition}
An \emph{$r$-fold bimodule} is 
\begin{enumerate}
\item a collection $(R_1,\ldots, R_r, M_0,\ldots, M_r)$ of chain complexes,
\item the structure of a differential graded unital associative algebra on $R_j$ for each $1\leq j\leq r$, and
\item the structure of an $(R_j, R_{j+1})$-bimodule on $M_j$ for each $0\leq j\leq r$, where $R_0=R_{r+1}=\mathbb{Z}$ by convention.
\end{enumerate}
A map from the $r$-fold bimodule $(R_1,\ldots, R_r, M_0,\ldots, M_r)$ to the $r$-fold bimodule $(R'_1,\ldots, R'_r,M'_0,\ldots, M'_r)$ consists of a collection of maps of algebras $R_j\to R'_j$ together with a map of $(R_j,R_{j+1})$-bimodules from $M_j$ to $M'_j$ for each $j$, where $M_j'$ carries the bimodule structure induced by restriction along the maps $R_j \to R_j'$ and $R_{j+1}\to R_{j+1}'$.
\end{definition}

It is clear from Construction~\ref{construction:algebra and module structure} and the definition of a local flow on a local invariant that a flow compatible map between local invariants intertwines the various algebra and module structures.

We summarize the discussion so far in categorical terms. There is a category $\isotopyinvariant{r}$ whose objects are triples $(X,\E,F)$ consisting of a complex $X$, an $r$-fold decomposition $\E$, and an $\E$-local invariant $F$ equipped with an isotopy invariant local flow (suppressed in the notation), and whose morphisms are pairs of a map of decompositions and a flow compatible map of local invariants. There is a second category $\bimod{r}$ whose objects are $r$-fold bimodules and whose morphisms are maps of such. What we have constructed so far is a canonical lift of the Morse complex to a functor \[I:\isotopyinvariant{r}\to \bimod{r}.\]

\begin{remark}\label{rmk:pointing}
In fact, the $r$-fold bimodule arising from an isotopy invariant local flow is always a \emph{pointed} $r$-fold bimodule, which is to say that each $M_j$ is equipped with a distinguished map $\mathbb{Z}\to M_j$ arising from the inclusion $\varnothing\to X_j$. Except in defining the natural transformation of Construction \ref{construction:morse and bar}, this extra structure will play little role in what follows.
\end{remark}

\subsection{The decomposition theorem}\label{subsec:decomposition theorem}
Our goal is to use the algebraic structures of Construction~\ref{construction:algebra and module structure} to express the global value of a local invariant with an isotopy invariant local flow in terms of its values on the pieces of the decomposition. In order to make a precise statement, we first need to spell out how the various pieces of an $r$-fold bimodule may be assembled to give a corresponding ``global value.''

\begin{definition}
Let $(R_1,\ldots, R_r, M_0,\ldots, M_r)$ be an $r$-fold bimodule.
\begin{enumerate}
\item The \emph{simplicial bar construction} on these data is the $r$-fold simplicial chain complex given in degree $(n_1,\ldots, n_r)$ by \[\mathrm{Bar}_\Delta(M_0, R_1,\ldots, R_r, M_r)_{(n_1,\ldots, n_r)}=M_0\otimes R_1^{\otimes n_1}\otimes\cdots\otimes R_r^{\otimes n_r} \otimes M_r.\] The face maps are defined by the respective module action and algebra multiplication maps, and the degeneracies are defined by the respective units.
\item The \emph{bar construction} or \emph{bar complex} on these data is the total chain complex of the multicomplex obtained from the simplicial bar construction by taking the respective alternating sums of the face maps in each simplicial direction.
\end{enumerate}
\end{definition}

The bar complex of the $r$-fold bimodule is a functor in a straightforward manner and computes the homology of the derived tensor product \[M_0\otimes^\mathbb{L}_{R_1}\cdots\otimes^\mathbb{L}_{R_r}M_r.\] Indeed, depending on one's point of view, this prescription may even be taken as the definition of the derived tensor product. For details on these matters, the reader may consult~\cite{Smith:HAEMSS}, for example.

We are now equipped to state our main result.

\begin{theorem}[Decomposition theorem]\label{thm:decomposition} There is a natural weak equivalence connecting the two composites in the diagram \[\begin{tikzcd}
\isotopyinvariant{r}\ar[swap]{d}{I}\ar{rrr}{(X,\E,F)\mapsto F(X)}&&&\Emb^\SD\ar{d}{C^\SD}\\
\bimod{r}\ar{rrr}{\mathrm{Bar}}&&&\Ch_\mathbb{Z}.
\end{tikzcd}\]
\end{theorem}

In particular, there is a weak equivalence \[C^\SD(F(X))\simeq I(F(X_0))\bigotimes_{I(F(A_1\times I))}^\mathbb{L}\cdots\bigotimes^\mathbb{L}_{I(F(A_r\times I))} I(F(X_r)).\]

We turn now to the proof of Theorem~\ref{thm:decomposition}. For the sake of brevity, when considering the $r$-fold bimodule arising from an $\E$-local invariant $F$, we use the somewhat abusive notation $\mathrm{Bar}(I(\E))$ for the corresponding bar complex. Since $\mathrm{Bar}(I(\E))$ arises from the multi-simplicial object $\mathrm{Bar}_\Delta(I(\E))$, and since $C^\SD(F(X))$ may be recovered as a homotopy colimit over $\gaps{r}^{op}$, our strategy in relating these two objects will be to relate the categories $\Delta^r$ and $\gaps{r}$.

\begin{definition}
Let $A\subseteq [0,r]$ be a gap. The $j$th \emph{trace} of $A$ is the ordered set \[\tau_j(A)=\pi_0\left(A\cap [j-1,j]\right),\] with the ordering induced by the standard orientation of $\mathbb{R}$.
\end{definition}

It follows from the definitions that the set $\tau_j(U)$ is always non-empty, so the various traces extend to a functor \[\tau:\gaps{r}\to \Delta^r.\]

Using the trace, we may relate the Morse complex to the bar construction. 

\begin{construction}\label{construction:morse and bar}
We define a natural transformation \[\psi:I(F(\gamma_\E))\to\mathrm{Bar}_\Delta(I(\E))\circ\tau^{op}\] of functors from $\gaps{r}^{op}$ to chain complexes. For a gap $A$, we have 
\begin{align*}
\mathrm{Bar}_\Delta(I(\E))(\tau^{op}(A))&\cong 
I(F(X_0))\otimes I(F(A_1\times I))^{\otimes|\tau_1(A)|-1}\otimes\cdots
\\&\quad{} \cdots\otimes I(F(A_r\times I))^{\otimes|\tau_r(A)|-1}\otimes I(F(X_r))\end{align*} and we define the component $\psi_A$ by expressing $\gamma_\E(A)$ as a disjoint union of basics and tensoring together the maps induced by the inclusions of each basic into the corresponding component or bridge, or of the empty set into the relevant component---see Remark \ref{rmk:pointing}. Naturality follows from flow compatibility of the structure maps involved, together with the fact that these Morse complexes arise from product flows, both of which are guaranteed by Definition~\ref{def:flow on invariant}.
\end{construction}

As a matter of terminology, we say that an $r$-fold gap $A\in \gaps{r}$ is \emph{separated} if $A\cap\{0,\ldots,r\}=\varnothing$. Write $\sepgaps{r}\subseteq \gaps{r}$ for the full subcategory of separated gaps.

\begin{lemma}\label{lem:pullback}
Let $\E$ be a decomposition of $X$ and $F$ an $\E$-local invariant with an isotopy-invariant local flow. The canonical map \[\hocolim_{\gaps{r}^{op}}I(F(\gamma_\E))\xrightarrow{\sim} \hocolim_{\gaps{r}^{op}}\mathrm{Bar}_\Delta(I(\E))\circ\tau^{op}\] induced by $\psi$ is a quasi-isomorphism.
\end{lemma}
\begin{proof}
A gap $A$ is separated if and only if $\gamma_\E(A)$ intersects each $\widetilde X_j$ non-vacuously, so, by isotopy invariance, $\psi_A$ is a quasi-isomorphism for separated $A$. Thus, by Lemma~\ref{lemma: hocolim and weak equivalence} and Proposition~\ref{prop:finality criterion}, it suffices to note that the inclusion $\sepgaps{r}\subseteq\gaps{r}$ is homotopy initial (so that the inclusion of opposite categories is homotopy final). Indeed, all of the overcategories in question are cofiltered and hence contractible.
\end{proof}

Theorem~\ref{thm:decomposition} relies on the following fact about the functor $\tau$.

\begin{lemma}\label{lem:fiber final}
For any object $S\in\Delta^{r}$, the inclusion $\iota:\tau^{-1}(S)\to(S\!\downarrow\!\tau)$ is homotopy initial. 
\end{lemma}

In one form or another, this fact is certainly well-known to experts. In the name of a self-contained narrative, we nevertheless include a proof, which is deferred to \S\ref{section:fiber final proof} below. For now, we draw the following consequence (see Appendix~\ref{section:hocolim appendix} for notation).

\begin{corollary}\label{cor:counit}
Let $V:(\Delta\!^{op})^{r}\to \Ch_\mathbb{Z}$ be a multi-simplicial chain complex. There is a natural weak equivalence $\hoLan_{\tau^{op}}(V\circ\tau^{op})\simeq V$.
\end{corollary}
\begin{proof}
We have that
\begin{align*}\hoLan_{\tau^{op}}(V\circ\tau^{op})(S)&=\hocolim_{(\tau^{op}\downarrow S)}\left(V\circ \tau^{op}\circ \forgettodomain\right)
&&\text{(\ref{def: holan})}\\
&\simeq\hocolim_{(\tau^{op})^{-1}(S)}
\left(V\circ \tau\circ \forgettodomain\circ \iota\right)
&&\text{(\ref{lem:fiber final}, \ref{prop:finality criterion})}\\
&=
\hocolim_{(\tau^{op})^{-1}(S)}\left(V\circ \underline{S}\right)
\\
&\simeq V(S)&&\text{(\ref{cor:contractible constant hocolim})}\end{align*} where in the last step we have used that the category $\tau^{-1}(S)$ is contractible. To see why this is so, we note that $(S\!\downarrow\!\tau)$ is contractible, having a final object (since $\gaps{r}$ has the final object $[0,r]$), and invoke Corollary~\ref{cor:final functor nerve} and Lemma~\ref{lem:fiber final} a second time.
\end{proof}

\begin{proof}[Proof of Theorem~\ref{thm:decomposition}] We have the following column of quasi-isomorphisms:
\begin{align*}
C^\SD(F(X))&\simeq \hocolim_{\gaps{r}^{op}}C^\SD(F(\gamma_\E))&&\text{(\ref{cor:hocolim})}\\
&\simeq \hocolim_{\gaps{r}^{op}}I(F(\gamma_\E))&&\text{(\ref{lemma: hocolim and weak equivalence})}\\
&\simeq \hocolim_{\gaps{r}^{op}}\tau^*\mathrm{Bar}_\Delta(I(\E))&&\text{(\ref{lem:pullback},\,\ref{lemma: hocolim and weak equivalence})}\\
&\cong \hoLan_*(\mathrm{Bar}_\Delta(I(\E))\circ \tau^{op})(*)&&\text{(\ref{example: Left Kan extensions recover homotopy colimits})}\\
&\simeq \hoLan_{*}\hoLan_\tau (\mathrm{Bar}_\Delta(I(\E))\circ \tau^{op})(*)&&\\
&\cong \hocolim_{(\Delta\!^{op})^{r}}\hoLan_\tau (\mathrm{Bar}_\Delta(I(\E))\circ\tau^{op})&&\text{(\ref{example: Left Kan extensions recover homotopy colimits})}\\
&\simeq \hocolim_{(\Delta\!^{op})^{r}}\mathrm{Bar}_\Delta(I(\E))&&\text{(\ref{cor:counit},\,\ref{lemma: hocolim and weak equivalence})}\\
&\simeq \mathrm{Bar}(I(\E))&&\text{(\ref{prop:dold-kan hocolim})}.
\end{align*} The unmarked quasi-isomorphism is formal, following from the fact that left Kan extensions compose. Naturality follows from flow compatibility and inspection of Construction~\ref{construction:morse and bar}.
\end{proof}

\subsection{Proof of Lemma~\ref{lem:fiber final}}\label{section:fiber final proof}

The fundamental observation in the proof is the following. Let $\sepgaps{1,k}\subseteq \sepgaps{1}$ denote the (non-full) subcategory with objects the separated gaps with exactly $k$ components and morphisms the $\pi_0$-bijective inclusions.

\begin{lemma}\label{lem:contractible intervals}
For every $k>0$, the category $\sepgaps{1,k}$ is contractible.
\end{lemma}
\begin{proof}
We define a functor $\chi:\sepgaps{1}\to \Top$ by letting $\chi(A)\subseteq B_k(A)$ be the subspace of configurations that intersect each connected component of $A$ non-trivially. 
The proof will be complete upon establishing the chain of weak homotopy equivalences \[|N \sepgaps{1,k}|\simeq \hocolim_{\sepgaps{1,k}}\underline\pt\simeq\hocolim_{\sepgaps{1,k}}\chi\simeq\pt,\] where $\underline\pt$ is the constant functor with value a singleton. The first equivalence is immediate from Definition~\ref{def:homotopy colimit}, and the second follows from Lemma~\ref{lemma: hocolim and weak equivalence} and the fact that $\chi(A)$ is homeomorphic to the product of the connected components of ${A}$ and hence contractible. To establish the third equivalence, we note that the collection $\{\chi(A):A\in\sepgaps{1,k})\}$ of open subsets of {$B_k((0,1))$} is a basis for its topology (and thus a complete cover).
Since this configuration space is contractible, the desired equivalence now follows from Theorem~\ref{thm:dugger-isaksen}.
\end{proof}

Now, a separated $r$-fold gap is nothing more or less than an $r$-tuple of separated $1$-fold gaps. In other words, there is a commuting diagram of functors \[
\begin{tikzcd}\sepgaps{1}^r\ar{d}&\sepgaps{r}\ar{l}[swap]{\sim}\ar{r}&\gaps{r}\ar{d}{\tau}\\
\gaps{1}^r\ar{rr}{\tau^r}&&\Delta^r.
\end{tikzcd}\] Writing $\widetilde\tau:\sepgaps{r}\to \Delta^r$ for either composite and fixing an object $S=(S_1,\ldots, S_r)\in \Delta^r$, we have another commuting diagram of functors of the form
\[\begin{tikzcd}
\tau^{-1}(S)\ar{r}{\iota}&(S\!\downarrow\!\tau)\\
\widetilde\tau^{-1}(S)\ar{d}[swap]{\wr}\ar{r}{\widetilde \iota}\ar{u}&(S\!\downarrow\!\widetilde \tau)\ar{u}\ar{d}{\wr}\\
\prod_{j=1}^r\widetilde \tau^{-1}(S_j)\ar{r}{\widetilde \iota^r}&\prod_{j=1}^r(S_j\!\downarrow\!\widetilde\tau).
\end{tikzcd}\] The strategy will be to understand $\iota$ by understanding each of the other arrows in the diagram (a direct comparison is possible but somewhat more involved), beginning with the upper vertical arrows.

\begin{lemma}\label{lem:comp reduction}
The inclusions of $\widetilde\tau^{-1}(S)$ and $(S\!\downarrow\!\widetilde\tau)$ into $\tau^{-1}(S)$ and $(S\!\downarrow\!\tau)$, respectively, are each homotopy initial.
\end{lemma}
\begin{proof}
We give the proof for the former inclusion only, the latter differing only in requiring more notation. Fixing $A\in\tau^{-1}(S)$, we must check the contractibility of the category of $r$-fold gaps $B$ such that
\begin{enumerate}
\item \label{item: A contains B}$B\subseteq A$,
\item \label{item: tau of B to A is id}$\tau(B\subseteq A)=\id_S$, and
\item \label{item: B hits components}$B$ is separated.
\end{enumerate}
First, we note that this category is non-empty; indeed, we may obtain such a $B$ from $A$ by removing a sufficiently small neighborhood of each $j\in\{0,\ldots, r\}$ from $A$. Moreover, any $B$ satisfying these three conditions is contained one of this form. Since this subcollection is clearly filtered, the claim follows by Example~\ref{example: cofiltered}.
\end{proof}

Next, we consider the middle horizontal arrow.

\begin{lemma}\label{lem:comp finality}
The functor $\widetilde \iota$ is homotopy initial.
\end{lemma}
\begin{proof}
The property of being homotopy initial is preserved by products and equivalences of categories, so it suffices to consider the case $r=1$. We establish some notation.
\begin{enumerate}
\item An object of $(S\!\downarrow\!\widetilde\tau)$ is a pair $(A,f)$, where $f:S\to T$ is a map of ordered sets and $A$ is a union of open subintervals of $(0,1)$ whose components have disjoint closures and such that $\pi_0(A)=T$ as ordered sets.
\item A morphism is an inclusion making the evident triangle in $\Delta$ commute. 
\item The functor $\widetilde\iota$ is defined by sending $A\in \tau^{-1}(S)$ to the pair $(A,\id_S)$. 
\end{enumerate}
We wish to prove the contractibility, for each $(A,f)$, of the category $(\widetilde\iota\!\downarrow\!(A,f))$, which is nothing other than the category of gaps $B\subseteq A$ with $\pi_0(B)=S$ and $\pi_0(B\subseteq A)=f$. By inspection, we have the isomorphism \[(\widetilde\iota\!\downarrow\!(A,f))\cong\prod_{t\in \im(f)}\sepgaps{1,|f^{-1}(t)|},\] so contractibility follows from Lemma \ref{lem:contractible intervals}.
\end{proof}

\begin{proof}[Proof of Lemma~\ref{lem:fiber final}]
It is a fact that homotopy initial functors satisfy a partial two-out-of-three property; that is, if $T_1$ is homotopy initial, then $T_2\circ T_1$ is homotopy initial if and only if $T_2$ is so (see \cite[Prop.~4.1.1.3(2)]{Lurie:HTT}, for example). Applying this fact to the composite $\widetilde\tau^{-1}(S)\to \tau^{-1}(S)\to (S\!\downarrow\!\tau)$ and invoking Lemma~\ref{lem:comp reduction} reduces the lemma to verifying that the composite is homotopy initial. This composite coincides with the composite $\widetilde\tau^{-1}(S)\to (S\!\downarrow\!\widetilde\tau)\to (S\!\downarrow\!\tau)$, so this claim follows from Lemmas~\ref{lem:comp reduction} and~\ref{lem:comp finality}, since homotopy initial functors compose.
\end{proof}

\section{\texorpdfstring{Application to graphs}{Chain models for graph braid groups}}\label{section:graph braid groups}

We use the theory developed above to study the configuration space of a graph. We establish conventions, introduce the \'{S}wi\k{a}tkowski complex, and use the technology developed in the previous section to prove Theorem~\ref{thm:comparison}, which asserts that this complex computes the homology of the configuration spaces of a graph functorially.

\subsection{Conventions on graphs}\label{section:conventions}
A \emph{graph} is a finite $1$-dimensional CW complex $\graf$. Its $0$-cells and $1$-cells are its \emph{vertices} and \emph{edges} $V(\graf)$ and $E(\graf)$, or simply $V$ and $E$. The vertices of an edge $V(e)$ are the vertices contained in the closure of that edge in $\graf$. We write $E(v)$ for the set of edges incident to the vertex $v$. A \emph{half-edge} is an end of an edge. The set of half-edges of $\graf$ is $H(\graf)$ or simply $H$. We write $H(v)$ for the set of half-edges incident to $v$ and $H(e)$ for the set of half-edges contained in $e$. For $h$ in $H$, we write $v(h)$ and $e(h)$ for the corresponding vertex and edge.

Any sufficiently small neighborhood of a vertex $v$ is homeomorphic to a cone on finitely many points; this finite number is the \emph{valence} of $v$, denoted $d(v)$. The vertex $v$ is \emph{isolated} if $d(v)=0$. An edge with a $1$-valent vertex is a \emph{tail}. A \emph{self-loop} at a vertex is an edge whose entire boundary is attached at that vertex. A graph is \emph{simple} if it has no self-loops and no pair of edges with the same vertices.

\begin{example}
The cone on $\{1,\ldots, n\}$ is a graph $\stargraph{n}$ with $n+1$ vertices. These graphs are called \emph{star graphs} and the cone point the \emph{star vertex}. 

The interval $\intervalgraph$ is a graph with two vertices $0$ and $1$ and one edge between them. It is isomorphic to $\stargraph{1}$ and homeomorphic to $\stargraph{2}$ but it will be convenient to have alternate notation for it.
\end{example}
\begin{figure}[ht]
\centering
\begin{tikzpicture}
\begin{scope}[xshift=-2cm]
\fill[black] (0,-.5) circle (2.5pt);
\fill[black] (0,-1) circle (2.5pt);
\draw(0,-1) -- (0,-.5);
\draw(0,-1.2) node[below]{$\stargraph{1}$};
\end{scope}
\fill[black] (0,-.5) circle (2.5pt);
\fill[black] (0,-1) circle (2.5pt);
\fill[black] (0,0) circle (2.5pt);
\draw(0,-1) -- (0,0);
\draw(0,-1.2) node[below]{$\stargraph{2}$};
\begin{scope}[xshift=2.5cm]
\begin{scope}[yshift=-.5cm]
\fill[black] (0,0) circle (2.5pt);
\fill[black] (0,-.5) circle (2.5pt);
\fill[black] (.433,.25) circle (2.5pt);
\fill[black] (-.433,.25) circle (2.5pt);
\draw(0,0) -- (0,-.5);
\draw(0,0) --(.433,.25);
\draw(0,0) --(-.433,.25);
\end{scope}
\draw(0,-1.2) node[below]{$\stargraph{3}$};
\end{scope}
\begin{scope}[xshift=5.5cm]
\begin{scope}[yshift=-.5cm]
\fill[black] (0,0) circle (2.5pt);
\fill[black] (.358,.358) circle (2.5pt);
\fill[black] (.358,-.358) circle (2.5pt);
\fill[black] (-.358,.358) circle (2.5pt);
\fill[black] (-.358,-.358) circle (2.5pt);
\draw(-.358,-.358) -- (.358,.358);
\draw(-.358,.358) -- (.358,-.358);
\end{scope}
\draw(0,-1.2) node[below]{$\stargraph{4}$};
\end{scope}
\end{tikzpicture}
\caption{Star graphs}
\end{figure}
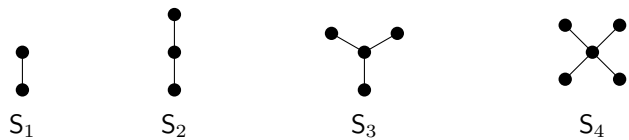

\begin{definition}
Let $f:\graf_1\to \graf_2$ be a continuous map between graphs. We say that $f$ is a \emph{graph morphism} if 
\begin{enumerate}
\item the inverse image $f^{-1}(V(\graf_2))$ is contained in $V(\graf_1)$ and 
\item the map $f$ is injective.
\end{enumerate}
We call a graph morphism a \emph{smoothing} if it is a homeomorphism and a \emph{graph embedding} if it preserves vertices. A graph morphism can be factored into a graph embedding followed by a smoothing. The composite of graph morphisms is a graph morphism, and we obtain in this way a category $\Gph$. Although the objects of $\Gph$ are simply finite 1-dimensional CW complexes, not all morphisms are cellular. A \emph{subgraph} is the image of a graph embedding. A graph morphism $f:\graf_1\to \graf_2$ induces a map $E(f):E(\graf_1)\to E(\graf_2)$, a partially defined map $V(f):f^{-1}(V(\graf_2))\to V(\graf_2)$, and a map $H(v)(f):H(v)\to H(f(v))$ for each $v\in f^{-1}(V(\graf_2))$.
\end{definition}
\begin{figure}[ht]
\centering
\begin{tikzpicture}
\fill[black] (0,0) circle (2.5pt);
\fill[black] (1.5,0) circle (2.5pt);
\fill[black] (.75,0) circle (2.5pt);
\draw(0,0) -- (1.5,0);
\draw(-.5,0) circle (.5cm);
\draw(2,0) circle (.5cm);
\begin{scope}[xshift=5cm]
\fill[black] (0,0) circle (2.5pt);
\fill[black] (1.5,0) circle (2.5pt);
\draw(0,0) -- (1.5,0);
\draw(-.5,0) circle (.5cm);
\draw(2,0) circle (.5cm);
\end{scope}
\end{tikzpicture}
\caption{There is a graph morphism (in fact a smoothing) from left to right but not from right to left.}
\end{figure}
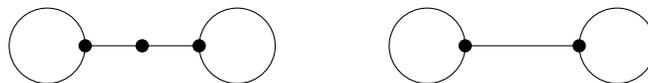

Since graph morphisms are injective, they induce maps at the level of configuration spaces. Thus, it is natural to view $H_*(B(-))$ as a functor from the category $\Gph$ to bigraded Abelian groups (with a weight grading for cardinality).

\subsection{The {Swiatkowski} complex}\label{section:intrinsic complex} We now introduce our main tool in the study of $H_*(B(\graf))$.

\begin{construction}[\'{S}wi\k{a}tkowski complex]
Let $\graf$ be a graph. For each vertex $v\in V$, we set $\localstates{v}=\mathbb{Z}\langle \varnothing,v, h\in H(v)\rangle$ and regard this Abelian group as bigraded with $|\varnothing|=(0,0)$, $|v|=(0,1)$, and $|h|=(1,1)$.

The \emph{\'{S}wi\k{a}tkowski complex} of $\graf$ is the differential bigraded $\mathbb{Z}[E]$-module \[\intrinsic{\graf}=\mathbb{Z}[E]\otimes\bigotimes_{v\in V}\localstates{v},\] where $|e|= (0,1)$ for $e\in E$, and differential determined by setting \[\partial(h)=e(h)-v(h).\]Since $\partial$ is $\mathbb{Z}[E]$-linear, the module structure descends to homology. 

A graph morphism $f:\graf_1\to\graf_2$ determines a map $\intrinsic{f}:\intrinsic{\graf_1}\to \intrinsic{\graf_2}$. This takes edges to their images under $f$. If $f(v)$ is a vertex of $\graf_2$, then the induced map takes $\localstates{v}$ to $\localstates{f(v)}$ using $f$. If $f(v)$ is in the edge $e$ in $\graf_2$, the map factors through $\localstates{v} \to \mathbb{Z}[e]$ where $\varnothing$ goes to 1, $v$ to $e$, and $h\in H(v)$ to 0. By inspection, $\intrinsic{f}$ respects the bigrading, differential, and module structures. 
\end{construction}

\begin{remark}The generators of $\intrinsic{\graf}$ describe ``states'' in the configuration spaces of $\graf$. The module $\localstates{v}$ records the local states allowed at the vertex $v$, with $\varnothing$ corresponding to the absence of a particle at $v$, the element $v$ to a stationary particle at $v$, and the element $h\in H(v)$ to a path in which a particle moves infinitesimally along the edge containing $h$. A general state is obtained by prescribing a local state at each vertex and a number of particles on each edge, and the differential is the cellular differential taking a path to its endpoints.
\end{remark}

We view $S$ as a functor from $\Gph$ to the category of bigraded chain complexes and the action of a (weight-graded) ring. A morphism is a weight-graded morphism of rings and a compatible morphism of differential bigraded modules. We denote the degree $i$ and weight $k$ component of $S(\graf)$ by $S_i(\graf)_k$. Our main result concerning the \'{S}wi\k{a}tkowski complex is the following:

\begin{theorem}[Comparison theorem]\label{thm:comparison}
There is an isomorphism \[H_*(B(\graf))\cong H_*(\intrinsic{\graf})\] of functors from $\Gph$ to bigraded Abelian groups.
\end{theorem}

\begin{remark}
The weight $k$ subcomplex $\intrinsic{\graf}_k$ is isomorphic to the cellular chains of the cubical complex exhibited in \cite{Swiatkowski:EHDCSG}, if every vertex is at least trivalent. 
This implies Theorem \ref{thm:comparison} at the level of objects in this special case. 
The full proof (in particular, functoriality) constitutes the content of this section.

Here we present the $\mathbb{Z}[E]$-module structure on the \'{S}wi\k{a}tkowski complex  algebraically. In subsequent work~\cite{AnDrummond-ColeKnudsen:ESHGBG}, we show that this structure arises from an $E$-indexed family of maps of topological spaces $B_k(\graf)\to B_{k+1}(\graf)$ that increase the number of points on an edge. Such stabilization maps were known to exist for tails~\cite{AnPark:OSBGC} and for trees at the level of Morse complexes \cite{Ramos:SPHTBG}, but stabilization at arbitrary edges is new, and the sequel is devoted to the study of its properties. 
\end{remark}

It is often useful to consider a smaller variation on the \'{S}wi\k{a}tkowski complex.

\begin{definition}\label{definition: reduced intrinsic definition}
Let $\graf$ be a graph and $U$ a subset of $V(\graf)$. For each $v\in U$, let $\reducedlocalstates{v}\subseteq \localstates{v}$ be the subspace spanned by $\varnothing$ and the differences $h_{ij}\coloneqq{}h_i-h_j$ of half-edges. The \emph{reduced \'{S}wi\k{a}tkowski complex} (relative to $U$), is \[\reducedintrinsic{U}{\graf}\coloneqq\mathbb{Z}[E]\otimes\bigotimes_{v\in V\backslash U}\localstates{v}\otimes \bigotimes_{v\in U}\reducedlocalstates{v},\] considered as a subcomplex and submodule of $\intrinsic{\graf}$. To be explicit, the differential is determined by $\partial(h_{ij})=e(h_i)-e(h_j)$.
\end{definition}

\begin{notation}
When $U=V$ is the full set of vertices, we write $\fullyreduced{\graf}\coloneqq\reducedintrinsic{V}{\graf}$. When $U$ is the set of $1$-valent vertices, we write $\tailreduced{\graf}\coloneqq \reducedintrinsic{U}{\graf}$. Both $\fullyreduced{-}$ and $\tailreduced{-}$ are functorial for graph morphisms. 

When $\graf$ is a disjoint union of star graphs and intervals, we write $\starreduced{\graf}$ (an abuse of notation) for the reduced \'{S}wi\k{a}tkowski complex of $\stargraph{n}$ relative to the set of non-star vertices. Since $\stargraph{1}$ and the interval are isomorphic, this requires a specification of which such components are star graphs (and which vertex is the star vertex). The construction $\starreduced{-}$ is functorial only for graph morphisms $f:\graf_1\to \graf_2$ such that $f^{-1}(v)$ is a star vertex whenever $v$ is a star vertex. Thus, for example, the smoothing $\stargraph{2}\to \intervalgraph$ is allowed, but the inclusion $\intervalgraph\to \stargraph{n}$ of a leg is not.
\end{notation}

\begin{proposition}\label{prop:reduced complex}
For any graph $\graf$ and any $U\subseteq V(\graf)$ containing no isolated vertices, the inclusion $\iota:\reducedintrinsic{U}{\graf}\to \intrinsic{\graf}$ is a quasi-isomorphism. 
\end{proposition}
We omit a detailed proof, which can be seen, e.g., by a spectral sequence argument, filtering by polynomial degree in vertices and half-edges intersecting $U$.

\begin{remark}
The full \'{S}wi\k{a}tkowski complex has a canonical basis, while the reduced version lacks one in general. The reduction at (some subset of) the $1$-valent vertices of a graph retains a canonical basis: if $v$ is $1$-valent, then $\reducedlocalstates{v}$ is spanned by $\{\varnothing\}$. The corresponding reduced complex is the \'{S}wi\k{a}tkowski complex of a ``graph'' in which the $1$-valent vertex has been deleted, leaving a half-open tail. All of our constructions and results can be made rigorous for such non-compact ``graphs.'' 
\end{remark}

Our strategy in proving Theorem~\ref{thm:comparison} will be to apply the tools developed earlier in the paper, especially Theorem~\ref{thm:decomposition}. In order to do so, we must produce a decomposition, introduce Morse theory on the subdivisional configuration spaces of the pieces, and analyze the resulting Morse complexes.

\begin{construction}[Canonical decomposition]
Let $\graf$ be a graph. Subdivide $\graf$ by adding four vertices to each edge; denote the resulting graph by $\graf_\#$. There is a canonical graph morphism $\graf_\#\to \graf$ inducing a homeomorphism on configuration spaces.
 
Removing from each edge of $\graf$ the open interval defined by the outer pair of added vertices produces a graph $\widetilde\graf_0\cong\coprod_{v\in V}\stargraph{d(v)}$. On the other hand, for each edge, we have the closed interval defined by the inner pair of added vertices, and we let $\widetilde\graf_1\cong E\times \intervalgraph$ be the union of these closed intervals. Clearly, we have \[\graf_\#\cong\widetilde\graf_0\coprod_{A_1\times\{0\}}(A_1\times \intervalgraph)\coprod_{A_1\times\{1\}}\widetilde\graf_1,\] where $A_1$ is a finite set with $|A_1|=2|E|$.

We define the \emph{canonical decomposition} of $\graf_\#$ (abusively, of $\graf$) by choosing the full poset of subdivisions $\SD(\graf_\#)$, which is filtered and convergent.
\end{construction}
The main ingredient in the proof of Theorem~\ref{thm:comparison} is the following.

\begin{figure}[ht]
\centering
\begin{tikzpicture}
\begin{scope}[xshift=.3cm]
\draw (2,-.7) arc [start angle=-90, end angle=90, radius=.5];
\draw[white, line width = 3pt] (2,-1.7) arc [line width = 3pt, start angle=-90, end angle=90, radius=.85];
\draw (2,-1.7) arc [start angle=-90, end angle=90, radius=.85];
\draw[white, line width = 3pt]  (2,-2.7) arc [start angle=-90, end angle=90, radius=1.2];
\draw (2,-2.7) arc [start angle=-90, end angle=90, radius=1.2];
\draw[white, line width = 3pt]  (2,-2) arc [start angle=-90, end angle=90, radius=.5];
\draw (2,-2) arc [start angle=-90, end angle=90, radius=.5];
\draw[white, line width = 3pt]  (2,-3) arc [start angle=-90, end angle=90, radius=.85];
\draw (2,-3) arc [start angle=-90, end angle=90, radius=.85];
\draw[white, line width = 3pt]  (2,-3.3) arc [start angle=-90, end angle=90, radius=.5];
\draw (2,-3.3) arc [start angle=-90, end angle=90, radius=.5];
\end{scope}
\fill[black] (0,0) circle (2.5pt);
\draw(0,0) -- (1,.3);
\draw(0,0) -- (1,0);
\draw(0,0) -- (1,-.3);
\draw(1,.3)--(2.3,.3);
\draw(1,0)--(2.3,0);
\draw(1,-.3)--(2.3,-.3);
\draw[densely dotted] (1,.5)--(1,-3.5);
\draw[densely dotted] (2.3,.5)--(2.3,-3.5);
\begin{scope}[yshift=-1cm]
\fill[black] (0,0) circle (2.5pt);
\draw(0,0) -- (1,.3);
\draw(0,0) -- (1,0);
\draw(0,0) -- (1,-.3);
\draw(1,.3)--(2.3,.3);
\draw(1,0)--(2.3,0);
\draw(1,-.3)--(2.3,-.3);
\end{scope}
\begin{scope}[yshift=-2cm]
\fill[black] (0,0) circle (2.5pt);
\draw(0,0) -- (1,.3);
\draw(0,0) -- (1,0);
\draw(0,0) -- (1,-.3);
\draw(1,.3)--(2.3,.3);
\draw(1,0)--(2.3,0);
\draw(1,-.3)--(2.3,-.3);
\end{scope}
\begin{scope}[yshift=-3cm]
\fill[black] (0,0) circle (2.5pt);
\draw(0,0) -- (1,.3);
\draw(0,0) -- (1,0);
\draw(0,0) -- (1,-.3);
\draw(1,.3)--(2.3,.3);
\draw(1,0)--(2.3,0);
\draw(1,-.3)--(2.3,-.3);
\end{scope}
\draw(.5,-3.3) node[below]{$\widetilde{\graf}_0$};
\draw(1.65,-3.3) node[below]{$\vphantom{\widetilde{\graf}_0}A_1\times \intervalgraph$};
\draw(2.8,-3.3) node[below]{$\widetilde{\graf}_1$};
\end{tikzpicture}
\caption{The canonical decomposition of the complete graph $\completegraph{4}$}
\end{figure}
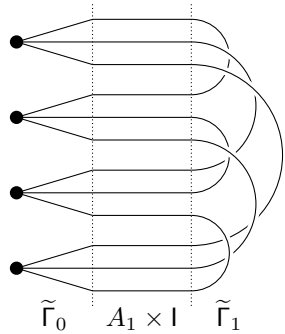

\begin{proposition}\label{prop:flows}
Let $\graf$ be a graph, and regard $B^\SD$ as a local invariant on the canonical decomposition of $\graf$. There is an isotopy invariant local flow on $B^\SD$ such that \begin{align*}I(A_1\times \intervalgraph)&\cong \mathbb{Z}[E]\otimes \mathbb{Z}[E]\\
I(\graf_0)&\cong \bigotimes_{v\in V} \starreduced{\stargraph{d(v)}}\\
I(\graf_1)&\cong\mathbb{Z}[E]
\end{align*} as associative algebras and modules, respectively.
\end{proposition}

The proof of this result amounts to constructing a couple abstract flows, calculating their Morse complexes, and checking various compatibilities. These tasks will occupy our attention in \S\ref{section:interval and star flows}--\ref{section:compatibilities} below, but the idea behind the constructions is very simple. Intuitively, we define a flow on a star by pulling the cone point up, allowing points to flow down the legs, and we define a flow on an interval by allowing points to flow according to some fixed orientation. The reader who finds this heuristic sufficiently convincing may skip ahead to the computations.

For now, we deduce the following:

\begin{proof}[Proof of Theorem~\ref{thm:comparison}, construction of isomorphism]
Applying Theorem~\ref{thm:decomposition} with the local flow of Proposition~\ref{prop:flows} produces an isomorphism\[H_*(B(\graf))\cong H_*\left(\bigotimes_{v\in V} \starreduced{\stargraph{d(v)}}\bigotimes^{\mathbb{L}}_{\mathbb{Z}[E]\otimes\mathbb{Z}[E]}\mathbb{Z}[E]\right).\] Since $\bigotimes_{v\in V}\starreduced{\stargraph{d(v)}}$ is a free $\mathbb{Z}[E]\otimes \mathbb{Z}[E]$-module, the derived tensor product is computed by the ordinary tensor product, and the proof is complete upon noting the canonical isomorphism \[
\bigotimes_{v\in V} \starreduced{\stargraph{d(v)}}\bigotimes_{\mathbb{Z}[E]\otimes\mathbb{Z}[E]}\mathbb{Z}[E]\cong \intrinsic{\graf}.\qedhere
\]
\end{proof}

\subsection{Interval and star flows}\label{section:interval and star flows} In this section, we construct the abstract flows (in the sense of \S\ref{subsec:subdivisional morse theory}) that form the building blocks of the local flows alluded to in Proposition~\ref{prop:flows}. These flows are inspired by the discrete flows of~\cite{FarleySabalka:DMTGBG} and could be constructed in the same manner, but, in working locally, we are able to avoid the machinery of discrete Morse theory. 

\begin{definition}[Interval flow]\label{defi: interval flow}
Let $\intervalgraph\to\intervalgraph'$ be a subdivision. We fix an orientation of $\intervalgraph'$ and define an order on the vertices by declaring the negative direction to be the direction of decrease. Define an endomorphism $\pi$ of $C(B_k^\Box(\intervalgraph'))$ by declaring that $\pi$
\begin{enumerate}
\item takes any (positively signed) generator which is a set of $k$ many $0$-cells to the (positively signed) set of the $k$ least $0$-cells and
\item takes any generator which is a set containing a $1$-cell to zero.
\end{enumerate}
\end{definition}

Using that the complex $B_k^\Box(\intervalgraph')$ is either contractible or empty, depending on whether $k$ is greater than the number of 0-cells, the following result is easily verified.

\begin{lemma} The interval flow is an abstract flow on $C(B_k^\Box(\intervalgraph'))$.
\end{lemma}

\begin{definition}[Star flow]\label{defi: star flow}
Let $\stargraph{n}\to \stargraph{n}'$ be a subdivision. For convenience, order the vertices in the $i$th leg by declaring the direction away from the star vertex $v$ to be the direction of decrease; call them $v_{i,0},\ldots, v_{i,N_i}$, with $v_0$ the least vertex and $v_{N_i}=v$ the star vertex. Let $e_{i,j}$ be the $1$-cell containing $v_{i,j-1}$ and $v_{i,j}$. Given a tuple $(r_1,\ldots, r_n)$ with $0\le r_i\le N_i$, write $V_{\vec{r}}$ for the set of vertices containing the $r_i$ least vertices in the $i$th leg. Define an endomorphism $\pi$ of $C(B_k^\Box(\stargraph{n}'))$ by declaring that $\pi$
\begin{enumerate}
\item takes any (positively signed) set of $0$-cells not including $v$ with $r_i$ $0$-cells in the $i$th leg to the (positively signed) set $V_{\vec{r}}$,
\item takes any (positively signed) set of $0$-cells including $v$ and $r_i$ additional $0$-cells in the $i$th leg to the (positively signed) union of $V_{\vec{r}}$ and $\{v\}$,
\item takes any set containing a $1$-cell not of the form $e_{j,N_j}$ to zero, and
\item takes any set containing $e_{j,N_j}$ and $r_i$ additional (positively signed) $0$-cells in the $i$th leg (necessarily not including $v$) to the sum 
\[\sum_{r={r_j}+1}^{N_j} V_{\vec{r}}\cup \{e_{j,r}\}.\]
\end{enumerate}
\end{definition}
See Figure~\ref{figure:star flow}.
\begin{figure}[ht]
\begin{center}
\subfigure[0-cell containing the star vertex]{\makebox[0.4\textwidth]{$
\begin{tikzpicture}
\begin{scope}
\draw[lightgray](0,0)--(245:1.5);
\draw[lightgray](0,0)--(270:1.5);
\draw[lightgray](0,0)--(295:1.5);
\fill[black](0,0) circle (2.5pt);
\fill[lightgray](245:0.375) circle (2.5pt);
\fill[lightgray](245:0.75) circle (2.5pt);
\fill[black](245:1.125) circle (2.5pt);
\fill[lightgray](245:1.5) circle (2.5pt);
\fill[lightgray](270:0.75) circle (2.5pt);
\fill[black](270:1.5) circle (2.5pt);
\fill[black](295:0.5) circle (2.5pt);
\fill[black](295:1) circle (2.5pt);
\fill[lightgray](295:1.5) circle (2.5pt);
\draw(0.9,-0.5) node{$\stackrel{\pi}{\mapsto}$};
\end{scope}
\begin{scope}[xshift=1.8cm]
\draw[lightgray](0,0)--(245:1.5);
\draw[lightgray](0,0)--(270:1.5);
\draw[lightgray](0,0)--(295:1.5);
\fill[black](0,0) circle (2.5pt);
\fill[lightgray](245:0.375) circle (2.5pt);
\fill[lightgray](245:0.75) circle (2.5pt);
\fill[lightgray](245:1.125) circle (2.5pt);
\fill[black](245:1.5) circle (2.5pt);
\fill[lightgray](270:0.75) circle (2.5pt);
\fill[black](270:1.5) circle (2.5pt);
\fill[lightgray](295:0.5) circle (2.5pt);
\fill[black](295:1) circle (2.5pt);
\fill[black](295:1.5) circle (2.5pt);
\end{scope}
\end{tikzpicture}
$}}
\subfigure[1-cell containing a central edge]{\makebox[0.55\textwidth]{$
\begin{tikzpicture}
\begin{scope}
\draw[lightgray](0,0)--(245:1.5);
\draw[lightgray](0,0)--(270:1.5);
\draw[lightgray](0,0)--(295:1.5);
\fill[black](0,0) circle (2.5pt);
\draw[black, line width=.5mm](0,0) -- (245:0.375);
\fill[black](245:0.375) circle (2.5pt);
\fill[lightgray](245:0.75) circle (2.5pt);
\fill[black](245:1.125) circle (2.5pt);
\fill[lightgray](245:1.5) circle (2.5pt);
\fill[lightgray](270:0.75) circle (2.5pt);
\fill[black](270:1.5) circle (2.5pt);
\fill[black](295:0.5) circle (2.5pt);
\fill[black](295:1) circle (2.5pt);
\fill[lightgray](295:1.5) circle (2.5pt);
\draw(0.9,-0.5) node{$\stackrel{\pi}{\mapsto}$};
\end{scope}
\begin{scope}[xshift=1.8cm]
\draw[lightgray](0,0)--(245:1.5);
\draw[lightgray](0,0)--(270:1.5);
\draw[lightgray](0,0)--(295:1.5);
\fill[black](0,0) circle (2.5pt);
\draw[black, line width=.5mm](0,0) -- (245:0.375);
\fill[black](245:0.375) circle (2.5pt);
\fill[lightgray](245:0.75) circle (2.5pt);
\fill[lightgray](245:1.125) circle (2.5pt);
\fill[black](245:1.5) circle (2.5pt);
\fill[lightgray](270:0.75) circle (2.5pt);
\fill[black](270:1.5) circle (2.5pt);
\fill[lightgray](295:0.5) circle (2.5pt);
\fill[black](295:1) circle (2.5pt);
\fill[black](295:1.5) circle (2.5pt);
\draw(0.9,-0.5) node{$+$};
\end{scope}
\begin{scope}[xshift=3.6cm]
\draw[lightgray](0,0)--(245:1.5);
\draw[lightgray](0,0)--(270:1.5);
\draw[lightgray](0,0)--(295:1.5);
\fill[lightgray](0,0) circle (2.5pt);
\fill[black](245:0.375) circle (2.5pt);
\draw[black, line width=.5mm](245:0.375) -- (245:0.75);
\fill[black](245:0.75) circle (2.5pt);
\fill[lightgray](245:1.125) circle (2.5pt);
\fill[black](245:1.5) circle (2.5pt);
\fill[lightgray](270:0.75) circle (2.5pt);
\fill[black](270:1.5) circle (2.5pt);
\fill[lightgray](295:0.5) circle (2.5pt);
\fill[black](295:1) circle (2.5pt);
\fill[black](295:1.5) circle (2.5pt);
\draw(0.9,-0.5) node{$+$};
\end{scope}
\begin{scope}[xshift=5.4cm]
\draw[lightgray](0,0)--(245:1.5);
\draw[lightgray](0,0)--(270:1.5);
\draw[lightgray](0,0)--(295:1.5);
\fill[lightgray](0,0) circle (2.5pt);
\fill[lightgray](245:0.375) circle (2.5pt);
\fill[black](245:0.75) circle (2.5pt);
\draw[black, line width=.5mm](245:0.75) -- (245:1.125);
\fill[black](245:1.125) circle (2.5pt);
\fill[black](245:1.5) circle (2.5pt);
\fill[lightgray](270:0.75) circle (2.5pt);
\fill[black](270:1.5) circle (2.5pt);
\fill[lightgray](295:0.5) circle (2.5pt);
\fill[black](295:1) circle (2.5pt);
\fill[black](295:1.5) circle (2.5pt);
\end{scope}\end{tikzpicture}
$}}
\end{center}
\vspace{-1em}
\caption{Examples of two of the cases for star flow (Definition~\ref{defi: star flow})}
\label{figure:star flow}
\end{figure}
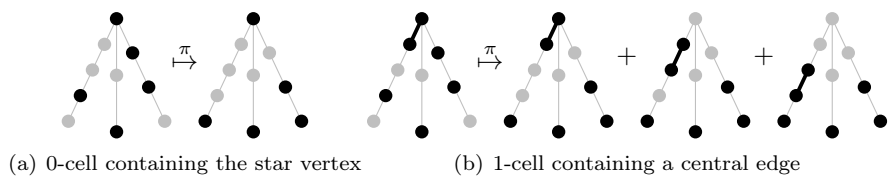

\begin{lemma}
The star flow is an abstract flow on $C(B_k^\Box(\stargraph{n}'))$. 
\end{lemma}
\begin{proof}
Idempotence is immediate except in the last case. There the image of the cell containing $e_{j,N_j}$ and $r_i$ vertices in the $i$th leg is the sum of the cell $V_{\vec{r}}\cup\{e_{j,N_j}\}$ with several cells in the kernel of $\pi$. The map $\pi$ is a chain map by inspection (it suffices to check $2$-cell and $1$-cell generators).

Let us see that $\pi$ is a quasi-isomorphism. For brevity, we write $C$ for $C(B_k^\Box(\stargraph{n}'))$ and $I$ for its Morse complex. Consider the subcomplex $A$ of $C$  spanned by generators (i.e., sets of cells of $B_k^\Box(\stargraph{n}')$) which do not contain any $1$-cell intersecting $v$. Since $\pi(A)\subseteq A$, there is a commuting diagram \[
\begin{tikzcd}
0\ar[r]&
A\ar[d]\ar[r]
&
C\ar[d]\ar[r]
&
C/A\ar[d]\ar[r]
&0
\\
0\ar[r]&
A\cap I\ar[r]
&
I\ar[r]
&
I/(A\cap I)\ar[r]
&0
\end{tikzcd}
\] of exact sequences.
The quotient complex $C/A$ is linearly isomorphic to the subspace $B\subseteq C$ spanned by generating cells containing an edge intersecting $v$, but, in the quotient, the differential ignores the special edge. Similarly, $I/(A\cap I)\cong(\pi(B),0)$. It follows that $A\xrightarrow{\pi} (A\cap I)$ and $C/A\xrightarrow{\pi} I/(A\cap I)$ are both quasi-isomorphisms, since the domains are either contractible or empty in every weight, while the codomains have no differentials. The claim follows by the five lemma.
\end{proof}

\begin{lemma}\label{lem:corolla morse}
There is a canonical chain map $I(B_k^\Box(\stargraph{n}'))\to \starreduced{\stargraph{n}}_k$ that is an isomorphism if $\stargraph{n}'$ has at least $k+1$ vertices in each leg.
\end{lemma}
\begin{proof} 
A basis for the weight $k$ subcomplex of $\starreduced{\stargraph{n}}$ is given by the set of elements of the form $e_1^{r_1}\cdots e_n^{r_n}\otimes x$, where $x\in\{\varnothing, v, h_1,\ldots, h_n\}$, the $r_j$ are non-negative integers, and $\sum r_j+\mathrm{wt}(x)=k$.

The coimage of $\pi$ in the star flow is a quotient of the set of cells containing either the $0$-cell $v$, the $1$-cell $e_{j,N_j}$, or no cell containing $v$, along with $r_j$ additional $0$-cells in the $j$th leg for each $j$. The relation identifies two such configurations with the same central configuration and the same $r_j$ for each $j$ but different choices of $0$-cells away from the center vertex. Then we send such a configuration to 
the element $e_1^{r_1}\cdots e_n^{r_n}\otimes x$ where $x$ is: \begin{enumerate}
\item the element $v$ if $c$ contains the star vertex $v$,
\item the element $h_i$ if $c$ contains the edge $e_{j,N_j}$, and
\item the element $\varnothing$ otherwise.
\end{enumerate} This map is prima facie injective, and it is surjective if $\stargraph{n}'$ has the hypothesized number of vertices in each leg. It is straightforward to verify that this map is also a chain map.
\end{proof}

\subsection{Compatibilities}\label{section:compatibilities} In this section, we check that the abstract flows constructed in the previous section are compatible with inclusion and subdivision, completing the proof of Proposition~\ref{prop:flows} and thereby of the isomorphism statement of Theorem~\ref{thm:comparison}.

Let $\graf\to \graf'$ be a subdivision of graphs, with $\graf$ either the interval graph $\intervalgraph$ or the star graph $\stargraph{n}$, and let $\Xigraph'\subseteq \graf'$ be a subgraph that is a subdivision of a disjoint union $\Xigraph$ of intervals and stars; thus, $\Xigraph\to \graf$ is a graph morphism. We equip $B_k^\Box(\graf')$ with either the interval flow or the star flow, depending on the case in consideration. Through the isomorphism of Lemma~\ref{lem:box monoidal}, the configuration complex $B_k^\Box(\Xigraph')$ also inherits an abstract flow, which depends on choices of star points, since subdivided intervals, 1-stars, and 2-stars are intrinsically indistinguishable, and on choices of orientation for the interval components.

\begin{lemma}\label{lem:inclusion compatibility}
The inclusion $i:B_k^\Box(\Xigraph')\to B_k^\Box(\graf')$ is flow compatible provided that any component of $\Xigraph'$ containing a star vertex of $\graf'$ carries the star flow at that vertex. Moreover, in this case, the following diagram commutes:\[\begin{tikzcd}
{I(B_k^\Box(\Xigraph'))}
\ar[d]
\ar{rr}{I(B_k^\Box(i))}
&&
{I(B_k^\Box(\graf'))}
\ar[d]\\
\starreduced{\Xigraph}_k
\ar{rr}{\starreduced{\Xigraph\to \graf}}
&&\starreduced{\graf}_k.
\end{tikzcd}\]
\end{lemma}

\begin{proof}
We prove only the case when $\graf$ is a star graph and $\graf'$ is equipped with the star flow (the interval case is easier), focusing on the verification of flow compatibility. Write $\pi_\Xi$ for the abstract flow on $B_k^\Box(\Xigraph')$ and $\pi_\Gamma$ for the abstract flow on $B_k^\Box(\graf')$. Let $c$ be a cell of $B_k^\Box(\Xigraph')$, written as a symmetric product of cells of $\Xigraph'$. 

If $c$ contains an edge that is not adjacent to the star vertex of $\graf'$, then so does $i(c)$. Then by Definitions~\ref{defi: interval flow}~and~\ref{defi: star flow}, $\pi_\Gamma(i(c))=0=\pi_\Gamma(i(\pi_\Xi(c)))$. 
On the other hand, if $c$ contains no such edge, then by the same definitions and inspection, $\pi_\Xi(c)$ is the sum of a cell $c'$, which is obtained by moving the vertices of $c$ into their minimal positions in the relevant components of $\Xigraph'$, with a linear combination of cells, each containing an edge not adjacent to the star vertex of $\graf'$. The image of these latter cells are all in the kernel of $\pi_\Gamma$, so $\pi_\Gamma(i(\pi_\Xi(c))=\pi_\Gamma(i(c'))$. But the same characterization also shows that $\pi_\Gamma(i(c'))=\pi_\Gamma(i(c))$, since both are obtained by moving the same number of vertices in each leg of $\graf'$ into their minimal positions.

The commutativity claim is essentially immediate from what has already been said and the description of the isomorphism in Lemma~\ref{lem:corolla morse}.
\end{proof}

Similar considerations apply in the case of a subdivision:

\begin{lemma}\label{lem:subdivision compatibility}
For subdivisions $\intervalgraph'\to \intervalgraph''$ and $\stargraph{n}'\to \stargraph{n}''$ of the interval and the star graph $\stargraph{n}$ respectively, the induced subdivisional embeddings $B_k^\Box(\intervalgraph')\to B_k^\Box(\intervalgraph'')$ and $B_k^\Box(\stargraph{n}')\to B_k^\Box(\stargraph{n}'')$ are flow compatible, and the following diagrams commute:
\[\begin{tikzcd}I(B_k^\Box(\intervalgraph'))\ar[r]\ar[d]&I(B_k^\Box(\intervalgraph''))\ar[d]&I(B_k^\Box(\stargraph{n}'))\ar[r]\ar[d]&I(B_k^\Box(\stargraph{n}''))\ar[d]\\
\mathbb{Z}
\ar[no head, r, shift left=.97pt]
\ar[no head, r, shift right=.97pt]
&\mathbb{Z}&\starreduced{\stargraph{n}}_k
\ar[no head, r, shift left=.97pt]
\ar[no head, r, shift right=.97pt]
&\starreduced{\stargraph{n}}_k.
\end{tikzcd}\] 
\end{lemma}

\begin{proof}[Proof of Proposition~\ref{prop:flows}] By Lemma~\ref{lem:subdivision compatibility}, the interval and star flows determine subdivisional flows on any disjoint union of intervals and stars equipped with any subdivisional structure. By the same result and Lemma~\ref{lem:inclusion compatibility}, any embedding among such is flow compatible if the preimage of every star point is a star point. We use the star flow on component type basics in $\graf_0$ with the vertices of $\graf$ as star vertices and the interval flow for all other basics. Since these subdivisional flows respect the isomorphisms $B^\SD(U\amalg V)\cong B^\SD(U)\times B^\SD(V)$, we obtain a local flow on the local invariant $B^\SD$, isotopy invariant by Lemma \ref{lem:inclusion compatibility}. By inspection in the interval case and Lemma~\ref{lem:corolla morse}, there is an identification of the Morse complex $I(B^\SD(U))\cong \starreduced{U}$ which is natural for inclusions among disjoint unions of basics by Lemmas~\ref{lem:inclusion compatibility} and~\ref{lem:subdivision compatibility}. Isotopy invariance and the identification of the algebra and module structures follow.
\end{proof}

\subsection{Naturality for graph morphisms} In this section, we show that the isomorphism of Theorem~\ref{thm:comparison} is functorial. To use the naturality clause of Theorem~\ref{thm:decomposition} we will show that every graph morphism lifts to a flow compatible map of local invariants. It suffices to do so for graph embeddings and smoothings separately.

If $f:\graf_1\to \graf_2$ is a graph embedding, then every edge of $\graf_1$ is mapped homeomorphically to an edge of $\graf_2$, so we may choose auxiliary vertices so that there is a commuting diagram of graph morphisms:
\[\begin{tikzcd}
\graf_1\ar{r}{f}&\graf_2\\
(\graf_1)_\#\ar[r]\ar[u]&(\graf_2)_\#.\ar{u}
\end{tikzcd}\] 
Then from the definition of the canonical decomposition, we have a commuting diagram of graph embeddings, and, in particular, of subdivisional spaces, which is to say a map of decompositions:
\[\begin{tikzcd}
\displaystyle\coprod_{v\in{V(\graf_1)}}\stargraph{d(v)}\ar{d}&\displaystyle\coprod_{e\in E(\graf_1)}\pt\amalg\pt\ar{d}\ar[r]\ar[l]& \displaystyle\coprod_{e\in E(\graf_1)} \intervalgraph\ar{d}\\
\displaystyle\coprod_{v\in{V(\graf_2)}}\stargraph{d(v)}&\displaystyle\coprod_{e\in E(\graf_2)}\pt\amalg\pt\ar[r]\ar[l]& \displaystyle\coprod_{e\in E(\graf_2)} \intervalgraph.
\end{tikzcd}\] 

It follows from Lemmas~\ref{lem:inclusion compatibility} and~\ref{lem:subdivision compatibility} that the identity map on $B^\SD$ is flow compatible when regarded as a map between the induced local invariants on these two decomositions. Since every basic is a disjoint union of intervals and stars, the same lemmas show that the induced map at the level of Morse complexes coincides with the induced map on \'{S}wi\k{a}tkowski complexes. Thus, naturality holds for graph embeddings.

The case of a smoothing reduces to the following scenario. Fix an edge $e_0\in \graf$, and let $\graf'$ be the graph obtained from $\graf$ by adding a bivalent vertex $v_0$ to $e_0$, subdividing it into $e_1$ and $e_2$. There is a smoothing $f:\graf'\to \graf$, and every smoothing is a composite of isomorphisms and smoothings of this kind. This smoothing does not respect the canonical decompositions of $\graf'$ and $\graf$, but it does respect a different pair of decompositions $\E'$ and $\E$, depicted in the following diagram and Figure~\ref{fig: 4 decompositions}:
\[\begin{tikzcd}
\displaystyle\coprod_{v_0\neq v\in {V'}}\stargraph{d(v)}
\ar[no head, d, shift left=.97pt]
\ar[no head, d, shift right=.97pt]
&\displaystyle \coprod_{e\in E} \pt\amalg\pt
\ar[no head, d, shift left=.97pt]
\ar[no head, d, shift right=.97pt] 
\ar[r]\ar[l]& \displaystyle \coprod_{e\in E'} \intervalgraph
\ar[no head, d, shift left=.97pt]
\ar[no head, d, shift right=.97pt]
&\pt\amalg\pt
\ar[no head, d, shift left=.97pt]
\ar[no head, d, shift right=.97pt]
\ar{r}\ar{l}&\stargraph{2}\ar{d}\\
\displaystyle\coprod_{v\in {V}}\stargraph{d(v)}&\displaystyle \coprod_{e\in E} \pt\amalg\pt \ar[r]\ar[l]& \displaystyle \coprod_{e\in E'} \intervalgraph&\pt\amalg\pt\ar{r}\ar{l}& \intervalgraph.
\end{tikzcd}\]
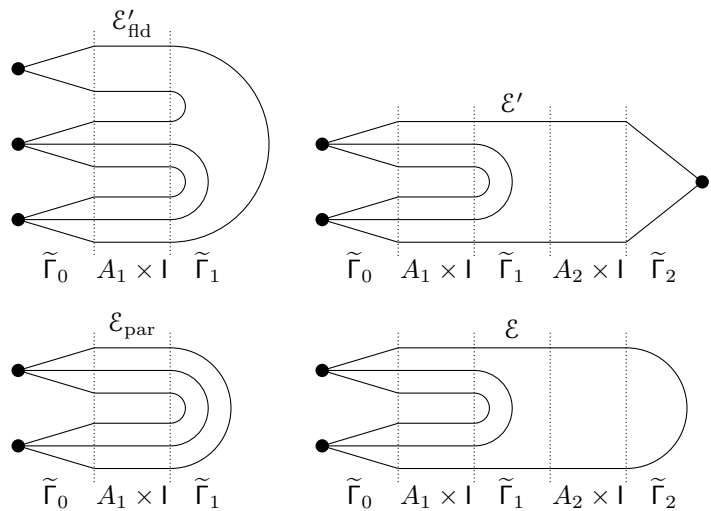
\begin{figure}[ht]
\centering
\begin{tikzpicture}
\begin{scope}[yshift=-1cm]
\fill[black] (0,0) circle (2.5pt);
\fill[black] (0,1) circle (2.5pt);
\draw(0,0) -- (1,.3);
\draw(0,0) -- (1,0);
\draw(0,0) -- (1,-.3);
\draw(1,.3)--(2,.3);
\draw(1,0)--(2,0);
\draw(1,-.3)--(2,-.3);
\draw(2,.7)--(1,.7)--(0,1)--(1,1.3)--(2,1.3);
\draw[densely dotted] (1,1.5)--(1,-1.5);
\draw[densely dotted] (2,1.5)--(2,-1.5);
\begin{scope}[yshift=-1cm]
\fill[black] (0,0) circle (2.5pt);
\draw(0,0) -- (1,.3);
\draw(0,0) -- (1,0);
\draw(0,0) -- (1,-.3);
\draw(1,.3)--(2,.3);
\draw(1,0)--(2,0);
\draw(1,-.3)--(2,-.3);
\end{scope}
\begin{scope}[xshift=.0cm]
\draw (2,-1.3) arc [start angle=-90, end angle=90, x radius = 1.3, y radius=1.3];
\draw (2,-1) arc [start angle=-90, end angle=90, radius=.5];
\draw (2,-.7) arc [start angle=-90, end angle=90, radius=.2];
\draw (2,.3) arc [start angle=-90, end angle=90, radius=.2];
\end{scope}
\draw(.5,-1.3) node[below]{$\widetilde{\graf}_0$};
\draw(1.5,-1.3) node[below]{$\vphantom{\widetilde{\graf}_0}A_1\times \intervalgraph$};
\draw(1.5,1.3) node[above]{$\E'_{\mathrm{fld}}$};
\draw(2.5,-1.3) node[below]{$\widetilde{\graf}_1$};
\end{scope}
\begin{scope}[yshift=-4cm]
\fill[black] (0,0) circle (2.5pt);
\draw(0,0) -- (1,.3);
\draw(0,0) -- (1,0);
\draw(0,0) -- (1,-.3);
\draw(1,.3)--(2,.3);
\draw(1,0)--(2,0);
\draw(1,-.3)--(2,-.3);
\draw[densely dotted] (1,.5)--(1,-1.5);
\draw[densely dotted] (2,.5)--(2,-1.5);
\begin{scope}[yshift=-1cm]
\fill[black] (0,0) circle (2.5pt);
\draw(0,0) -- (1,.3);
\draw(0,0) -- (1,0);
\draw(0,0) -- (1,-.3);
\draw(1,.3)--(2,.3);
\draw(1,0)--(2,0);
\draw(1,-.3)--(2,-.3);
\end{scope}
\begin{scope}[xshift=.0cm]
\draw (2,-1.3) arc [start angle=-90, end angle=90, radius=.8];
\draw (2,-1) arc [start angle=-90, end angle=90, radius=.5];
\draw (2,-.7) arc [start angle=-90, end angle=90, radius=.2];
\end{scope}
\draw(.5,-1.3) node[below]{$\widetilde{\graf}_0$};
\draw(1.5,-1.3) node[below]{$\vphantom{\widetilde{\graf}_0}A_1\times \intervalgraph$};
\draw(1.5,.3) node[above]{$\E_{\mathrm{par}}$};
\draw(2.5,-1.3) node[below]{$\widetilde{\graf}_1$};
\end{scope}
\begin{scope}[xshift=4cm,yshift=-1cm]
\fill[black] (0,0) circle (2.5pt);
\draw(0,0) -- (1,.3);
\draw(0,0) -- (1,0);
\draw(0,0) -- (1,-.3);
\draw(1,.3)--(2,.3);
\draw(1,0)--(2,0);
\draw(1,-.3)--(2,-.3);
\draw[densely dotted] (1,.5)--(1,-1.5);
\draw[densely dotted] (2,.5)--(2,-1.5);
\draw[densely dotted] (3,.5)--(3,-1.5);
\draw[densely dotted] (4,.5)--(4,-1.5);
\begin{scope}[yshift=-1cm]
\fill[black] (0,0) circle (2.5pt);
\draw(0,0) -- (1,.3);
\draw(0,0) -- (1,0);
\draw(0,0) -- (1,-.3);
\draw(1,.3)--(2,.3);
\draw(1,0)--(2,0);
\draw(1,-.3)--(2,-.3);
\end{scope}
\begin{scope}[xshift=.0cm]
\draw(2,.3) -- (4,.3);
\draw(2,-1.3) -- (4,-1.3);
\draw(4,.3) -- (5,-.5) -- (4,-1.3);
\fill[black] (5,-.5) circle (2.5pt);
\draw (2,-1) arc [start angle=-90, end angle=90, radius=.5];
\draw (2,-.7) arc [start angle=-90, end angle=90, radius=.2];
\end{scope}
\draw(.5,-1.3) node[below]{$\widetilde{\graf}_0$};
\draw(1.5,-1.3) node[below]{$\vphantom{\widetilde{\graf}_0}A_1\times \intervalgraph$};
\draw(2.5,-1.3) node[below]{$\widetilde{\graf}_1$};
\draw(2.5,.3) node[above]{$\E'$};
\draw(3.5,-1.3) node[below]{$\vphantom{\widetilde{\graf}_1}A_2\times \intervalgraph$};
\draw(4.5,-1.3) node[below]{$\widetilde{\graf}_2$};
\end{scope}
\begin{scope}[xshift = 4cm, yshift=-4cm]
\fill[black] (0,0) circle (2.5pt);
\draw(0,0) -- (1,.3);
\draw(0,0) -- (1,0);
\draw(0,0) -- (1,-.3);
\draw(1,.3)--(2,.3);
\draw(1,0)--(2,0);
\draw(1,-.3)--(2,-.3);
\draw[densely dotted] (1,.5)--(1,-1.5);
\draw[densely dotted] (2,.5)--(2,-1.5);
\draw[densely dotted] (3,.5)--(3,-1.5);
\draw[densely dotted] (4,.5)--(4,-1.5);
\begin{scope}[yshift=-1cm]
\fill[black] (0,0) circle (2.5pt);
\draw(0,0) -- (1,.3);
\draw(0,0) -- (1,0);
\draw(0,0) -- (1,-.3);
\draw(1,.3)--(2,.3);
\draw(1,0)--(2,0);
\draw(1,-.3)--(2,-.3);
\end{scope}
\begin{scope}[xshift=.0cm]
\draw(2,.3) -- (4,.3);
\draw(2,-1.3) -- (4,-1.3);
\draw (4,-1.3) arc [start angle=-90, end angle=90, radius=.8];
\draw (2,-1) arc [start angle=-90, end angle=90, radius=.5];
\draw (2,-.7) arc [start angle=-90, end angle=90, radius=.2];
\end{scope}
\draw(.5,-1.3) node[below]{$\widetilde{\graf}_0$};
\draw(1.5,-1.3) node[below]{$\vphantom{\widetilde{\graf}_0}A_1\times \intervalgraph$};
\draw(2.5,-1.3) node[below]{$\widetilde{\graf}_1$};
\draw(2.5,.3) node[above]{$\E$};
\draw(3.5,-1.3) node[below]{$\vphantom{\widetilde{\graf}_1}A_2\times \intervalgraph$};
\draw(4.5,-1.3) node[below]{$\widetilde{\graf}_2$};
\end{scope}
\end{tikzpicture}
\caption{The canonical decomposition (labelled $\E'_{\mathrm{fld}}$) and the decomposition $\E'$ of the theta graph $\thetagraph{3}$ with an additional bivalent vertex are depicted in the first row. The canonical decomposition (labelled $\E_{\mathrm{par}}$) and the decomposition $\E$ of the theta graph $\thetagraph{3}$ itself are depicted in the second row. The notation used is that of Appendix~\ref{section:comparing decompositions}; compare Figure~\ref{figure: three decompositions}.}\label{fig: 4 decompositions}
\end{figure}
As before, Lemma~\ref{lem:inclusion compatibility} guarantees that corresponding map of local invariants is flow compatible; moreover, Theorem~\ref{thm:decomposition} implies that the right-hand cell commutes in the diagram below. 
The left-hand cell commutes by inspection. Since the commuting of the outer square is what is to be proven, it will then suffice to verify that the two remaining triangles also commute.
\[
\begin{tikzcd}[column sep=.1mm, row sep=2ex]
H_*(\intrinsic{\graf'})
\ar{ddddd}[swap, description]{H_*(\intrinsic{f})}
\ar{dr}{\simeq}
\ar{rr}{\simeq}
&&
H_*(B(\graf'))
\ar{ddddd}[description]{H_*(B(f))}
\\
&\displaystyle H_*\left(\bigotimes_{v\in V'}\starreduced{\stargraph{d(v)}}\bigotimes_{\mathbb{Z}[E']\otimes\mathbb{Z}[E']}\mathbb{Z}[E']\right)
\ar{d}
\ar{ur}{\simeq}
\\
&\displaystyle H_*\left(\bigotimes_{v\in V}\starreduced{\stargraph{d(v)}}\bigotimes_{\mathbb{Z}[E]\otimes\mathbb{Z}[E]}\mathbb{Z}[E']\bigotimes_{\mathbb{Z}[e_1]\otimes\mathbb{Z}[e_2]}\starreduced{\stargraph{2}}\right)\ar{d}
\ar[uur, start anchor = {[xshift=2.2cm]},swap,"\simeq"]
\ar[uur, start anchor = {[xshift=.35cm]},phantom, "\text{\S\ref{subsection:folding}}", pos=.6]
\\
&\displaystyle H_*\left(\bigotimes_{v\in V}\starreduced{\stargraph{d(v)}}\bigotimes_{\mathbb{Z}[E]\otimes\mathbb{Z}[E]}\mathbb{Z}[E']\bigotimes_{\mathbb{Z}[e]\otimes\mathbb{Z}[e]}\mathbb{Z}[e]\right)\ar{d}\ar[ddr, start anchor = {[xshift=2.2cm]},"\simeq"]
\ar[ddr, start anchor = {[xshift=.35cm]},phantom, "\text{\S\ref{subsection:parenthesization}}", pos=.6]
\\
&\displaystyle H_*\left(\bigotimes_{v\in V}\starreduced{\stargraph{d(v)}}\bigotimes_{\mathbb{Z}[E]\otimes\mathbb{Z}[E]}\mathbb{Z}[E]\right)\ar{dr}[swap]{\simeq}
\\
H_*(\intrinsic{\graf})
\ar{rr}[swap]{\simeq}
\ar{ur}[swap]{\simeq}
&&H_*(B(\graf)).
\end{tikzcd} 
\]

This question of comparing different decompositions of a fixed complex is the subject of Appendix~\ref{section:comparing decompositions}. The main tool there, Proposition~\ref{prop:abstract comparison}, is used twice. In \S\ref{subsection:parenthesization}, it is used to establish a compatibility which is applicable to $\E$ and the canonical decomposition of $\graf$. This first compatibility guarantees the commutativity of the lower triangle provided that every inclusion among disjoint unions of basics for the two decompositions is flow compatible. Then, in \S\ref{subsection:folding}, the proposition is used to establish a compatibility applicable to $\E'$ and the canonical decomposition of $\graf'$. This second compatibility guarantees the commutativity of the upper triangle provided the corresponding inclusions are flow compatible. As before, these compatibilities are a direct consequence of Lemma~\ref{lem:inclusion compatibility}.

\section{Homology of graph braid groups}\label{section:tools}
We begin our computational study of the homology of configuration spaces of graphs in earnest. Highlights include Corollary~\ref{cor: les vertex reduced}, an exact sequence associated to the removal of a vertex; Proposition~\ref{proposition: edge multiplication is injective}, which asserts that multiplication by any fixed edge is injective; Proposition~\ref{proposition: top homology}, which identifies the homology in top degree in the case of a trivalent graph; Proposition~\ref{prop:complete graph}, a full computation in the case of the complete graph on four vertices; and Appendix~\ref{appendix:degree one}, which gives a streamlined derivation of the characterization of the first homology due to \cite{KoPark:CGBG}. Except for this last, these results appear to be new.
\subsection{Graphs, star and loop classes, and relations}\label{subsection:examples} Since a graph embedding induces a map at the level of \'{S}wi\k{a}tkowski complexes, some homology classes originate in subgraphs. Moreover, the homology of the configuration spaces of a graph is constrained by relations originating in its subgraphs. In this section, we acquaint ourselves with a few useful generators and relations of this kind.
The generators we discuss are present in~\cite[Section 4]{HarrisonKeatingRobbinsSawicki:NPQSG}, as is the $Q$-relation of Definition~\ref{definition:relations} (present as their equation (9) and Lemma~3).

Recall that the star graph $\stargraph{n}$ is the cone on the set $\{1,\ldots, n\}$. This graph has $n+1$ vertices: a central vertex $v_0$ of valence $n$ and $n$ vertices $\{v_1,\ldots, v_n\}$ of valence $1$. Its half-edges $E$ are $\{e_1,\ldots, e_n\}$, and it has $2n$ corresponding half-edges $H$, of which $n$ of these, $h_1,\ldots, h_n$, are at $v_0$ and one, $h'_i$, is at $v_i$ for $i\ne 0$.

We will also consider several other graphs in this section.

\begin{definition}
The \emph{cycle graph} $\cyclegraph{n}$ is a topological circle equipped with $n$ bivalent vertices and $n$ edges. 
The \emph{lollipop graph} $\lollipopgraph{n}$ is obtained by attaching an extra edge $e_0$ to the cycle graph $\cyclegraph{n}$ at one half-edge. The \emph{theta graph} $\thetagraph{n}$ is the topological suspension (double cone) on $n$ points, with vertices the cone points $v_1$ and $v_2$.
\end{definition}
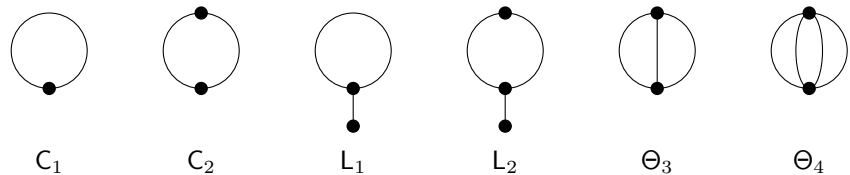
\begin{figure}[ht]
\centering
\begin{tikzpicture}
\fill[black] (0,-.5) circle (2.5pt);
\draw(0,0) circle (.5cm);
\draw(0,-1.2) node[below]{$\cyclegraph{1}$};
\begin{scope}[xshift=2cm]
\fill[black] (0,-.5) circle (2.5pt);
\fill[black] (0,.5) circle (2.5pt);
\draw(0,0) circle (.5cm);
\draw(0,-1.2) node[below]{$\cyclegraph{2}$};
\end{scope}
\begin{scope}[xshift=4cm]
\fill[black] (0,-.5) circle (2.5pt);
\fill[black] (0,-1) circle (2.5pt);
\draw(0,0) circle (.5cm);
\draw(0,-1) -- (0,-.5);
\draw(0,-1.2) node[below]{$\lollipopgraph{1}$};
\end{scope}
\begin{scope}[xshift=6cm]
\fill[black] (0,-.5) circle (2.5pt);
\fill[black] (0,.5) circle (2.5pt);
\fill[black] (0,-1) circle (2.5pt);
\draw(0,0) circle (.5cm);
\draw(0,-1) -- (0,-.5);
\draw(0,-1.2) node[below]{$\lollipopgraph{2}$};
\end{scope}
\begin{scope}[xshift=8cm]
\fill[black] (0,-.5) circle (2.5pt);
\fill[black] (0,.5) circle (2.5pt);
\draw(0,0) circle (.5cm);
\draw(0,.5) -- (0,-.5);
\draw(0,-1.2) node[below]{$\thetagraph{3}$};
\end{scope}
\begin{scope}[xshift=10cm]
\fill[black] (0,-.5) circle (2.5pt);
\fill[black] (0,.5) circle (2.5pt);
\draw(0,0) circle (.5cm);
\draw(0,0) ellipse (.175cm and .5cm);
\draw(0,-1.2) node[below]{$\thetagraph{4}$};
\end{scope}
\end{tikzpicture}
\caption{Cycle graphs, lollipop graphs, and theta graphs}
\end{figure}

The reduced \'{S}wi\k{a}tkowski complex of the star graph $\fullyreduced{\stargraph{n}}$ is concentrated in degrees $0$ and $1$ with a single differential 
\[\mathbb{Z}[E]\langle h_{12},\ldots, h_{1n}\rangle\xrightarrow{\partial}\mathbb{Z}[E]\langle \varnothing\rangle.
\] 

\begin{lemma}\label{lemma:three star}
\begin{enumerate}
\item The homology group $H_1(B(\stargraph{3}))$ is freely generated as a $\mathbb{Z}[E]$-module by a single class $\alpha$ in $H_1(B_2(\stargraph{3}))$.
\item The homology group $H_1(B(\cyclegraph{n}))$  is generated as a $\mathbb{Z}[E]$-module by a single class $\gamma$ in $H_1(B_1(\cyclegraph{n}))$ subject to the relations $e_i\gamma=e_j\gamma$. 
\end{enumerate}
\end{lemma}
\begin{proof}

For $\stargraph{3}$, it is a straightforward to see that $H_0(B(\stargraph{3}))$ is rank one in each weight. By observation, the chain $a\coloneqq{}a_{123}=e_1h_{23} + e_2h_{31} + e_3h_{12}$ is closed. Checking the Euler characteristic shows that $\mathbb{Z}[E]\langle a\rangle$ has the correct rank and thus is the entirety of the kernel of $\partial$.

For $\cyclegraph{1}$, the complex $\fullyreduced{\cyclegraph{1}}$ has no differential and its degree one subspace is isomorphic to the module described. There is a (non-unique) smoothing from $\cyclegraph{n}$ to $\cyclegraph{1}$; the induced map on \'{S}wi\k{a}tkowski complexes identifies every edge of $\cyclegraph{n}$ with the unique edge of $\cyclegraph{1}$.
\end{proof}
The sign of the generating class $\alpha$ in $H_1(B_2(\stargraph{3}))$ depends on the choice of ordering of the half-edges. We will use the notation $\alpha_{123}=\alpha_{231}=\alpha_{312}$ for $\alpha$ with the convention employed here and $\alpha_{132}=\alpha_{321}=\alpha_{213}$ for the generator with the opposite sign. 

\begin{definition}
Let $\stargraph{3}\to \graf$ be a graph morphism. We call the image of the class $\alpha$ in $H_1(B_2(\graf))$ a \emph{star class} and a representing cycle of the form above a \emph{star cycle}.

Let $\cyclegraph{n}\to \graf$ be a graph morphism. We call the image of the class $\gamma$ in $H_1(B(\graf))$ a \emph{loop class} and a representing cycle a \emph{loop cycle}.
\end{definition}
A star class depends only on the isotopy class of the graph morphism which induces it. Note also that we have such a morphism whenever there are three distinct half-edges at a vertex; it is not necessary that the corresponding edges be distinct.
\begin{lemma}\label{lemma: star classes are non-trivial}
Star and loop classes are non-trivial.
\end{lemma}
\begin{proof}
There is a natural homomorphism $\sigma_\graf: H_1(B_2(\graf))\to \mathbb{Z}/2\mathbb{Z}$, which is induced on Abelianizations by the homomorphism from the braid group $\pi_1(B_2(\graf))$ recording the permutation of the endpoints of a braid. By naturality, evaluating $\sigma_\graf$ on a star class gives the same answer as evaluating $\sigma_{\stargraph{3}}$ on a generator of $H_1(B_2(\stargraph{3}))$. Since any cycle representing such a generator interchanges the two points of the configuration, $\sigma_\graf$ takes the value $1$ on this generator.

Loop classes live in $H_1(B_1(\graf))$ which is naturally isomorphic to $H_1(\graf)$; a loop in $\graf$ is never trivial.
\end{proof}

Star classes play a pivotal role in the remainder of the paper. For example, we have the following result.

\begin{lemma}\label{lemma: star classes generate stars}
The $\mathbb{Z}[E]$-module $H_1(B_k(\stargraph{n}))$ is generated by star classes.
\end{lemma}
Similar statements go back at least to~\cite{FarleySabalka:DMTGBG,FarleySabalka:PGBG}.
\begin{proof}
We proceed by induction on $n$. Since $\stargraph{n}$ is topologically an interval for $n\in\{1,2\}$, the claim holds trivially in these cases, and the case $n=3$ follows from Lemma~\ref{lemma:three star}, so we may assume that $n\ge 4$.

Any degree $1$ element $a\in \tailreduced{\stargraph{n}}$ can be written in the form $a=\sum_{i=1}^n p_i h_i$ with $p_i\in \mathbb{Z}[E]$, and imposing the condition $\partial a=0$ yields the two equations
\begin{equation}\label{equation:cycles in stargraph}
\partial a=\sum_{i=1}^n p_i(e_i - v)=0\Longleftrightarrow
\begin{cases}
\sum_{i=1}^n p_i=0;\\
\sum_{i=1}^n e_i p_i=0.
\end{cases}
\end{equation} For each $i\le n-2$, we write $p_i=(e_n-e_{n-1})p_i'+r_i$, where $r_i$ does not involve the variable $e_n$. 
Now we can rewrite $a$ partially in terms of the star cycles $a_{i,n-1,n}$ as follows:
\begin{align*}
a=\sum_{i=1}^{n-2} (p'_ia_{i,n-1,n} + r_ih_i) + q_{n-1}h_{n-1}+q_nh_n,
\end{align*}
where
\begin{align*}
q_{n-1}&\coloneqq{}p_{n-1}-\sum_{i=1}^{n-2} (e_i-e_n)p_i',&
q_n&\coloneqq{}p_n-\sum_{i=1}^{n-2}(e_{n-1}-e_i)p_i'.
\end{align*}
Write $q_{n-1}$ and $q_n$ as
\begin{align*}
q_{n-1}&=e_n q_{n-1}' +r_{n-1},&
q_n&=e_n q_n'+r_{n},
\end{align*}
where $r_n$ and $r_{n-1}$ do not involve the variable $e_n$. Then considering terms involving $e_n$ in the equations (\ref{equation:cycles in stargraph}) using the fact that $\partial a_{i,n-1,n}=0$, we have
\[
\begin{cases}
e_n (q_{n-1}'+q_n')=0;\\
e_{n-1}q_{n-1}'+e_n q_n' +r_n =0
\end{cases}
\Longleftrightarrow
(e_n-e_{n-1})q_n' +r_n =0.
\]
Since $r_n$ does not involve the variable $e_n$, it follows that $q_n'=r_n=0$, whence $q_n=0$. We further conclude that $q_{n-1}'=0$, and so $q_{n-1}$ does not involve the variable $e_n$; therefore, $a-\sum_{i=1}^{n-2}a_{i,n-1,n}$ does not involve $e_n$ or $h_n$ and so must lie in the image of the map $\tailreduced{-}$ induced by the inclusion $\stargraph{n-1}\to \stargraph{n}$ that misses the $n$th leg. The inductive hypothesis now completes the proof.
\end{proof}

\begin{proposition}\label{prop:H1 generators}
If $\graf$ is connected, then $H_1(B(\graf))$ is generated as a $\mathbb{Z}[E]$-module by star classes and loop classes.
\end{proposition}
\begin{remark}
Again this is more or less implicit in the work of Farley and Sabalka~\cite{FarleySabalka:DMTGBG,FarleySabalka:PGBG}. We provide a proof using our language and methods.
\end{remark}
\begin{proof}
Assume first that $\graf$ is a tree. In this case the claim follows from Proposition~\ref{proposition: one-bridges}, Lemma~\ref{lemma: star classes generate stars}, and induction on the number of vertices of $\graf$ of valence at least three. 
In the general case, subdivide each edge of $\graf$ into three edges, calling the resulting graph $\graf'$. Each edge of $\graf$ corresponds to three edges of $\graf'$; let $E_{\mathrm{mid}}$ be the set of edges of $\graf'$ none of whose vertices are vertices of $\graf$. 
Let $E_0$ be a subset of $E_{\mathrm{mid}}$ whose complement is a spanning tree $\treegraph$ of $\graf'$. 
For $e\in E_0$, write $\graf^{[e]}$ for the disjoint union of the unique cycle subgraph $\cyclegraph{e}$ contained in $e\cup \treegraph$ with a disjoint edge for each edge of $E_{\mathrm{mid}}\backslash E(\cyclegraph{e})$. There are canonical graph embeddings of $\treegraph$ and $\cyclegraph{e}$ into $\graf'$ which induce graph morphisms from $\treegraph$ and $\cyclegraph{e}$ into $\graf$. Write $\widehat{S}(\graf)\coloneqq \fullyreduced{\treegraph}\oplus\bigoplus_{e\in E_0} \fullyreduced{\graf^{[e]}}$; then we have a map of differential graded $\mathbb{Z}[E]$-modules  \[\phi:\widehat{S}(\graf)\to \widetilde{S}(\graf)\] induced by these graph morphisms. It is clear that $H_1(B(\graf^{[e]}))$ is generated as a $\mathbb{Z}[E]$-module by its unique loop class for each $e\in E_0$, and we have already shown that $H_1(B(\treegraph))$ is generated by star classes; therefore, it will suffice to show that $\phi$ induces a surjection on $H_1$. 

For $e\in E_0$, write $e'$ and $e''$ for the unique pair of edges of $\treegraph$ with $\phi(e)=\phi(e')=\phi(e'')$. At the chain level, $\phi|_{\widetilde{S}_1(\treegraph)}$ is surjective with kernel the $\mathbb{Z}[E]$-span of the set $\{(e'-e''):e\in E_0\}$. Let $c_e\in \fullyreduced{\cyclegraph{e}}$ be a cycle representing a loop class and choose a degree $1$ element $b_e\in \fullyreduced{\treegraph}$ with $\phi(b_e)=\phi(c_e)$ and $\partial b_e = \pm (e'- e'')$. Then $\phi(b_e-c_e)=0$ and $\partial(b_e-c_e)=\pm (e'-e'')\in \fullyreduced{\treegraph}$. 

Now, suppose we are given $b\in \fullyreduced{\treegraph}$ with $\phi(\treegraph)$ closed. Then \[\boundary b = \sum_{e\in E_0}p_e(e'-e''),\] and we set
\[b'\coloneqq b+\sum_{e\in E_0}\mp p_e(b_e-c_e)\in \widehat S(\graf).\]
Then $\phi(b')=\phi(b)$ and $b'$ is closed, as desired.
\end{proof}

Now we turn to relations.

\begin{lemma}
\begin{enumerate}
\item The homology group $H_0(B(\stargraph{2}))$ is generated by the class of the empty configuration $\varnothing$ subject to the relation $e_1\varnothing=e_2\varnothing$.
\item In $H_1(B(\stargraph{4}))$, the star classes satisfy the relation \[e_1\alpha_{234}-e_2\alpha_{341}+e_3\alpha_{412}-e_4\alpha_{123}=0.\]
\item In $H_1(B(\lollipopgraph{n}))$, the loop and star classes satisfy the relation 
\[(e-e_0)\gamma = \alpha\]
where $e$ is an edge of the cycle subgraph.
\item 
Let the edges of $\thetagraph{3}$ be numbered from $1$ to $3$ and likewise for the half-edges at $v_1$ and at $v_2$. Then the star class $\alpha_{123}$ at $v_1$ and the star class $\alpha_{321}$ at $v_2$ are equal.
\end{enumerate}
\end{lemma}
\begin{proof}
\begin{enumerate}
\item The cokernel of the differential is generated freely by $(e_1-e_2)$.
\item The claim follows by expansion of the star cycles $a_{ijk}$.
\item It suffices to verify the claim for $\lollipopgraph{1}$ where it is already true for the (unique) chain level representatives using the reduced \'{S}wi\k{a}tkowski complex.
\item The chain $h_{12}\otimes h'_{13}- h_{13}\otimes h'_{12}$ in $S_2(\thetagraph{3})$ bounds the difference between the corresponding star cycles.

\end{enumerate}
\end{proof}

\begin{definition}
\label{definition:relations}
Let $\graf$ be a graph.
\begin{enumerate}
\item Let $\stargraph{2}$ to $\graf$ be a graph morphism. We call the induced relation in $H_0(B(\graf))$ an \emph{$I$-relation}.
\item Let $\stargraph{4}\to \graf$ be a graph morphism. We call the induced relation on star classes in $H_1(B(\graf))$ an \emph{$X$-relation}.
\item Let $\lollipopgraph{n}\to \graf$ be a graph morphism. We call the induced relation on loop and star classes in $H_1(B(\graf))$ a \emph{$Q$-relation}.
\item Let $\thetagraph{}'\to \thetagraph{3}$ be a smoothing and $\iota:\thetagraph{}'\to \graf$ a graph embedding. Denoting the preimage of $v_i$ in $\thetagraph{}'$ by $v_i'$, we call the relation induced on the star classes at $\iota(v'_1)$ and $\iota(v'_2)$ in $H_1(B(\graf))$ a \emph{$\Theta$-relation}.
\item Let $\cyclegraph{n}\to \graf$ be a graph morphism. We call the induced relation (from Lemma~\ref{lemma:three star}) on loop classes in $H_1(B(\graf))$ an \emph{$O$-relation}.
\end{enumerate}
\end{definition}

Repeated use of the $I$-relation implies the well-known fact that $H_0(B_k(\graf))$ is one-dimensional for $\graf$ connected.

These atomic classes combine naturally.
\begin{definition}
Consider a graph $\graf_0$ written as the disjoint union of $n_1$ cycle graphs and $n_2$ copies of $\stargraph{3}$. Since the configuration spaces of the components all have torsion-free homology, there is a class $\beta$ in $H_{n_1+n_2}(B_{n_1+2n_2}(\graf_0))$ corresponding to the tensor product of the loop classes in each cycle graph and the star classes in each star graph. If $\graf_0\to \graf$ is a graph morphism, we call the image of $\beta$ in $H_{n_1+n_2}(B_{n_1+2n_2}(\graf))$ the \emph{external product} of the corresponding loop and star classes. We will use juxtaposition to indicate the external product---writing $\alpha_1\alpha_2$ for the external product of classes $\alpha_1$ and $\alpha_2$---and we caution the reader that this construction is only partially defined and so does not define a product on homology.
\end{definition}

\subsection{Euler characteristic}\label{section:euler characteristic}
One calculation that requires no further tools is that of the Euler characteristic. This result has been known at least since~\cite{Gal:ECCSC}.

\begin{corollary}\label{corollary: Euler characteristic}
The Euler characteristic of $B_k(\graf)$ is given by
\[
\chi(B_k(\graf))=\sum_{U\subseteq V}(-1)^{|U|}\binom{k-|U|+|E|-1}{|E|-1}\prod_{v\in U}(d(v)-1).
\]
\end{corollary}

Defining the Euler--Poincar\'e series of $\graf$ to be the formal power series
\[P_\chi(\graf)(t)=\sum_{k=0}^\infty \chi(B_k(\graf))t^k,\]
we have the following reformulation.

\begin{corollary}
Let $\graf$ be a graph. Then
\[
P_\chi(\graf)(t)=\prod_v \frac{1+t(1-d(v))}{(1-t)^{\frac{d(v)}{2}}}.
\]
\end{corollary}
\begin{proof}
For $n\geq0$, the formula holds for the graph $\graphfont{G}_{2n}$ that is the wedge of $n$ circles. It also holds, for $m\leq n$ non-negative, for the graph $\graphfont{G}_{2m+1,2n+1}$, which is given by connecting the vertices of $\graphfont{G}_{2m}$ and $\graphfont{G}_{2n}$ by a single edge. Since the formula of Corollary \ref{corollary: Euler characteristic} does not depend on which pairs of vertices share an edge, the claim follows.
\end{proof}

\subsection{Spectral sequences and exact sequences}\label{section:exact sequences} 
One particularly nice class of tools afforded by the \'{S}wi\k{a}tkowski complex consists of spectral sequences that arise by decomposing the differential into a bicomplex. In general these facilitate the reduction of computations of $H_*(B(\graf))$ to computations for simpler graphs. Choosing our input data judiciously, we get an exact sequence which arises from deleting a vertex of the graph.
\begin{lemma}\label{lemma: spectral sequence existence}
Let $J$ be a set of half-edges of a graph $\graf$, and let $U$ be the subset of vertices of $\graf$ that half-edges of $J$ are incident on. Then:
\begin{enumerate}
\item  the differential $\boundary$ of the \'{S}wi\k{a}tkowski complex decomposes into the sum of two commuting $\mathbb{Z}[E]$-linear differentials $\boundary_J+(\boundary-\boundary_J)$, where $\boundary_J$ changes the number of generators containing half-edges in $J$, and
\item the (first) spectral sequence associated to this bicomplex collapses at the $(|U|+1)$st page, necessarily to the homology of the \'{S}wi\k{a}tkowski complex.
\end{enumerate}
\end{lemma}
\begin{proof}
For the first statement, it suffices to note, first, that $\boundary$ can be rewritten as the sum $\boundary=\sum \boundary_h$ over all half-edges, and, second,that the operators $\boundary_h$ individually square to zero and pairwise commute. 

For the second statement, since the differential $\boundary_J$ lowers the number of half-edges in $J$ in a homogeneous monomial by $1$ and since there are necessarily between $0$ and $|U|$ such half-edges in any homogeneous monomial, all differentials starting from page $(|U|+1)$ vanish identically.
\end{proof}
When $J$ consists of all half-edges incident on a set of vertices, there is a particularly straightforward description of the $E^1$ page of this spectral sequence, or rather of the corresponding spectral sequence for the reduced \'{S}wi\k{a}tkowski complex. 

\begin{definition}
Let $v$ be a vertex of the graph $\graf$. Then we write $\graf_v$ for the \emph{vertex explosion} of $\graf$ at $v$, that is, the graph obtained by 
\begin{enumerate}
\item replacing the vertex $v$ with $\{v\}\times H(v)$ and 
\item modifying the attaching map for half-edges attached at $v$ by letting such a half-edge $h$ be attached at $v\times h$.
\end{enumerate}
\end{definition}
There is a graph morphism from $\graf_v$ to $\graf$ which takes each edge to itself, takes the vertex $v\times h$ to $e(h)$, and takes each other vertex to itself. Defining this morphism requires choices (of precisely where in $e(h)$ to send $v\times h$) but the isotopy class of this graph morphism is unique. See Figure~\ref{figure: vertex explosion}.
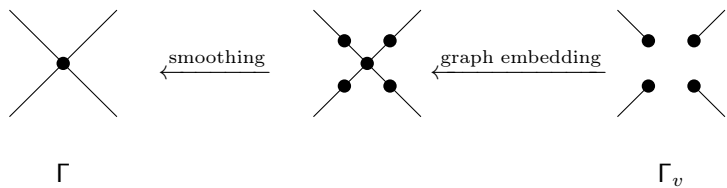
\begin{figure}
\centering
\begin{tikzpicture}
\begin{scope}
\fill[black] (0,0) circle (2.5pt);
\draw(-.707,-.707) -- (.707,.707);
\draw(-.707,.707) -- (.707,-.707);
\draw(0,-1.2) node[below]{$\graf$};
\draw(2,0) node{$\xleftarrow{\text{smoothing}}$};
\draw(6,0) node{$\xleftarrow{\text{graph embedding}}$};
\end{scope}
\begin{scope}[xshift=4cm]
\fill[black] (0,0) circle (2.5pt);
\fill[black] (-.3,-.3) circle (2.5pt);
\fill[black] (-.3,.3) circle (2.5pt);
\fill[black] (.3,-.3) circle (2.5pt);
\fill[black] (.3,.3) circle (2.5pt);
\draw(-.707,-.707) -- (.707,.707);
\draw(-.707,.707) -- (.707,-.707);
\end{scope}
\begin{scope}[xshift=8cm]
\fill[black] (-.3,-.3) circle (2.5pt);
\fill[black] (-.3,.3) circle (2.5pt);
\fill[black] (.3,-.3) circle (2.5pt);
\fill[black] (.3,.3) circle (2.5pt);
\draw(-.707,-.707) -- (-.3,-.3);
\draw(.707,.707) -- (.3,.3);
\draw(-.707,.707) -- (-.3,.3);
\draw(.707,-.707) -- (.3,-.3);
\draw(0,-1.2) node[below]{$\graf_v$};
\end{scope}
\end{tikzpicture}
\caption{A local picture of vertex explosion along with an intermediate graph which admits a graph morphism from $\graf_v$ and a smoothing to $\graf$.}
\label{figure: vertex explosion}
\end{figure}
More generally, if $U$ is a set of vertices, then the isomorphism class of vertex explosion $\graf_U$ at all vertices of $U$ (by sequentially exploding the vertices of $U$) is well-defined and independent of the choice of order of explosion.
\begin{lemma}\label{lemma:SS reduced}
Suppose $J$ is the set of all half-edges incident on a set $U$ of vertices of a graph $\graf$. Then the spectral sequence of Lemma~\ref{lemma: spectral sequence existence} is naturally defined for the reduced \'Swi\k{a}tkowski complex. In weight $k$, the entry $E^1_{p,q}$ in this spectral sequence for the reduced complex is (unnaturally) isomorphic to
\[
E^1_{p,q}\cong \bigoplus_{v_1,\ldots, v_q\in U}\bigoplus^{\prod (d(v_i)-1)}H_{p}(B_{k-q}(\graf_U)).
\]
with differential as described in the proof.

\end{lemma}
\begin{proof}
The condition on $J$ implies that we can understand the induced filtration on the reduced subcomplex by counting generators that are differences of half-edges incident on vertices in $U$, so that the bicomplex structure is natural in the reduced setting.

Moreover, the entry $E^0_{p,q}$ is isomorphic to the sum of copies of the reduced \'Swi\k{a}tkowski complex of $\graf_U$ indexed by the choice of $q$ half-edge differences at distinct vertices in $U$. The differential from $E^0_{p,q}$ to $E^0_{p-1,q}$ is the sum of indexed copies of the reduced \'Swi\k{a}tkowski differential of $\graf_U$, yielding the indicated homology groups.

By fixing a half-edge $h_0^v$ of each $v\in U$, we can index this double sum by sequences of $q$ half-edges incident on distinct vertices of $U$, none of them $h_0^v$.
Then the differential $d^1:E^1_{p,q}\to E^1_{p,{q-1}}$ takes the signed sum over deleting a half-edge $h$ from the sequence and multiplying the homology class by $e(h)-e(h_0^{v(h)})$:
\[
d^1(\alpha_{(h_1,\ldots,h_q)})=\sum_{i=1}^q (-1)^{i-1}\left(e(h_i)-e(h_0^{v(h_i)})\right)\alpha_{(h_1,\ldots,\widehat{h_i},\ldots,h_q)}.
\]
\end{proof}
In the computations of the rest of this paper, we will only use the simplest case of this spectral sequence, which degenerates at page $E^2$.

\begin{corollary}\label{cor: les vertex reduced}
Fix a half-edge $h_0\in H(v)$. There is a long exact sequence of differential bigraded $\mathbb{Z}[E]$-modules
\[
\begin{tikzcd}[column sep=2.9mm, row sep=2mm]
\cdots\rar
&
H_{n+1}(B_k(\graf_v))\rar
&
H_{n+1}(B_k(\graf))\rar
&
\displaystyle\bigoplus_{h\in H(v)\backslash\{h_0\}}
H_n(B_{k-1}(\graf_v))\rar
&\ 
\\
\rar
&
H_{n}(B_k(\graf_v))\rar
&
H_{n}(B_k(\graf))\rar
&
\displaystyle\bigoplus_{h\in H(v)\backslash\{h_0\}}
H_{n-1}(B_{k-1}(\graf_v))\rar
&\cdots
%
\end{tikzcd}
\]
The connecting homomorphism $\delta$ from $\bigoplus H_*(B(\graf_v))\to H_*(B(\graf_v))\{-1\}$ is explicitly given by the formula
\[
\delta \beta_h = (e(h)-e(h_0))\beta_h.\]
\end{corollary}
\begin{proof}
The spectral sequence of Lemma~\ref{lemma:SS reduced} in the case $U=\{v(h_0)\}$ collapses at $E^2$, yielding a long exact sequence with connecting homomorphisms the $E^1$ differentials, which are described in the proof of the lemma.
\end{proof}

We now use this spectral sequence to recover a homologico-algebraic version of our decomposition theorem (Theorem~\ref{thm:decomposition}) for the following special type of decomposition.

\begin{definition}
Let $\graf$ be a graph and $k$ a non-negative integer. 
A \emph{$k$-bridge decomposition} $\graf=\graf_1\sqcup_L \graf_2$ consists of a set $L$ of $k$ distinct edges of $\graf$, called the \emph{$k$-bridge}, and two subgraphs $\widetilde\graf_1$ and $\widetilde\graf_2$ of $\graf$ such that \[\graf=\widetilde\graf_1\cup_{V_1} L\cup_{V_2}\widetilde\graf_2,\] where $V_1\cup V_2$ is a set of $2k$ distinct vertices. In this case, we write $\graf_j=\widetilde\graf_j\cup L$.
\end{definition}

\begin{proposition}\label{prop:spectral sequence}
Let $\graf= \graf_1\sqcup_L \graf_2$ be a graph with a $k$-bridge decomposition and $\mathbb{K}$ a field.
There is a spectral sequence converging to $H(B(\graf),\mathbb{K})$ with $E^2$ page given by $\tor^{\mathbb{K}[L]}(H(B(\graf_1),\mathbb{K}),H(B(\graf_2),\mathbb{K}))$.
For $k$ at least $2$, this spectral sequence collapses at page $k$ (for $k=1$ it collapses at page $2$). 
\end{proposition}

\begin{proof}
Subdivide each edge of the bridge and let $J$ be the set of $k$ new vertices. 
The spectral sequence of Lemma~\ref{lemma:SS reduced} for this vertex set $J$ has a special property. 
Because each vertex is valence $2$, the $E^0$ page has $E^1_{p,q}=\wedge^q H_p(B_{k-q}(\graf_1\sqcup \graf_2))$. 
After passing to $\mathbb{K}$ coefficients and invoking the K\"{u}nneth theorem, examination of the $d^1$ differential reveals that it is the Koszul differential which calculates the $\tor$ groups of the statement of the proposition. 
Moreover, $d^1$ from $E^1_{p,k}$ to $E^1_{p,k-1}$ is injective by Proposition~\ref{proposition: edge multiplication is injective}, which means that $E^2$ is concentrated in rows $q=0$ to $q=k-1$.
\end{proof}

\begin{corollary}
Let $\graf= \graf_1\sqcup_L \graf_2$ be a graph equipped with a $2$-bridge decomposition. 
Then the associated graded of $H(B(\graf))$ is isomorphic to $\tor^{\mathbb{K}[L]}(H(B(\graf_1)), H(B(\graf_2)))$.
\end{corollary}

We defer consideration of the simpler situation of a $1$-bridge decomposition to Proposition \ref{proposition: one-bridges}, which is an integral statement.

\begin{remark}
Under the quasi-isomorphism supplied by the decomposition theorem, the spectral sequence of Proposition \ref{prop:spectral sequence} is identified with the K\"{u}nneth spectral sequence.
\end{remark}

\subsection{Edge injectivity}\label{section:faithfulness}
We will make repeated use of the following result as a technical tool, but it is interesting in itself, and we think of it as a first step in an investigation of the $\mathbb{Z}[E]$-module structure enjoyed by the homology of the configuration spaces of $\graf$, in the tradition of the study of homological stability. This study is continued in detail in the sequel \cite{AnDrummond-ColeKnudsen:ESHGBG}---see also \cite{Ramos:SPHTBG}.

\begin{proposition}\label{proposition: edge multiplication is injective}
For any $e\in E$, multiplication by $e$ is injective on $H_*(B(\graf))$.
\end{proposition}
\begin{proof}
Any monomial $p(E,V)$ in the edges and vertices of $\graf$ can be rewritten as a polynomial in variables $e$, $e-e'$, and $e-v$. Write $U$ for the collection of generators other than $e$, which is in bijection with $E\backslash\{e\}\amalg V$. Then an arbitrary chain in $\intrinsic{\graf}$ can be written in the form
\[
b = \sum_{i=0}^M e^i b_i(U,H)
\]
for some (graded commutative) polynomials $b_i$ (not all polynomials are possible).

Let $c$ be a cycle in $\intrinsic{\graf}_k$, and suppose that $ec=\partial b$. The proof will be complete once we are assured that $c$ is itself a boundary. To see this, we note that the differential, acting on $b_i(U,H)$, cannot introduce a factor of $e$ in this basis; that is,
\[
\partial b = \sum_{i=0}^M e^i b'_i(U,H),
 \] 
where $b'_i=\partial b_i$. Then we have
\[
\sum_{i=1}^M e^i {c}_{i-1}(U,H) = \sum_{i=0}^M e^i b'_i(U,H).
\]
Thus, $\boundary b_0 = b'_0=0$, so $ec = \boundary b = \boundary (b-b_0)$. Since $b-b_0$ is divisible by $e$, we have $c=\partial(\frac{b-b_0}{e})$, as desired.
\end{proof}

With Proposition~\ref{proposition: edge multiplication is injective} in hand, we can describe more precisely the effect on $H_*(B(\graf))$ of cutting $\graf$ into two disconnected components at a vertex.

\begin{proposition}\label{proposition: one-bridges}
Let $\graf$ be a connected graph and $v$ a bivalent vertex whose removal disconnects $\graf$. Write $e_1$ and $e_2$ for the edges at $v$. There is an isomorphism \[H_*(B(\graf))\cong H_*(B(\graf_v))/e_1\sim e_2\] of $\mathbb{Z}[E]$-modules. In particular, for any field $\mathbb{K}$, we have \[H_*(B(\graf),\mathbb{K})\cong H_*(B(\graf'_v),\mathbb{K})\bigotimes_{\mathbb{K}[e]} H_*(B(\graf''_v),\mathbb{K}),\] where $\graf'_v$ and $\graf''_v$ are the connected components of $\graf_v$. 
\end{proposition}
\begin{proof}
Applying Corollary~\ref{cor: les vertex reduced} at the vertex $v$, we obtain the exact sequence
\[
0\to \fullyreduced{\graf_v}\xrightarrow{\phi} \fullyreduced{\graf}\xrightarrow{\psi} \fullyreduced{\graf_v}[1]\{1\}\to 0.
\] 
We claim that the connecting homomorphism in the corresponding long exact sequence is injective. Indeed, suppose that $\delta\beta=0$ for some $\beta\in H_i(B_k(\graf_v))[1]\{1\}$; then, applying the formula of Corollary~\ref{cor: les vertex reduced}, we have $(e_1-e_2)\beta=0$. Since $\graf_v$ is disconnected, the homology of $B_k(\graf_v)$ is naturally graded by the number of points lying in $\graf_v'$; therefore, writing $\beta=\sum_{j=0}^k\beta_j$, we obtain the system of equations \begin{align*}
e_1\beta_k&=e_2\beta_0=0\\
e_1\beta_j&=e_2\beta_{j+1}\qquad (0\leq j<k).
\end{align*} Applying Proposition~\ref{proposition: edge multiplication is injective} repeatedly now implies that $\beta_j=0$ for every $0\leq j\leq k$. 

The statement over $\mathbb{Z}$ now follows by exactness and the formula for $\delta$, and the statement over $\mathbb{K}$.
\end{proof}

\begin{remark}
This situation has been considered before~\cite[Sections III and VI]{MaciazekSawicki:HGP1CG}.
There, assuming their Conjecture VI.1, they prove the special case of Proposition~\ref{proposition: one-bridges} for homological degree two.
To be more precise, their Theorem VI.2 provides a count of the rank of the homology group $H_2(B(\graf))$ in terms of $H_*(B(\graf_v'))$ and $H_*(B(\graf_v''))$. Proposition~\ref{proposition: edge multiplication is injective} implies that $H_*(B(\graf_v'))$ and $H_*(B(\graf_v''))$ are free $\mathbb{K}[e]$-modules, and their formula for the rank comes directly from the formula for the weighted rank of the tensor product of free $\mathbb{K}[e]$-modules (a side remark: there seems to be a typo in the range of the summation of their equation~(24), which omits $\beta_1^{(0)}(\Gamma_1)\beta_1^{(n)}(\Gamma_2)$). 
It seems likely that the injectivity of our connecting homomorphisms used in the proof of Proposition~\ref{proposition: one-bridges} implies their Conjecture VI.1 (justifying the rest of their section VI), but their iterated Mayer--Vietoris decompositions differ enough from the ``one-stage'' decomposition of our vertex long exact sequence that it might take some work to translate between the two. 
Since Proposition~\ref{proposition: one-bridges} supercedes their Theorem VI.2, we have chosen not to further pursue such verification.
\end{remark}

\subsection{High degree homology for unitrivalent graphs}\label{section:high degree trivalent} Call a graph \emph{unitrivalent} if every vertex has valence either $3$ or $1$. 
In this section we apply the exact sequence of Corollary~\ref{cor: les vertex reduced} to study $H_i(B(\graf))$ when $\graf$ is unitrivalent and $i$ is close to the number of trivalent vertices of $\graf$. According to our conventions, a graph may have self-loops (at most one self-loop per vertex by unitrivalency), and multiple edges connecting a given pair of vertices. In this subsection, unless otherwise specified, $\graf$ is a unitrivalent graph and $N$ denotes the number of trivalent vertices of $\graf$.

\begin{construction}
Suppose $\graf$ is unitrivalent and has $r$ self-loops. To each trivalent vertex $v$ we associate a homology class $\beta_v\in H_1(B(\graf))$, well-defined up to sign: if $v$ has a self-loop, then $\beta_v$ is the corresponding loop class; otherwise, $\beta_v$ is the star class at $v$. Taking the external product over the trivalent vertices of the $\beta_v$, we obtain a homology class $\beta
\in H_{N}(B_{2N-r}(\graf))$, well-defined up to sign, called the \emph{canonical class} of $\graf$.
\end{construction}

\begin{proposition}\label{proposition: top homology}
For $\graf$ unitrivalent, we have an isomorphism \[H_{N}(B(\graf))\cong \mathbb{Z}[E]\beta.\]
\end{proposition}
\begin{proof}
Without loss of generality, we may assume $\graf$ is connected. We proceed by induction on $N$ with two base cases. 
The first base case is the $3$-star, which was dealt with in \S\ref{subsection:examples}.
The second base case is the lollipop $\lollipopgraph{1}$. 
The reduced \'{S}wi\k{a}tkowski complex $\fullyreduced{\lollipopgraph{1}}$ is concentrated in degrees $0$ and $1$, with the explicit form
\[
\fullyreduced{\lollipopgraph{1}}\cong \mathbb{Z}[e_0,e]\langle \varnothing, h_{01},h_{02}\rangle
\]
with differential $\partial(h_{01})=\partial(h_{02})=e_0-e$.
By inspection, the kernel of the differential in degree $1$ is spanned by $\mathbb{Z}[e_0,e](h_{01}-h_{02})$.

Now assume that $N\geq2$, fix a trivalent vertex $v$, and assume that the claim is known for $\graf_v$. 
Applying Corollary~\ref{cor: les vertex reduced} at $v$, we obtain the exact sequence
 \[
\cdots\to H_{N}(\fullyreduced{\graf_v}) \to H_{N}(\fullyreduced{\graf})\to \bigoplus_{\mathclap{h\in H(v)\setminus \{h_0\}}}H_{N-1}(\fullyreduced{\graf_v})\{1\}\xrightarrow{\delta}\cdots
\]Since $\graf_v$ has $N-1$ trivalent vertices, the first term vanishes, so the homology group of interest is the kernel of $\delta$. The inductive hypothesis identifies the domain of $\delta$ as
\[\left(\bigoplus_{h\in H(v)\setminus\{h_0\}}\mathbb{Z}[E(v)]\right)\otimes \mathbb{Z}[E\backslash E(v)]\beta',\]
where $\beta'$ is the canonical class of $\graf_v$ and $\delta$ acts on the first factor. Denoting by $\graf^{(v)}$ the subgraph containing $v$ and all of its edges (either a star or a lollipop), this first factor is naturally identified with the degree $1$ component of $\fullyreduced{\graf^{(v)}}$. Under this identification, it is immediate from the formula of Corollary~\ref{cor: les vertex reduced} that the condition of lying in the kernel of $\delta$ is precisely the condition of lying in the kernel of $\partial$ from degree one to degree zero of $\fullyreduced{\graf^{(v)}}$. Since $\graf^{(v)}$ has only one vertex, \[\ker \partial\cong H_1(B(\graf^{(v)}))\cong \mathbb{Z}[E(v)]\beta_v,\] and the claim follows.
\end{proof}

\begin{remark}
For a general graph (not necessarily unitrivalent), it is straightforward to show that $H_{N}(B_k(\graf))=0$ for $k<N$. Similar methods to those used in Proposition~\ref{proposition: top homology} show, for instance, that if $\graf$ has no bivalent vertices or self-loops and at least one vertex, that $H_N(B_N(\graf))=0$ as well. Any further sharpening of this result must take into account, e.g., that $H_2(B_3(\thetagraph{4}))$ is non-zero.
\end{remark}

Similar inductive arguments using Corollary~\ref{cor: les vertex reduced} may be used to demonstrate the following two results. Recall that a graph is simple if it has no self-loops and no multiple edges.

\begin{proposition} \label{proposition: codimension one}
Suppose $\graf$ is simple and unitrivalent. There is an isomorphism
\[
H_{N-1}(B(\graf))\cong \bigoplus_{d(v)=3}\mathbb{Z}[E]\hat{\beta}_v/ (e\hat{\beta}_v - e'\hat{\beta}_v: e,e'\in E(v))
\]
of $\mathbb{Z}[E]$-modules, where $\hat{\beta}_v\in H_{N-1}(B_{2N-2}(\graf))$ is the external product of $\beta_w$ for $w\ne v$.
\end{proposition}
\begin{proposition} \label{proposition: codimension two}
If $\graf$ is simple and unitrivalent, then $H_{N-2}(B(\graf))$ is torsion-free. 
\end{proposition}
\begin{proof}[Sketch of proofs of Proposition~\ref{proposition: codimension one} and~\ref{proposition: codimension two}]
In both cases, applying Corollary~\ref{cor: les vertex reduced} at a fixed trivalent vertex $v_0$, the desired homology group $A$ fits into a short exact sequence 
\[
0\to \coker\delta_M \to A\to \ker \delta_{M-1}\to 0
\]
for some $M$. 

In the case of Proposition~\ref{proposition: codimension one}, we can use Proposition~\ref{proposition: top homology} and the explicit formula for the connecting homomorphism to show that the cokernel entry of the short exact sequence consists of the terms indexed by vertices other than $v_0$. Then by induction and an explicit calculation, the kernel entry splits and yields the term indexed by $v_0$.

In the case of Proposition~\ref{proposition: codimension two}, the kernel entry is torsion-free by induction so any torsion would have to come from the cokernel term. By explicitly examining two cases using the formula for the connecting homomorphism and Proposition~\ref{proposition: codimension one} (the cases corresponding to whether the vertex indexing the summand in Proposition~\ref{proposition: codimension one} is adjacent to $v_0$ or not), we conclude that the cokernel term is itself torsion-free.
\end{proof}

\subsection{Case study: the complete graph \texorpdfstring{$\mathsf{K}_4$}{K4}}\label{subsection:complete graph}

We apply our results to give a complete description of $H_*(B(\completegraph{4}))$ as a $\mathbb{Z}[E]$-module, where $\completegraph{4}$ is the complete graph on four vertices. As an intermediary result, we also give the corresponding computation for the net graph $\netgraph$.

\begin{figure}[ht]
\centering
\begin{tikzpicture}
\fill[black] (0,0) circle (2.5pt);
\fill[black] (0,1) circle (2.5pt);
\fill[black] (.866,-.5) circle (2.5pt);
\fill[black] (-.866,-.5) circle (2.5pt);
\draw(0,0) -- (0,1) -- (.866,-.5) -- (-.866,-.5) -- (0,1);
\draw(0,0) --(.866,-.5);
\draw(0,0) --(-.866,-.5);
\draw(0,-.6) node[below]{$\completegraph{4}$};
\begin{scope}[xshift=5cm]
\fill[black] (0,1) circle (2.5pt);
\fill[black] (.866,-.5) circle (2.5pt);
\fill[black] (-.866,-.5) circle (2.5pt);
\fill[black] (0,2) circle (2.5pt);
\fill[black] (1.732,-1) circle (2.5pt);
\fill[black] (-1.732,-1) circle (2.5pt);
\draw(-1.732,-1)--(-.866,-.5)--(.866,-.5)--(1.732,-1);
\draw(-.866,-.5)--(0,1)--(.866,-.5);
\draw(0,1)--(0,2);
\draw(0,-.6) node[below]{$\netgraph$};
\end{scope}
\end{tikzpicture}
\caption{The complete graph $\completegraph{4}$ and the net graph $\netgraph$}
\label{figure: k4 netgraph}
\end{figure}
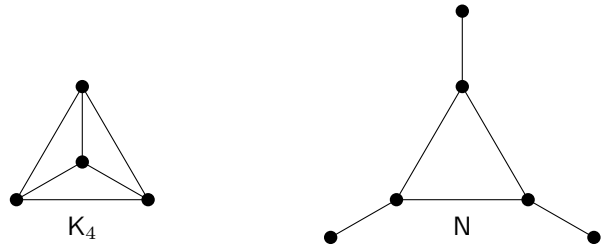

\begin{notation} We number the vertices of $\completegraph{4}$ and write $e_{ab}$ for the edge connecting vertices $a$ and $b$; thus, $e_{ab}=e_{ba}$ and $e_{aa}$ is undefined. Write $h^{(a)}_{b}$ for the half-edge at $a$ with edge $e_{ab}$. The action of the symmetric group $\Sigma_4$ on the vertices extends to an action on $S(\completegraph{4})$ by bigraded chain maps intertwining the $\mathbb{Z}[E]$-action; indeed, a choice of parametrization defines an action of $\Sigma_4$ on $\completegraph{4}$ by graph isomorphisms, and the \'{S}wi\k{a}tkowski complex is functorial. We write $\gamma_a$ for the loop class avoiding vertex $a$ and $\alpha_a$ for the star class at vertex $a$, with signs fixed as follows:
\begin{itemize}
\item $\gamma_4=[h^{(1)}_{2} - h^{(1)}_3+h^{(2)}_3-h^{(2)}_1+h^{(3)}_{1}-h^{(3)}_2]$,
\item $\alpha_4=\alpha^{(4)}_{123}=[(h^{(4)}_1-h^{(4)}_2)e_{34}+(h^{(4)}_3-h^{(4)}_1)e_{24}+(h^{(4)}_2-h^{(4)}_3)e_{14}]$, 
\item $\gamma_a=(-1)^{\mathrm{sgn}(\sigma)}\sigma(\gamma_4)$, and similarly for $\alpha_a$, where $\sigma$ is the cyclic permutation taking $4$ to $a$.
\end{itemize}
We extend this notation to the net graph $\netgraph$ by naming all $1$-valent vertices $4$; this is not ambiguous because we will never refer to the half-edges at these vertices and because every vertex is adjacent to at most one such vertex.
\end{notation}

\begin{proposition}\label{prop:complete graph} The homology $H_*(B(\completegraph{4}))$ is free Abelian and presented as a $\mathbb{Z}[E]$-module in terms of generators and relations as follows. Let $a,b,c,d\in\{1,2,3,4\}$ be distinct.
\begin{itemize}
\item $H_0(B(\completegraph{4}))$ is generated by $\varnothing$ subject to relations identifying all edges.
\item $H_1(B(\completegraph{4}))$ is generated by $\gamma_a$ and $\alpha_a$ subject to the following relations: 
\renewcommand{\labelitemii}{$\circ$}
\begin{align*}
\textstyle\sum_{a=1}^4\gamma_a&=0&
\gamma_a (e_{bc}- e_{cd})&=0\\
\alpha_a&=\alpha_b&
\gamma_a (e_{ab}- e_{ac})&=0\\
\alpha_a (e-e')&=0&
\gamma_a (e_{bc}-e_{ab})&= \alpha_b
\end{align*}
\item $H_2(B(\completegraph{4}))$ is generated by $\gamma_a\alpha_a$ and $\alpha_a\alpha_b$ subject to the following relations:
\begin{align*}
\textstyle \sum_{a=1}^4 \gamma_a\alpha_a &=0& \gamma_a\alpha_a (e_{bc}-e_{cd})&=0\\
\alpha_a\alpha_b(e-e')&=0\text{ if }e,e'\neq e_{ab}&\gamma_a\alpha_a (e_{bc}-e_{ab})&=\alpha_b\alpha_a
\end{align*}
\item $H_3(B(\completegraph{4}))$ is generated by $\alpha_a\alpha_b\alpha_c$ subject to the relations $\alpha_a\alpha_b\alpha_c (e_{ad}-e_{bd})=0$.
\item $H_4(B(\completegraph{4}))$ is freely generated by $\alpha_1\alpha_2\alpha_3\alpha_4$. \end{itemize}
\end{proposition}

We will prove this result by relating $\completegraph{4}$ to the net graph $\netgraph$, which is obtained by exploding a vertex. The corresponding result for $\netgraph$ is the following.

\begin{lemma}\label{lem:net graph}
The homology $H_*(B(\netgraph))$ is free Abelian and presented as a $\mathbb{Z}[E]$-module in terms of generators and relations as follows. Let $a,b,c\in\{1,2,3\}$ be distinct.
\begin{itemize}
\item $H_0(B(\netgraph))$ is generated by $\varnothing$ subject to relations identifying all edges.
\item $H_1(B(\netgraph))$ is generated by $\gamma_4$ and $\alpha_a$ subject to the following relations: 
\begin{align*}
\alpha_a (e-e')&=0\text{ if }e,e'\ne e_{a4}&
\gamma_4 (e_{ab}-e_{ac})&=0\\
&&\gamma_4 (e_{ab}-e_{a4})&= \alpha_a
\end{align*}
\item $H_2(B(\netgraph))$ is generated by $\alpha_a\alpha_b$ subject to the relation
 $\alpha_a\alpha_b(e-e')=0$ if $e,e'\notin\{e_{a4},e_{b4},e_{ab}\}$
\item $H_3(B(\completegraph{4}))$ is freely generated by $\alpha_1\alpha_2\alpha_3$.
\end{itemize}
\end{lemma}
\begin{proof}
This follows from Propositions~\ref{proposition: top homology},~\ref{proposition: codimension one}, and~\ref{prop:H1 generators}, along with the $O$, $Q$, and $I$-relations and a rank-counting argument (using Corollary~\ref{corollary: Euler characteristic}, say).
\end{proof}

\begin{proof}[Proof of Proposition~\ref{prop:complete graph}]
The statement for $H_0$ holds for any connected graph, the statement for $H_3$ follows from Proposition~\ref{proposition: codimension one}, and the statement for $H_4$ follows from Proposition~\ref{proposition: top homology}. For $H_1$, generation follows from Proposition~\ref{prop:H1 generators} (along with the identification of loop classes containing four edges as the sum of two loop classes of the given form). The relation involving only loop classes follows from a relation in $H_1(\completegraph{4})$ itself. The other stated relations follow directly from the $\Theta$, $I$, $Q$, and $O$-relations. In fact, $H_1(B_n(\graf))$ is completely known in general~\cite{KoPark:CGBG}. The relevant part of that computation, which we reprove as Lemma~\ref{lemma: triconnected}, shows by inspection of the relations that our presentation is complete. 
Thus, it remains to prove the statement for $H_2$. We may assume that $k\ge 2$, since $H_2$ vanishes in lower weight.

The outline of the proof is as follows. Using the long exact sequence of Corollary~\ref{cor: les vertex reduced}, we obtain an explicit $\mathbb{Z}$-linear description of $H_2$. Then, denoting by $M$ the $\mathbb{Z}[E]$-module presented by generators and relations in the statement of the proposition, we exhibit a $\mathbb{Z}[E]$-linear map from $M$ to $H_2$, which our integral description shows to be surjective. Finally, we verify that the two modules have the same integral rank in each weight.

The portion of the long exact sequence relevant to our purpose is
\begin{align*}
\cdots\to H_2(B_{k-1}(\netgraph))^{\oplus 2}&\xrightarrow{\delta_2} H_2(B_k(\netgraph))\to H_2(B_k(\completegraph{4}))\xrightarrow{\psi}
\\ H_1(B_{k-1}(\netgraph))^{\oplus 2}&\xrightarrow{\delta_1} H_1(B_k(\netgraph))\to \cdots
\end{align*}
The kernel of $\delta_1$ is free over $\mathbb{Z}$, so this portion of the sequence splits integrally, and we have $H_2(B_k(\completegraph{4}))\cong \coker\delta_2\oplus \ker\delta_1$ as $\mathbb{Z}$-modules. Passing to the cokernel of $\delta_2$ merely identifies $e_{a4}$ and $e_{b4}$ in $H_2(B_k(\netgraph))$ and thus contributes the space of external products $\alpha_a\alpha_b p(E)$ for $a,b\in\{1,2,3\}$, subject to the final relation $\alpha_a\alpha_b(e-e')=0$ if $e,e'\neq e_{ab}$. On the other hand, $\ker \delta_1$ consists of the isomorphic image of the classes $e_{a4}^me_{12}^{k-m-3}\gamma_a\alpha_a\in H_2(B_k(\completegraph{4}))$ for $a\in\{1,2,3\}$ and $0\le m\le k-3$. Verifying this fact is a tedious calculation. 

As for relations, the second, third, and fourth follow from the $O$-relation, the $I$-relation, and the $Q$-relation, respectively. For the first relation, counting ranks shows that $\sum_{a=1}^4 r_a\gamma_a\alpha_a=0$ with at least one $r_a$ non-zero, and symmetry under cyclic permutations implies that $r_a$ is constant in $a$. The absence of torsion in $H_2(B(\completegraph{4}))$ forces $r_a$ to be a unit.

For the rank calculation, we note that the $\coker \delta_2$ is the direct sum of three copies of the space of degree $k-4$ polynomials in two variables, which has integral rank $3(k-3)$, while $\ker\delta_1$ is the direct sum of three copies of the space of degree $k-3$ polynomials in two variables, which has integral rank $3(k-2)$. It follows that the integral rank of $H_2(B_k(\completegraph{4}))$ is $6k-15$, which is also a lower bound on the rank of $M$ by surjectivity. Thus, it remains to show that the integral rank of $M$ is at most $6k-15$. To see why this is the case, we note that the relations shown imply that $M$ is spanned integrally in weight $k$ by the classes $e_{ab}^{k-4-m}e^m\alpha_a\alpha_b$, where $1\le a,b\le 4$, $0\leq m\leq k-4$, and $e\neq e_{ab}$ is fixed, together with the classes $e_{bc}^{k-3}\gamma_a\alpha_a$, where $a,b,c\neq 4$. There are $6(k-3)$ generators of the former type and $3$ of the latter type, so the rank is at most $6k-15$, as claimed.
\end{proof}

\appendix

\section{Reminders on homotopy colimits}\label{section:hocolim appendix}
In this appendix, we present a summary of some relevant facts and definitions concerning homotopy colimits and related matters. For the general theory and full details,~\cite{Dugger:PHC},~\cite{GoerssJardine:SHT} and~\cite{Hirschhorn:MCL} are good references.

The objects of the simplicial indexing category $\Delta$, being finite ordered sets, may naturally be regarded as categories, and the arrows of $\Delta$ determine functors among these categories.

\begin{definition}
Let $\D$ be a category. The \emph{nerve} of $\D$ is the simplicial set \[N\D_\bullet\coloneqq{}\Fun(\Delta\!^\bullet, \D).\] 
\end{definition}
Aa $p$-simplex $\sigma\in N\D_p$ is a string $\sigma(p)\to \cdots\to \sigma(0)$ of composable morpisms in $\D$. We say that $\D$ is \emph{contractible} if the geometric realization $|N\D|$ is so.

\begin{example}\label{example: cofiltered}
A filtered or cofiltered category is contractible provided it is not empty.
In particular, any category with an initial or terminal object is contractible.
\end{example}

\begin{definition}\label{def:homotopy colimit}
Let $F:\D\to \Top$ be a functor. The \emph{homotopy colimit} of $F$, denoted $\hocolim_\D F$, is the geometric realization of the simplicial space given in simplicial degree $p$ by \[\coprod_{\sigma\in N\D_p}F(\sigma(p))\] with face and degeneracy maps induced by those of $N\D$.\footnote{Lemma~\ref{lemma: hocolim and weak equivalence} implies that the homotopy colimit induces a functor from the \emph{homotopy category} of functors from $\D$ to $\Top$ to the homotopy category of $\Top$. It is typical to either define the homotopy colimit as this latter functor or to define \emph{a} homotopy colimit as any functor with similar properties to that in our definition which induces this same functor at the level of homotopy categories. We will not need this level of generality.}
\end{definition}
The homotopy colimit, thus defined, is functorial on the functor category, so that a natural transformation of functors $F\to G$ induces a map on homotopy colimits. A fundamental property of homotopy colimits is that they are \emph{homotopy invariant}.

\begin{lemma}\label{lemma: hocolim and weak equivalence}
The homotopy colimit of a natural weak homotopy equivalence between functors $F$ and $G$ is a weak homotopy equivalence.
\end{lemma}

\begin{definition}\label{def:complete cover}
Let $X$ be a topological space and $\mathcal{U}$ an open cover of $X$. We say $\mathcal{U}$ is \emph{complete} if it is possible to write any finite intersection of elements of $\mathcal{U}$ as a union of elements of $\mathcal{U}$.
\end{definition}

We view the open cover $\mathcal{U}$ as partially ordered under inclusion and thereby as a category. There is a tautological functor $\Gamma:\mathcal{U}\to \Top$ taking $U\in\mathcal{U}$ to the topological space $U$. Our main use of homotopy colimits is via the following result.

\begin{theorem}[{\cite[Prop.~4.6]{DuggerIsaksen:THA1R}}]
\label{thm:dugger-isaksen}
If $\mathcal{U}$ is a complete cover of $X$, then the natural map \[\hocolim_{\mathcal{U}}\Gamma\longrightarrow X\] is a weak homotopy equivalence.
\end{theorem}

One can also define homotopy colimits for functors valued in chain complexes. For simplicity we only define a special case; more detail is available in the references at the beginning of the appendix.

Given a simplicial chain complex $V:\Delta\!^{op}\to \Ch_\mathbb{Z}$, we may construct a bicomplex $\mathrm{Alt}(V)$ by taking the alternating sum of the face maps. 

\begin{proposition}[{\cite[Prop.~19.9]{Dugger:PHC}}]\label{prop:dold-kan hocolim}
Let $V:\Delta\!^{op}\to \Ch_\mathbb{Z}$ be a simplicial chain complex concentrated in non-negative degrees. There is a natural weak equivalence \[\hocolim_{\Delta\!^{op}} V\simeq\mathrm{Tot}(\mathrm{Alt}(V)),\] where $\mathrm{Tot}$ denotes the total complex.
\end{proposition}

The following standard result asserts that the notions of homotopy colimit valued in topological spaces and chain complexes are compatible.

\begin{proposition}\label{prop:sing commutes with hocolim}
Let $F:\D\to \Top$ be a functor. There is a natural quasi-isomorphism \[\hocolim_\D C^\sing(F)\simeq C^\sing(\hocolim_\D F).\]
\end{proposition}

We shall also make use of a relative version of the homotopy colimit construction. Recall that, given a functor $T:\D_1\to \D_2$ and the existence of enough colimits in $\Cc$, the restriction functor $T^*:\Fun(\D_2, \Cc)\to \Fun(\D_1,\Cc)$ admits a left adjoint $\Lan$, the \emph{left Kan extension} functor. The homotopical version of this construction is the following.\footnote{As with the homotopy colimit, the homotopy left Kan extension may be defined more invariantly in terms of homotopy categories. With this setup, our definition is a proposition---see {\cite[Prop.~10.2]{Dugger:PHC}}, for example.}

\begin{definition}\label{def: holan}
Let $\D_2\xleftarrow{T}\D_1\xrightarrow{F}\Cc$ be functors with $\Cc=\Top$ or $\Cc=\Ch_\mathbb{Z}$. The \emph{homotopy left Kan extension} of $F$ along $T$ is the functor from $\D_2$ to $\Cc$ given (on objects) by\[\hoLan_TF(d)= \hocolim_{(T\downarrow d)}(F\circ\forgettodomain).
\]
\end{definition}

Here, for an object $d\in\D_2$, the \emph{overcategory} $(T\!\downarrow\! d)$ has as objects the pairs $(d', f)$ with $d'\in \D_1$ and $f:T(d')\to d$ a morphism in $\D_2$. A morphism from $f:T(d')\to d$ to $g:T(d'')\to d$ is a morphism $h:d'\to d''$ such that $g\circ T(h)=f$. The forgetful functor to $\D_1$ takes $(d',f)$ to $d'$. The construction $(T\downarrow{d})$ is functorial in $d$, and $\hoLan_TF$ extends to a functor using this functoriality and the functoriality of the homotopy colimit.

\begin{example}\label{example: Left Kan extensions recover homotopy colimits}
When $\D_2$ is the trivial category $\DeltaZero$ with one object $*$ (so that there is a unique functor $*:\D_1\to\DeltaZero$), this construction recovers the homotopy colimit:
\[
\hoLan_{*} F(*) \cong \hocolim_{\D_1}F
\]
\end{example}

Dually, the objects of the \emph{undercategory} $(d\!\downarrow\!T)$ has as objects the pairs $(f,d')$ with $d'\in \D_1$ and $f:d\to T(d')$ a morphism in $\D_2$, and as morphisms morphisms $h$ in $\D_1$ satisfying a dual condition. Overcategories and undercategories are very useful in the calculation of homotopy colimits. In order to say how this is so, we require a preliminary definition.

\begin{definition}
Let $T:\D_1\to \D_2$ be a functor. We say that $T$ is \begin{enumerate}
\item \emph{homotopy final} if $(d\!\downarrow\! T)$ is contractible for every $d\in\D_2$, or
\item \emph{homotopy initial} if $(T\!\downarrow\!d)$ is contractible for every $d\in \D_2$.
\end{enumerate}
\end{definition}

That is, $T$ is homotopy initial just in case $T^{op}:\D_1^{op}\to \D_2^{op}$ is homotopy final.

\begin{proposition}[{\cite[Thm.~8.5.6]{Riehl:CHT}}]\label{prop:finality criterion}
Let $\D_1\xrightarrow{T}\D_2\xrightarrow{F}\Cc$ be functors with $\Cc=\Top$ {\rm(}or $\Cc=\Ch_\mathbb{Z}${\rm)}. If $T$ is homotopy final, then the natural map \[\hocolim_{\D_1}T^*F\to \hocolim_{\D_2}F\] is a weak  homotopy equivalence {\rm(}quasi-isomorphism{\rm)}.
\end{proposition}

We will make use of the following immediate consequences of this result.

\begin{corollary}\label{cor:contractible constant hocolim}
Let $\D$ be a category and $\underline c:\D\to \Cc$ the constant functor at $c\in \Cc$, where $\Cc=\Top$ {\rm(}or $\Cc=\Ch_\mathbb{Z}${\rm)}. If $\D$ is contractible, then the natural map \[\hocolim_\D\underline c\to c\] is a weak  homotopy equivalence {\rm(}quasi-isomorphism{\rm)}.
\end{corollary}

\begin{corollary}\label{cor:final functor nerve}
Let $T:\D_1\to \D_2$ be any functor. If $T$ is homotopy final, then the induced map $N\D_1\to N\D_2$ is a weak homotopy equivalence.
\end{corollary}

\section{Comparing decompositions}\label{section:comparing decompositions}

In this appendix, we tighten the combinatorial side of the correspondence between decompositions and relative tensor products arising from Theorem~\ref{thm:decomposition}. We present a general conceptual framework in which comparison questions may be formulated, and we provide one uniform answer to such questions in the form of Proposition~\ref{prop:abstract comparison}. We then apply this comparison result in two examples used in the main text. We hope that these examples will make manifest that decompositions and relative tensor products may be manipulated in an essentially identical way.

\subsection{Abstract setup} Throughout this section, we work exclusively with decompositions $\E$ and $\F$ of a fixed complex $X$. We assume that $F$ is a symmetric monoidal functor determining both an $\E$-local invariant and an $\F$-local invariant. With more notational baggage, one could work in a more general setting.

\begin{definition}
An $(r,s)$-\emph{comparison scheme} is a functor $\alpha:\Delta^r\to \Delta^s$.
\end{definition}

We think of a comparison scheme as specifying the pattern of an operation on iterated bimodules, as reflected in the simplicial bar construction, or, equally, the pattern of an operation on decompositions, as reflected in the trace.

\begin{notation} Given an $r$-fold decomposition $\E$, an $s$-fold decomposition $\F$, and an $(r,s)$-comparison scheme $\alpha$, we write $\gaps{\alpha}$ for the category of pairs $(A,B)$ with $A\in\gaps{r}$, $B\in\gaps{s}$, $\gamma_\E(A)\subseteq\gamma_\E(B)$, and $\alpha(\tau(A))=\tau(B)$. By design, this category fits into a commuting diagram \[\begin{tikzcd}
\gaps{r}\ar{d}[swap]{\tau}&\gaps{\alpha}\ar{l}[swap]{p}\ar{r}{q}&\gaps{s}\ar{d}{\tau}\\
\Delta^{r}\ar{rr}{\alpha}&&\Delta^{s}.
\end{tikzcd}\]
\end{notation}

\begin{definition}
We say that $\E$ and $\F$ are $\alpha$-\emph{comparable} if $p:\gaps{\alpha}\to \gaps{r}$ is homotopy initial.
\end{definition}

Roughly, $\alpha$-comparability asserts that every subspace in the image of $\gamma_\E$ is contained in a contractible collection of subspaces in the image of $\gamma_\F$ with compatible patterns of connected components.

We now add local flows into the mix.

\begin{definition}
Suppose that $F$ is equipped with isotopy invariant local flows on both $\E$ and $\F$.
\begin{enumerate}
\item We say that a comparison scheme $\alpha$ is \emph{flow compatible} if, for every $(A,B)\in \gaps{\alpha}$, the inclusion $\gamma_\E(A)\subseteq\gamma_\F(B)$ is flow compatible.
\item Fix a flow compatible comparison scheme $\alpha$. A \emph{comparison datum} is a map $\mathrm{Bar}_\Delta(I(\E))\to \mathrm{Bar}_\Delta(I(\F))\circ\alpha^{op}$ in the homotopy category of $r$-fold simplicial chain complexes, which fits into the commuting diagram \[\begin{tikzcd}
p^*I(F(\gamma_\E))\ar{r}\ar[swap]{dd}{p^*\psi}&q^*I(F(\gamma_\F))\ar{d}{q^*\psi}\\
&q^*(\mathrm{Bar}_\Delta(I(\F))\circ\tau^{op})
\ar[no head, d, shift left=.97pt]
\ar[no head, d, shift right=.97pt]
\\
p^*(\mathrm{Bar}_\Delta(I(\E))\circ\tau^{op})\ar[dashed]{r}&p^*(\mathrm{Bar}_\Delta(I(\F))\circ\alpha^{op}\circ\tau^{op}).
\end{tikzcd}\]
\end{enumerate}
\end{definition}
The upper horizontal arrow exists by flow compatibility. We call the map $\mathrm{Bar}(I(\E))\to \mathrm{Bar}(I(\F))$ induced by a comparison datum the ``comparison map.''

We these definitions in hand, the following result is essentially formal.

\begin{proposition}\label{prop:abstract comparison}
If $\E$ and $\F$ are $\alpha$-comparable, then the composite \[\mathrm{Bar}(I(\E))\to \mathrm{Bar}(I(\F))\simeq C^\SD(F(X))\] of the comparison map followed by the weak equivalence of Theorem~\ref{thm:decomposition} applied to $\F$ coincides with the weak equivalence of Theorem~\ref{thm:decomposition} applied to $\E$. In particular, the comparison map is a quasi-isomorphism.
\end{proposition}

\subsection{Parenthesization}\label{subsection:parenthesization} For our first sample application, we begin by noting that the derived tensor product computed by the bar construction on an iterated bimodule enjoys a great deal of symmetry, remaining invariant under any parenthesization of the factors. To give one example, there is a canonical quasi-isomorphism \[M_0\otimes^\mathbb{L}_{R_1}M_1\otimes^\mathbb{L}_{R_2} M_2\simeq M_0\otimes^\mathbb{L}_{R_1}(M_1\otimes^\mathbb{L}_{R_2}M_2).\] The corresponding comparison in the context of the decomposition theorem is beween, on the one hand, the decomposition $\E$ with components $X_0$, $X_1$, and $X_2$ and bridges $A_1\times I$ and $A_2\times I$; and, on the other hand, the decomposition $\E_{\mathrm{par}}$ with components $X_0$ and $X_1\amalg_{A_2\times I}X_2$ and bridge $A_1\times I$ (see Figure~\ref{figure: three decompositions}). In order to apply Proposition~\ref{prop:abstract comparison}, we proceed as follows.
\begin{enumerate}
\item We choose our $(2,1)$-comparison scheme $\alpha:\Delta^2\to \Delta$ to be the projection onto the first factor.
\item With this choice, $\E$ and $\E_{\mathrm{par}}$ are $\alpha$-comparable. Indeed, the undercategory in question is filtered.
\item We assume we are given isotopy invariant local flows for which $\alpha$ is flow compatible. This point will depend on the specifics of the situation.
\item In the manner of Construction~\ref{construction:algebra and module structure}, we build a collection of maps \[I(F(X_1))\otimes I(F(A_2))\otimes\cdots\otimes I(F(A_2))\otimes I(F(X_2))\to I(F(X_1\amalg_{A_2\times I}X_2))\] using isotopy invariance and flow compatibility of the structure maps of $F$ and the comparison scheme. These maps respect the simplicial structure maps, and the induced map of simplicial chain complexes is a comparison datum.
\end{enumerate} Invoking Proposition~\ref{prop:abstract comparison}, we see that the quasi-isomorphism \[\mathrm{Bar}(I(\E))\xrightarrow{\sim} \mathrm{Bar}(I(\E_{\mathrm{par}}))\] is compatible with those supplied by Theorem~\ref{thm:decomposition}. 

\begin{remark}
It should be clear that an arbitrary parenthesization may be treated in an identical fashion.
\end{remark}

\begin{figure}
\begin{align*}
\E&\vcenter{\hbox{\includegraphics{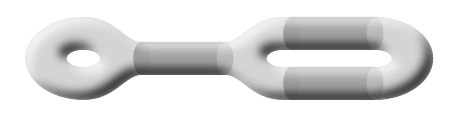}}}\\
\E_{\mathrm{par}}&\vcenter{\hbox{\includegraphics{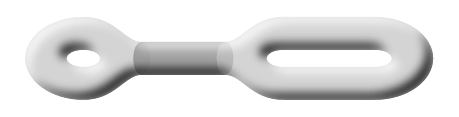}}}\\
\E_{\mathrm{fld}}&\vcenter{\hbox{\includegraphics{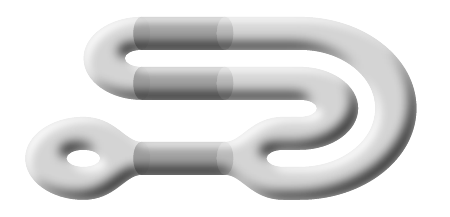}}}
\end{align*}
\caption{From top to bottom, the decompositions $\E$, $\E_{\mathrm{par}}$, and $\E_{\mathrm{fld}}$}
\label{figure: three decompositions}
\end{figure}

\subsection{Folding}\label{subsection:folding} Our second example concerns the canonical quasi-isomorphism \[M_0\otimes^\mathbb{L}_{R_1}M_1\otimes^\mathbb{L}_{R_2}M_2\simeq (M_0\otimes M_2)\otimes^\mathbb{L}_{R_1\otimes R_2^{op}}M_1\] The corresponding comparison in the context of the decomposition theorem is between the decomposition $\E$ already introduced and the decomposition $\E_{\mathrm{fld}}$ with components $X_0\amalg X_2$ and $X_1$ and bridge $(A_1\amalg A_2)\times I$ obtained by ``folding the end over'' (see Figure~\ref{figure: three decompositions}). The data in the this case is the following.

\begin{enumerate}
\item We choose our $(1,2)$-comparison scheme $\alpha:\Delta\to \Delta^2$ to be the identity in the first factor and the ``flip'' automorphism in the second.
\item With this choice, $\E_{\mathrm{fld}}$ and $\E$ are $\alpha$-comparable. Indeed, the undercategory in question has a final object.
\item We assume we are given isotopy invariant local flows for which $\alpha$ is flow compatible. This point will depend on the specifics of the situation.
\item Observe that the simplicial chain complexes $\mathrm{Bar}_\Delta(\E_{\mathrm{fld}})$ and $\mathrm{Bar}_\Delta(\E)\circ\alpha^{op}$ are identical in each simplicial degree, except that the former uses the local flow for $\E_{\mathrm{fld}}$, while the latter uses the local flow for $\E$. The comparison datum comes from matching terms and using the flow compatibility of $\alpha$.
\end{enumerate} 
Invoking Proposition~\ref{prop:abstract comparison}, we see that the quasi-isomorphism \[\mathrm{Bar}(I(\E_{\mathrm{fld}}))\xrightarrow{\sim}\mathrm{Bar}(I(\E))\] is compatible with those supplied by Theorem~\ref{thm:decomposition}. 

\section{Degree one homology of graph braid groups}\label{appendix:degree one}

The homology group $H_1(B_k(\graf))$ for a general connected graph $\graf$ is completely understood; indeed, according to a theorem of Ko--Park \cite[Thm.~3.16]{KoPark:CGBG}, this group may be identified solely in terms of connectivity and planarity data from $\graf$. Their argument proceeds through an intricate combinatorial and linear algebraic analysis of the discrete Morse data constructed in \cite{FarleySabalka:DMTGBG}, which quickly becomes very technical. 

In this appendix, we use the \'{S}wi\k{a}tkowski complex to simplify the calculation of $H_1$ by giving streamlined proofs of four key lemmas of \cite{KoPark:CGBG}, which appear in \S\ref{section:key lemmas} below, cross-referenced with the corresponding results in the original text (and in~\cite{HarrisonKeatingRobbinsSawicki:NPQSG}, where versions of them also appear). These four results directly imply the calculation of $H_1$, which we outline below but do not state in full. We make no claim of originality in this appendix, our purpose being only to demonstrate the efficiency of the \'{S}wi\k{a}tkowski complex in applications.

\subsection{Cuts and connectivity} In order to proceed, we will require some terminology from graph theory. Note that these invariants should only be used as defined on \emph{simple} graphs (see \S\ref{section:conventions}) and may behave unexpectedly on general graphs.

\begin{definition}
Let $\graf$ be a simple graph. A \emph{$k$-cut} is a set of $k$ vertices whose removal topologically separates at least two vertices of $\graf$. A simple graph is \emph{$k$-connected} if it has at least $k+1$ vertices and no $(k-1)$-cuts.

Given a $1$-cut $v$ in $\graf$, a \emph{$v$-component} of $\graf$ is the closure in $\graf$ of a connected component of the complement of $v$ in $\graf$.
\end{definition}

The importance of connectivity for our purposes arises from the following classical result, called Menger's theorem.

\begin{theorem}[\cite{Menger:SAK}]\label{thm: Menger}
Let $k>0$. A simple graph is $k$-connected if and only if, for distinct vertices $x$ and $y$ in $\graf$, there exist $k$ paths from $x$ to $y$ in $\graf$, disjoint except at endpoints.
\end{theorem}
In other words, a simple graph is $k$-connected (for positive $k$) if and only if any pair of vertices $x$ and $y$ there is a graph embedding of a subdivision of the theta graph $\thetagraph{k}$ into $\graf$ taking the two vertices of the theta graph to $x$ and $y$.

As the following result shows, high connectivity places strong constraints on the behavior of the first homology.

\begin{proposition}\label{prop:triconnected have only one star class}
If $\graf$ is simple and $3$-connected, then any two star classes in $H_*(B(\graf))$ coincide up to sign.
\end{proposition}
\begin{proof}
By $3$-connectivity, any two vertices of $\graf$ are joined by three distinct paths, so a star class at one vertex coincides with some star class at any other vertex by the $\Theta$-relation. Thus, it suffices to verify that two star classes $\alpha$ and $\alpha'$ at a fixed vertex $v$ coincide up to sign. We may assume further that $\alpha$ and $\alpha'$ are induced by inclusions of $\stargraph{3}$ differing only at a single half-edge, so that we have an inclusion of $\stargraph{4}$ into $\graf$ (with endpoint vertices $v_1,\ldots,v_4$ and central vertex $v$ all distinct by $3$-connectedness). We will show that the four  star classes in $\graf$ obtained in this way agree up to sign.

A corollary of Menger's theorem is the fact that in a $k$-connected graph, any two sets of vertices, each of size $k$, can be joined by $k$ disjoint paths~\cite[Prop.~9.4]{BondyMurty:GT}. Since $\Gamma_v$ is $2$-connected, two disjoint paths join $\{v_1, v_4\}$ and $\{v_2,v_3\}$. By relabelling suppose one joins $v_1$ to $v_2$ and the other joins $v_3$ to $v_4$. Similarly, two disjoint paths join $\{v_1,v_2\}$ to $\{v_3,v_4\}$. By switching the labels on $v_1$ and $v_2$ if need be, we may assume one joins $v_1$ to $v_3$ and the other joins $v_2$ to $v_4$. We now have two extensions of the inclusion $\stargraph{4}\to \graf$ to an inclusion of a subdivision of the figure-eight graph $\figureeightgraph$ as in Figure~\ref{figure: figure eight}, and the $Q$-relation and $I$-relation imply that $\alpha_{123}=\alpha_{124}$ and that $\alpha_{432}=\alpha_{431}$ (using the first set of paths), and that $\alpha_{312}=\alpha_{314}$ (using the second set of paths).
\end{proof}
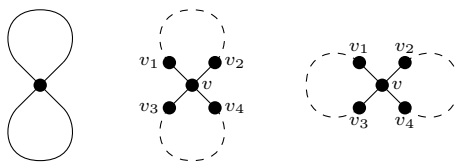
\begin{figure}[ht]
\centering
\begin{tikzpicture}
\fill[black] (0,0) circle (2.5pt);
\draw (-.3,-.3)--(.3,.3);
\draw (-.3,.3)--(.3,-.3);
\draw (.3,.3) .. controls (.5,.5) and (.5,1) .. (0,1) .. controls (-.5,1) and (-.5,.5) .. (-.3,.3);
\draw (.3,-.3) .. controls (.5,-.5) and (.5,-1) .. (0,-1) .. controls (-.5,-1) and (-.5,-.5) .. (-.3,-.3);
\begin{scope}[xshift=2cm]
\fill[black] (0,0) circle (2.5pt);
\draw(0,0) node[right]{$\scriptstyle v$};
\draw(-.3,.3) node[left]{$\scriptstyle v_1$};
\draw(.3,.3) node[right]{$\scriptstyle v_2$};
\draw(-.3,-.3) node[left]{$\scriptstyle v_3$};
\draw(.3,-.3) node[right]{$\scriptstyle v_4$};
\fill[black] (.3,.3) circle (2.5pt);
\fill[black] (-.3,.3) circle (2.5pt);
\fill[black] (-.3,-.3) circle (2.5pt);
\fill[black] (.3,-.3) circle (2.5pt);
\draw (-.3,-.3)--(.3,.3);
\draw (-.3,.3)--(.3,-.3);
\draw [dashed](.3,.3) .. controls (.5,.5) and (.5,1) .. (0,1) .. controls (-.5,1) and (-.5,.5) .. (-.3,.3);
\draw [dashed](.3,-.3) .. controls (.5,-.5) and (.5,-1) .. (0,-1) .. controls (-.5,-1) and (-.5,-.5) .. (-.3,-.3);
\end{scope}
\begin{scope}[xshift=4.5cm]
\fill[black] (0,0) circle (2.5pt);
\fill[black] (.3,.3) circle (2.5pt);
\fill[black] (-.3,.3) circle (2.5pt);
\fill[black] (-.3,-.3) circle (2.5pt);
\fill[black] (.3,-.3) circle (2.5pt);
\draw(0,0) node[right]{$\scriptstyle v$};
\draw(-.3,.3) node[above]{$\scriptstyle v_1$};
\draw(.3,.3) node[above]{$\scriptstyle v_2$};
\draw(-.3,-.3) node[below]{$\scriptstyle v_3$};
\draw(.3,-.3) node[below]{$\scriptstyle v_4$};
\draw (-.3,-.3)--(.3,.3);
\draw (-.3,.3)--(.3,-.3);
\draw [dashed](.3,.3) .. controls (.5,.5) and (1,.5) .. (1,0) .. controls (1,-.5) and (.5,-.5) .. (.3,-.3);
\draw [dashed](-.3,.3) .. controls (-.5,.5) and (-1,.5) .. (-1,0) .. controls (-1,-.5) and (-.5,-.5) .. (-.3,-.3);
\end{scope}
\end{tikzpicture}
\caption{The figure-eight graph $\figureeightgraph$ and two inclusions of subdivisions of it at a vertex of valence at least $4$ in a triconnected graph}
\label{figure: figure eight}
\end{figure}

\subsection{Minors and planarity}

Recall that a graph is said to be \emph{planar} if it can be embedded in $\mathbb{R}^2$. The goal of this section is to establish the following connection between planarity and configuration spaces.

\begin{lemma}\label{lem:non planar 2 torsion}
Let $\graf$ be a $3$-connected simple graph. 
If $\graf$ is non-planar, then any star class of $H_1(B_2(\graf))$ is $2$-torsion.
\end{lemma}

As before, the proof will proceed by reduction to atomic cases. In order to see how this reduction will proceed, we require an auxuiliary notion.

\begin{definition}\label{defi:minor}
Let $\graf$ and $\graf'$ be graphs. We say that $\graf$ is a \emph{minor} of $\graf'$ if $\graf$ may be obtained up to isomorphism from $\graf$ by repeated application of the following operations:
\begin{enumerate}
\item \label{item: contract an edge} contract an edge;
\item \label{item: remove an edge} remove an edge;
\item \label{item: remove a vertex} remove an isolated vertex.
\end{enumerate}
\end{definition}

The relevance of minors for our purposes is the following classical result (see Figures~\ref{figure: k5} and~\ref{figure: k33} for terminology).

\begin{theorem}[\cite{Wagner:UEEK}]\label{thm: Wagner}
A graph is non-planar if and only if it admits $\completegraph{5}$ or $\completegraph{3,3}$ as a minor.
\end{theorem}

In order to apply this criterion, we must first clarify the relationship between the \'{S}wi\k{a}tkowski complex of a graph and that of its minors. We begin by noting that, if $\graf$ is a minor of $\graf'$, then $E$ is naturally identifed with a subset of $E'$.

\begin{lemma}\label{lemma: graph minors}
Let $\graf$ be a minor of $\graf'$. There is a map $f:\intrinsic{\graf}\to \intrinsic{\graf'}$ of differential graded $\mathbb{Z}[E]$-modules such that $f_*(\alpha)$ is a star class whenever $\alpha$ is.
\end{lemma}

\begin{proof}
We may assume that $\graf$ is obtained from $\graf'$ by a single application of one of the operations of Definition~\ref{defi:minor}. In the case of operation (\ref{item: remove an edge}) or (\ref{item: remove a vertex}), $\graf$ is a subgraph of $\graf'$, and the claim is immediate from the standard functoriality of the \'{S}wi\k{a}tkowski complex. Since the same holds for the contraction of a tail or a self-loop, we may assume that $\graf$ is obtained from $\graf'$ by contracting an edge with two distinct vertices $v_1$ and $v_2$. Denote the corresponding half-edges by $h_1$ and $h_2$, and let $v_0$ be the vertex that is the image of the closure of $e$ under the quotient map $\graf'\to \graf$. 

In order to define $f$, we note that, under the quotient map, each vertex $v\neq v_0$ of $\graf$ is canonically identified with a vertex $\tilde v$ of $\graf'$, and each half-edge of $h$ of $\graf$ is canonically identified with a half-edge $\tilde h$ of $\graf'$. With this in mind, we set 
\begin{align*}
f(v)&=\begin{cases}
\tilde v,& v\ne v_0\\
e, & v=v_0
\end{cases}
&
f(h)&=\begin{cases}
\tilde h,& v(h)\neq v_0\\
\tilde h-h_j, & v(\tilde h)=v_j.
\end{cases}
\end{align*}
By inspection, $f$ respects the differential and $\mathbb{Z}[E]$-action, and the claim regarding star classes is a direct calculation with star cycles.
\end{proof}

\begin{proof}[Proof of Lemma~\ref{lem:non planar 2 torsion}]
Since $\graf$ is non-planar, it admits $\completegraph{5}$ or $\completegraph{3,3}$ as a minor, and we may push a star class of the minor forward using Lemma~\ref{lemma: graph minors} to obtain a star class in $\graf$. Therefore, it suffices to prove the claim for $\completegraph{5}$ and $\completegraph{3,3}$.

In the case of $\completegraph{5}$, we can find an embedded copy of the theta graph $\thetagraph{3}$ as in Figure~\ref{figure: k5}. This shows that the star class $\alpha_{123}$ at $v$ is the negative of the star class $\alpha_{3'1'2'}=\alpha_{1'2'3'}$ at $v'$. The picture depicted and concomitant $\Theta$-relation can be rotated by $\frac{2\pi}{5}$; applying such relations $5$ times shows that the star class $\alpha_{123}$ is equal to $(-1)^{5}\alpha_{123}$. 
\begin{figure}[ht]
\centering
\begin{tikzpicture}
\node[draw=none,minimum size=2.5cm,regular polygon,regular polygon sides=5] (a) {};
\foreach \x in {1,...,5}
\foreach \y in {1,...,\x}
\draw(a.corner \x) -- (a.corner \y);
\draw(-.3,1.3) node {$\scriptstyle h_1$};
\draw(-.1,.9) node[fill=white, inner sep=0pt] {$\scriptstyle h_2$};
\draw(.3,1.3) node {$\scriptstyle h_3$};
\draw(0,1.45) node {$\scriptstyle v$};
\draw(-1.13,.7) node {$\scriptstyle h'_3$};
\draw(-.9,.2) node[fill=white, inner sep=0pt] {$\scriptstyle h'_2$};
\draw(-1.27,.1) node {$\scriptstyle h'_1$};
\draw(-1.35,.45) node {$\scriptstyle v'$};
\foreach \x in {1,2,...,5}
\fill[black] (a.corner \x) circle (2.5pt);
\begin{scope}[xshift= 5cm]
\node[draw=none,minimum size=2.5cm,regular polygon,regular polygon sides=5] (a) {};
\draw [dashed, very thin, gray] (a.corner 1) -- (a.corner 4) -- (a.corner 3) -- (a.corner 5) -- (a.corner 2);
\draw[colorskyblue](a.corner 1) -- (a.corner 3) -- (a.corner 2);
\draw[colorvermillion](a.corner 2) -- (a.corner 4)-- (a.corner 5) -- (a.corner 1);
\draw[coloryellow](a.corner 1) -- (a.corner 2);
\foreach \x in {1,2,...,5}
\fill[black] (a.corner \x) circle (2.5pt);
\end{scope}
\end{tikzpicture}
\caption{The complete graph $\completegraph{5}$ and an embedded subdivision of $\thetagraph{3}$ with edges of $\thetagraph{3}$ color-coded.}\label{figure: k5}
\end{figure}
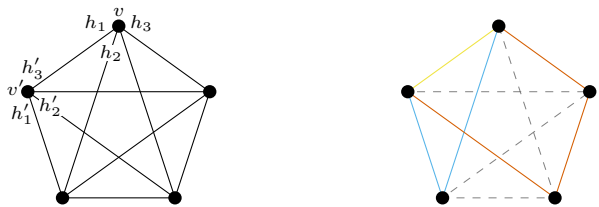

In the case of $\completegraph{3,3}$, we have two embedded subdivisions of $\thetagraph{3}$, as in Figure~\ref{figure: k33}. Applying the $\Theta$-relation twice, we see that both a star class at the upper left vertex and its additive inverse are equal to the same star class at the upper right vertex.
\end{proof}

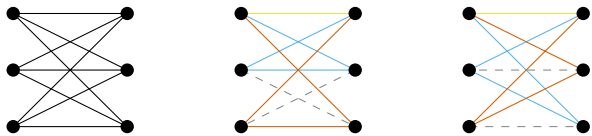
\begin{figure}[ht]
\centering
\begin{tikzpicture}
\foreach  \x in {-.75,.75}\foreach \y in {-.75,0,.75}
\fill[black] (\x,\y) circle (2.5pt);
\draw (-.75,.75) -- (.75,.75) -- (-.75,0) -- (.75,0) -- (-.75, -.75) -- (.75, .75);
\draw (.75, -.75) -- (-.75,.75) -- (.75, 0);
\draw (-.75,0) -- (.75,-.75)-- (-.75,-.75);
\begin{scope}[xshift= 3cm]
\draw [dashed, very thin, gray](.75, 0) -- (-.75,-.75);
\draw [dashed, very thin, gray](-.75, 0) -- (.75,-.75);
\draw[coloryellow] (-.75,.75) -- (.75,.75);
\draw[colorskyblue](.75,.75) -- (-.75,0) -- (.75,0) -- (-.75, .75);
\draw[colorvermillion](-.75, .75) -- (.75, -.75) -- (-.75,-.75) -- (.75, .75);
\foreach  \x in {-.75,.75}\foreach \y in {-.75,0,.75}
\fill[black] (\x,\y) circle (2.5pt);
\end{scope}
\begin{scope}[xshift= 6cm]
\draw [dashed, very thin, gray](.75, 0) -- (-.75,0);
\draw [dashed, very thin, gray](-.75,-.75) -- (.75,-.75);
\draw[coloryellow] (-.75,.75) -- (.75,.75);
\draw[colorskyblue](.75,.75) -- (-.75,0) -- (.75,-.75) -- (-.75, .75);
\draw[colorvermillion](-.75, .75) -- (.75, 0) -- (-.75,-.75) -- (.75, .75);
\foreach  \x in {-.75,.75}\foreach \y in {-.75,0,.75}
\fill[black] (\x,\y) circle (2.5pt);
\end{scope}
\end{tikzpicture}
\caption{The complete bipartite graph $\completegraph{3,3}$ and two embedded subdivisions of $\thetagraph{3}$ with edges of $\thetagraph{3}$ color-coded.}\label{figure: k33}
\end{figure}

\subsection{Four lemmas of Ko--Park}\label{section:key lemmas}
We are now equipped to prove the results of \cite{KoPark:CGBG} alluded to above, which together easily yield a complete calculation of $H_1(B_k(\graf))$ for an arbitrary connected graph $\graf$. We content ourselves with an outline of this computation, directing the interested reader to \cite{KoPark:CGBG} for details.
Versions of these lemmas are also proven in~\cite{HarrisonKeatingRobbinsSawicki:NPQSG}, and we have included citations to the appropriate parts of that paper as well. Harrison et al. use combinatorial ingredients similar to ours, but the logic of their arguments often diverges substantially from ours.

Since any graph has a simple subdivision (see \S\ref{section:conventions}), we may assume that $\graf$ is simple. In this case, there are classical decomposition theorems from graph theory associating to a $1$-connected simple graph a set of $2$-connected simple graphs, obtained by repeated $1$-cuts, and to each of these a set of $3$-connected simple graphs, obtained by repeated $2$-cuts. With this tool in hand, together with the observation that it suffices to assume that $\graf$ is $1$-connected, the argument proceeds as follows.
\begin{enumerate}
\item In characterizing the effect of a $1$-cut, Lemma~\ref{lem:1-cut} reduces the computation for $\graf$ to the computation for each associated $2$-connected graph.
\item Assuming now that $\graf$ is $2$-connected, Lemma~\ref{lem:biconnected stabilization} grants that $H_1(B_k(\graf))$ is independent of $k$ for $k\geq2$.
\item Lemma~\ref{lem:2-cut} describes the effect of a $2$-cut on $H_1(B_2(\graf))$ for $2$-connected $\graf$, reducing the calculation to graphs that are simpler in some technical sense. 
A variant of a classical decomposition theorem in graph theory~\cite{Tutte:CG,CunninghamEdmonds:CDT} says iterating this simplification procedure eventually reduces all $2$-connected graphs to $3$-connected graphs, cycle graphs, and theta graphs---see Remark~\ref{remark: how to go from 2-connected to 3-connected}.
\item Finally, Lemma~\ref{lemma: triconnected} computes $H_1(B_2(\graf))$ for $3$-connected $\graf$. 
The calculation for theta graphs can be understood by inspection or via the vertex long exact sequence of Corollary~\ref{cor: les vertex reduced}.
\end{enumerate}
With careful bookkeeping of cuts and these graph decomposition theorems, one can assemble these results to arrive at an explicit answer, which we do not repeat here.

The first result is a sharpened version of the degree 1 component of Proposition~\ref{proposition: edge multiplication is injective} in the presence of sufficient connectivity.

\begin{lemma}[{\cite[Lem.~3.12]{KoPark:CGBG}; see also~\cite[Theorem 5]{HarrisonKeatingRobbinsSawicki:NPQSG}}]\label{lem:biconnected stabilization}
Suppose that $\graf$ is simple and $2$-connected. For any $e\in E$ and $k\geq2$, multiplication by $e^{k-2}$ induces an isomorphism \[H_1(B_2(\graf))\xrightarrow{\simeq} H_{1}(B_k(\graf)).\]
\end{lemma}
\begin{proof}
By Propositions~\ref{proposition: edge multiplication is injective} and~\ref{prop:H1 generators}, it will suffice to show that $p(E)\alpha$ and $p(E)\gamma$ are divisible by $e$ in the expected range, where $p(E)$ is a monomial in $E$, $\alpha$ is a star class at the vertex $v$, and $\gamma$ is a loop class.

In the case of a star class $\alpha$, there is nothing to show provided $p(E)\in\mathbb{Z}[e]$, so we may write $p(E)\alpha=[q(E)e'a],$ where $e'\neq e$ and $[a]=\alpha$. Since $\graf$ is $2$-connected, both $e'$ and $e$ have vertices distinct from $v$, and there is an injective path in $\graf\setminus\{v\}$ between them. Letting $h$ denote the alternating sum of the half-edges involved in this path, we have $\partial(q(E)ha)=q(E)ea-q(E)e'a$. An easy induction on the degree of $q$ completes the argument in this case.

In the case of a loop class $\gamma$, we write $p(E)\gamma=[q(E)e'c],$ where $e'\neq e$ and $[c]=\gamma$. Since $\graf$ is connected, there is a path connecting $e'$ to $e$. If this path contains none of the edges or vertices involved in $c$, then the same argument as before shows that $q(E)ec\sim q(E)e'c$; on the other hand, if the path is contained entirely in $c$, then the same conclusion follows by the $O$-relation. Thus, we may assume that $e$ lies in the complement of $c$, $e'$ lies in $c$, and that the two share a vertex. By the $Q$-relation, we have $q(E)ec\sim \pm q(E)e'c\pm q(E)a$, where $[a]$ is a star class. Since the case of a star class is known, this completes the proof.
\end{proof}

The second result is the base calculation, the $3$-connected case.

\begin{lemma}[{\cite[Lem.~3.15]{KoPark:CGBG}; see also~\cite[Lemma 3, Theorem 4]{HarrisonKeatingRobbinsSawicki:NPQSG}}]\label{lemma: triconnected}
If $\graf$ is $3$-connected, then $H_1(B_2(\graf))\cong \mathbb{Z}^{\beta_1(\graf)}\oplus K$, where \[K=\begin{cases}
\mathbb{Z}&\quad \text{if }\graf\text{ is planar}\\
\mathbb{Z}/2&\quad \text{else.}
\end{cases}\]
\end{lemma}
Here $K$ is generated by a star class and $\mathbb{Z}^{\beta_1(\graf)}$ by the inclusion of cycle graphs into $\graf$.
\begin{proof}
Choose a set of loops $c_i$ in $\graf$ representing a basis for $H_1(\graf)$ and a star class $\alpha$. Define a map
\[
\left(\bigoplus_{i=1}^{\beta_1(\graf)}H_1(B_2(c_i))\right)\oplus K\cong\mathbb{Z}^{\beta_1(\graf)}\oplus K\xrightarrow{\psi} H_1(B_2(\graf))
\]using the inclusions of $c_i$ in $\graf$ and taking the generator of $K$ to $\alpha$. This map is well-defined by Lemma~\ref{lem:non planar 2 torsion} and surjective by Propositions~\ref{prop:H1 generators} and~\ref{prop:triconnected have only one star class}.

Next, we ``thicken'' $\graf$, replacing vertices with disks and edges with strips, obtaining a surface $\Sigma$ of genus $g$ with $b$ boundary components equipped with an embedding $\iota:\graf\to\Sigma$ that is also a homotopy equivalence (see, \cite[p.~4-5]{Kontsevich:ITMSCMAF}, for example, and the references therein for more on this classical tactic). Note that $g>0$ if $\graf$ is non-planar, while we may take $g=0$ if $\graf$ is planar. Moreover, by comparing Euler characteristics, we find that
\[
\beta_1(\graf)=2g+b-1.
\]

Now, from the explicit presentation for $\pi_1(B_2(\Sigma))$ given in \cite{Bellingeri:OPSBG}, it is easy to see that $\iota$ induces a surjection $\pi_1(B_2(\graf))\to\pi_1(B_2(\Sigma))$ on fundamental groups, and thus also on first homology. By direct computation, the same presentation gives
\begin{align*}
H_1(B_2(\Sigma))&\cong\begin{cases}
\mathbb{Z}\oplus\mathbb{Z}^{b-1} & g=0\\
\mathbb{Z}/2\oplus \mathbb{Z}^{2g+b-1} & g\ge 1
\end{cases}\\
&\cong\mathbb{Z}^{\beta_1(\graf)}\oplus K,
\end{align*}
Thus, $H_1(B_2(\graf))$ both admits a surjection from and a surjection onto $\mathbb{Z}^{\beta_1(\graf)}\oplus K$, which implies the claim by finite dimensionality.
\end{proof}

The third result describes the effect on $H_1$ of a 1-cut.

\begin{lemma}[{\cite[Lem.~3.11]{KoPark:CGBG}; see also~\cite[Section 4.4 and Section 7]{HarrisonKeatingRobbinsSawicki:NPQSG}}]\label{lem:1-cut}
Let $\graf$ be a graph, $v$ a $1$-cut of valence $\nu$ in $\graf$, and $\{\graf^{(i)}\}_{i=1}^\mu$ the set of $v$-components of $\graf_v$. There is an isomorphism
\[H_1(B_k(\graf))\cong \left(\bigoplus_{i=1}^\mu H_1(B_k(\graf^{(i)}))\right)\oplus \mathbb{Z}^{N(k,\graf,v)}\]
where
\[
N(k,\graf,v)=(\nu-2)\binom{k+\mu-2}{k-1}-\binom{k+\mu-2}{k} - (\nu-\mu-1).
\]
\end{lemma}
\begin{proof}
Applying Corollary~\ref{cor: les vertex reduced} at $v$, we obtain the exact sequence

\begin{align*}
\cdots \to \bigoplus^{|H(v)|-1} H_1 (B_{k-1}(\graf_v))&\xrightarrow{\delta_1}
H_1(B_k(\graf_v))\to \\
H_1 (B_k(\graf))\to
\bigoplus^{|H(v)|-1} H_0 (B_{k-1}(\graf_v))&\xrightarrow{\delta_0}
H_0 (B_{k}(\graf_v))\to\cdots
\end{align*}
Since zeroth homology is free, $\ker \delta_0$ is as well, so $H_1(B_k(\graf))\cong \coker \delta_1\oplus \ker\delta_0$. 
Assume $\graf$ is connected for simplicity; then $\pi_0(B_k(\graf_v))$ is in bijection with the set of partitions of $k$ into $\mu$ distinguished blocks and we conclude by exactness that \begin{align*}\rk \ker \delta_0=(\nu-1)\binom{k+\mu-2}{k-1}-\binom{k+\mu-1}{k}+1.\end{align*} 
A similar argument shows $H_1(B_k(\graf^{(i)}))\cong \coker\delta^{(i)}_1\oplus \ker\delta^{(i)}_0$ and $\bigoplus_i \ker\delta^{(i)}_0\cong \mathbb{Z}^{(\nu-\mu)}$. Then, after a little combinatorial rearrangement, all that remains is to show that $\coker\delta_1\cong \bigoplus_{i=1}^\mu \coker\delta^{(i)}_1$. The maps $H_1(B_k(\graf^{(i)}_v))\to H_1(B_k(\graf_v))$ arising from the various inclusions induce a map $\bigoplus_{i=1}^\mu \coker\delta^{(i)}_1\to \coker\delta_1$. We will show that this map is an isomorphism.

By the K\"unneth formula (again using the fact that zeroth homology is free) the homology $H_1(B_k(\graf_v))$ splits as a direct sum over $i$. The $i$th summand consists of homology classes in $H_1(B_{k'}(\graf^{(i)}_v))$ equipped with an ordered partition of $k-k'$ into $\mu-1$ blocks for some $k'\le k$. Passing to the quotient $\coker\delta_1$, every equivalence class in the $i$th summand has a representative homology class with $k'=k$. Since this representative is in the image of the map in question, surjectivity follows.

For injectivity, we note that the map $H_1(B_k(\graf^{(i)}_v))\to H_1(B_k(\graf_v))$ lands in the $i$th summand of the direct sum decomposition, and $\delta_1$ respects this decomposition. 
Now for fixed $i$, we change basis for the direct sum indexing the domain of $\delta_1$. 
The given basis is spanned by all half-edges of $v$ except a special half-edge $h_0$. 
We choose the special half-edge to lie in the $v$-component $\graf^{(i)}_v$. 
Then for each $v$-component $\graf^{(j)}_v$ other that $\graf^{(i)}_v$, we keep one basis half-edge $h_{j0}$ corresponding to an edge in $\graf^{(j)}_v$ and change basis to $h_{jk}-h_{j0}$ for each other basis half-edge $h_{jk}$ in that $v$-component.
We retain all basis half-edges in the $v$-component $\graf^{(i)}_v$ without change.
 All of the changed basis elements are in the kernel of $\delta_1$.
The retained basis element $h_{j0}$ from other $v$-components each change one of the entries of the ordered partition of $k-k'$ under $\delta_1$. 
The set of $h_{j0}$ and the set of $\mu-1$ blocks are in bijection under this correspondence.
A homology class $\alpha$ in the codomain $H_1(B_k(\graf_v))$ of $\delta_1$ which is in the image of the inclusion of $H_1(B_k(\graf^{(i)}_v))$ has $k'=k$. 
Suppose $\alpha$ is in the image of $\delta_1$, say $\alpha=\delta_1(p_jh_{j0}+q)$, with $q$ in the summands indexed by half-edges in $\graf^{(i)}_v$. Then since $k'=k$ each $p_j$ is zero, so that $\alpha$ is also in the image of $\delta_1^{(i)}$.
 \end{proof}

The final result describes the effect on $H_1$ of a 2-cut. We will establish a result for a certain type of decomposition.
\begin{notation}
Let $\{x,y\}$ be a $2$-cut in a $2$-connected simple graph $\graf$. An \emph{$\{x,y\}$-decomposition} consists of a (redundant) collection of subgraphs of $\graf$, namely $(\graphfont{G}_1, \graphfont{G}_{-1}, \graphfont{P}_1, \graphfont{P}_{-1}, \graf_1,\graf_{-1}, \cyclegraph{})$, where
\begin{enumerate}
\item the subgraph $\graphfont{G}_1$ contains $x$ and $y$ and $\graphfont{G}_1\backslash \{x,y\}$ is a connected component of $\graf\backslash\{x,y\}$,
\item the subgraph $\graphfont{G}_{-1}$ is $(\graf\backslash \graphfont{G}_1)\cup \{x,y\}$,
\item the subgraph $\graphfont{P}_i$ is a simple path in $\graphfont{G}_i$ between $x$ and $y$,
\item the subgraph $\graf_i$ is the union of $\graphfont{G}_i$ and $\graphfont{P}_{-i}$,
\item the subgraph $\cyclegraph{}$ is the union of $\graphfont{P}_1$ and $\graphfont{P}_{-1}$.
\end{enumerate}
\end{notation}
See Figure~\ref{figure:xy decomposition} for an example. This notation is not symmetric in its indices because of the connectivity requirement on $\graphfont{G}_1\backslash\{x,y\}$.
Note that the graphs $\graf_{\pm 1}$ are still $2$-connected and simple. 
\begin{remark}

\label{remark: how to go from 2-connected to 3-connected}
This type of decomposition comes from classical graph theory~\cite{Tutte:CG,CunninghamEdmonds:CDT}, although our notation and terminology differ. In particular, we use a different definition of $2$-cut which coincides with the classical definition on simple graphs, and our $\graf_{\pm 1}$ are subdivisions of the components considered in the above references.

Using the argument of~\cite[Theorem 1]{CunninghamEdmonds:CDT}, for example, we can make a sequence of carefully chosen $\{x,y\}$-decompositions to decompose $\graf$ into a collection of subdivisions of cycle graphs, $\theta$-graphs, and $3$-connected graphs. Thus, since cycles and $\theta$-graphs may be treated by ad hoc means, reducing the problem of understanding $H_1(B(\graf))$ to that of understanding $H_1(B(\graf_{\pm 1}))$, as Lemma \ref{lem:2-cut} below, is tantamount to reducing the problem of understanding $H_1(B(\graf))$ for $2$-connected $\graf$ to the problem of understanding $H_1(B(\graf'))$ for $3$-connected $\graf'$. See~\cite{KoPark:CGBG} for details of this argument.
\end{remark}

\begin{figure}
\centering
\begin{tikzpicture}[scale=.8]
\begin{scope}[xshift=12cm, yshift=3cm]
\fill[black] (0,-1) circle (3.125pt);
\fill[black] (0,1) circle (3.125pt);
\fill[black] (-.35,0) circle (3.125pt);
\fill[black] (-1,0) circle (3.125pt);
\fill[black] (1,0) circle (3.125pt);
\fill[black] (.707,-.707) circle (3.125pt);
\draw (1, 0) arc (67.5:-112.5:0.3827cm);
\draw(1,0) arc (0:360:1cm);
\draw(.35,0) arc (0:360:.35cm and 1cm);
\draw(-.35,0) -- (-1,0);
\draw(1.15,0) node[right]{$\graf$};
\end{scope}
\begin{scope}[xshift = 9cm, yshift=4.5cm]
\fill[black] (0,-1) circle (3.125pt);
\fill[black] (0,1) circle (3.125pt);
\fill[black] (-.35,0) circle (3.125pt);
\fill[black] (-1,0) circle (3.125pt);
\fill[black] (1,0) circle (3.125pt);
\fill[black] (.707,-.707) circle (3.125pt);
\draw(0,1) arc (90:450:1cm);
\draw(0,1) arc (90:270:.35cm and 1cm);
\draw(-.35,0) -- (-1,0);
\draw(0,1.1) node[above]{$\graf_1$};
\end{scope}
\begin{scope}[xshift = 9cm, yshift=1.5cm]
\fill[black] (0,-1) circle (3.125pt);
\fill[black] (0,1) circle (3.125pt);
\fill[black] (-1,0) circle (3.125pt);
\fill[black] (1,0) circle (3.125pt);
\fill[black] (.707,-.707) circle (3.125pt);
\draw (1, 0) arc (67.5:-112.5:0.3827cm);
\draw(0,-1) arc (-90:270:1cm);
\draw(0,-1) arc (-90:90:.35cm and 1cm);
\draw(0,-1.1) node[below]{$\graf_{-1}$};
\end{scope}
\begin{scope}[xshift = 6cm, yshift=6cm]
\fill[black] (0,-1) circle (3.125pt);
\fill[black] (0,1) circle (3.125pt);
\fill[black] (-.35,0) circle (3.125pt);
\fill[black] (-1,0) circle (3.125pt);
\draw(0,1) arc (90:270:1cm);
\draw(0,1) arc (90:270:.35cm and 1cm);
\draw(-.35,0) -- (-1,0);
\draw(0,1.1) node[above]{$\graphfont{G}_1$};
\end{scope}
\begin{scope}[xshift = 6cm, yshift=3cm]
\fill[black] (0,-1) circle (3.125pt);
\fill[black] (0,1) circle (3.125pt);
\fill[black] (-1,0) circle (3.125pt);
\fill[black] (1,0) circle (3.125pt);
\fill[black] (.707,-.707) circle (3.125pt);
\draw(0,-1) arc (-90:270:1cm);
\draw(0,0) node{$\cyclegraph{}$};
\end{scope}
\begin{scope}[xshift = 6cm, yshift=0cm]
\fill[black] (0,-1) circle (3.125pt);
\fill[black] (0,1) circle (3.125pt);
\fill[black] (1,0) circle (3.125pt);
\fill[black] (.707,-.707) circle (3.125pt);
\draw (1, 0) arc (67.5:-112.5:0.3827cm);
\draw(0,-1) arc (-90:90:1cm);
\draw(0,-1) arc (-90:90:.35cm and 1cm);
\draw(0,-1.1) node[below]{$\graphfont{G}_{-1}$};
\end{scope}
\begin{scope}[xshift = 3cm, yshift=4.5cm]
\fill[black] (0,-1) circle (3.125pt);
\fill[black] (0,1) circle (3.125pt);
\fill[black] (-1,0) circle (3.125pt);
\draw(0,1) arc (90:270:1cm);
\draw(0,1.1) node[above]{$\graphfont{P}_1$};
\end{scope}
\begin{scope}[xshift = 3cm,yshift=1.5cm]
\fill[black] (0,-1) circle (3.125pt);
\fill[black] (0,1) circle (3.125pt);
\fill[black] (1,0) circle (3.125pt);
\fill[black] (.707,-.707) circle (3.125pt);
\draw(0,-1) arc (-90:90:1cm);
\draw(0,-1.1) node[below]{$\graphfont{P}_{-1}$};
\end{scope}
\begin{scope}[xshift=0cm, yshift=3cm]
\fill[black] (0,-1) circle (3.125pt);
\fill[black] (0,1) circle (3.125pt);
\draw(0,-1) node[left]{$y$};
\draw(0,1) node[left]{$x$};
\end{scope}
\draw[->](10.2,3.9)--(10.8,3.6);
\draw[->](10.2,2.1)--(10.8,2.4);
\draw[->](1.2,2.4)--(1.8,2.1);
\draw[->](1.2,3.6)--(1.8,3.9);
\draw[->](4.2,3.9) --(4.8,3.6);
\draw[->](4.2,2.1) --(4.8,2.4);
\draw[->](4.2,.9) --(4.8,.6);
\draw[->](4.2,5.1) --(4.8,5.4);
\draw[->](7.2,3.6) -- (7.8,3.9);
\draw[->](7.2,2.4) -- (7.8,2.1);
\draw[->](7.2,.6) -- (7.8,.9);
\draw[->](7.2,5.4) -- (7.8,5.1);
\end{tikzpicture}
\caption{An example of an $\{x,y\}$-decomposition. Every diamond is a push-out in $\Top$.}
\label{figure:xy decomposition}
\end{figure}
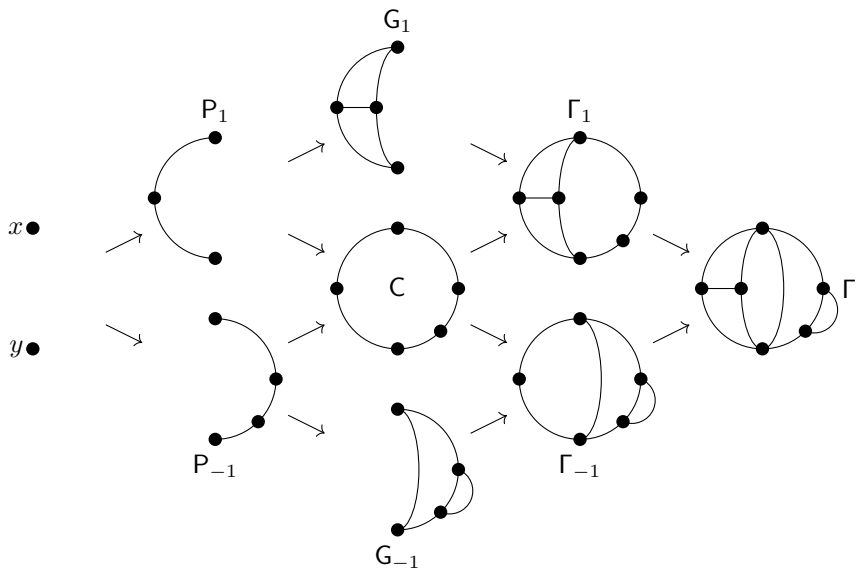

\begin{lemma}\label{lemma: xy decomposition retract}
Let $\graf$ be a $2$-connected simple graph with a $2$-cut $\{x,y\}$ and an $\{x,y\}$-decomposition. The natural map of \'{S}wi\k{a}tkowski complexes $\intrinsic{\graphfont{P}_i}\to \intrinsic{\graphfont{G}_i}$ admits a retraction $\intrinsic{\graphfont{G}_i}\to \intrinsic{\graphfont{P}_i}$ with the property that generators of $\intrinsic{\graphfont{G}_i}$ which avoid $x$ and its half-edges (respectively $y$ and its half-edges) are taken to sums of generators of $\intrinsic{\graphfont{P}_i}$ which avoid $x$ and its half-edge (respectively $y$ and its half-edge).
\end{lemma}
We think of $\graphfont{G}_{\pm 1}$ and $\graphfont{P}_{\pm 1}$ as relatively simpler subgraphs of the $\{x,y\}$-decomposition. In Lemma~\ref{lem:2-cut}, we will use these retractions to build similar retractions involving the more complicated subgraphs $\graf$ and $\graf_{\pm 1}$.
\begin{proof}
Suppose we are given a retraction $f:S_{\leq m}(\graphfont{G}_i)+\intrinsic{\graphfont{P}_i}\to \intrinsic{\graphfont{P}_i}$ with the desired property. The problem of extending this map to a generator $b$ of degree $m+1$ is that of choosing a nullhomotopy in $\intrinsic{\graphfont{P}_i}$ for $f(\partial(b))$ having the desired property. Since $\intrinsic{\graphfont{P}_i}$ and its reductions at $x$, at $y$, and at $\{x,y\}$ are all contractible in each weight, this extension is unobstructed as long as $m>0$. Thus, by induction on $m$, it will suffice to produce a retraction $f:S_{\le 1}(\graphfont{G}_i)+\intrinsic{\graphfont{P}_i}\to \intrinsic{\graphfont{P}_i}$ with the desired property. For this, we extend the identity on $S_{\le 1}(\graphfont{P}_i)$ by 
\begin{enumerate}\item 
sending each edge or vertex not lying in $\graphfont{P}_i$ to a fixed arbitrarily chosen edge $e$ of $\graphfont{P}_i$, 
\item sending each half-edge whose vertex is not in $\graphfont{P}_i$ to zero, and
\item sending each half-edge whose vertex $v$ is in $\graphfont{P}_i$ to a chain in $\intrinsic{\graphfont{P}_i}$ (avoiding $x$ and $y$ except potentially at $v$) whose boundary is the difference between $v$ and $e$. This is possible because $\graphfont{P}_i$ is topologically an interval with endpoints $x$ and $y$.\qedhere
\end{enumerate}
\end{proof}

\begin{lemma}[{\cite[Lem.~3.13]{KoPark:CGBG}; see also~\cite[Section 4.3]{HarrisonKeatingRobbinsSawicki:NPQSG}}]\label{lem:2-cut}
Let $\graf$ be a $2$-connected simple graph with a $2$-cut $\{x,y\}$ and an $\{x,y\}$-decomposition. Then the sequence
\[
0\to H_1(B_2(\cyclegraph{}))\to H_1(B_2(\graf_1))\oplus H_1(B_2(\graf_{-1}))
\to H_1(B_2(\graf))\to 0
\]
is split exact, where the maps are induced by the respective inclusions, with the map $H_1(B_2(\graf_{-1}))\to H_1(B_2(\graf))$ twisted by an overall sign.
\end{lemma}
\begin{proof}
The composite of the middle two maps is the sum of the map induced by the inclusion of $\cyclegraph{}$ into $\graf$ with its negative, hence zero. For injectivity of the second map, we appeal to the argument of Lemma~\ref{lemma: triconnected}, which shows that the loop class corresponding to $\cyclegraph{}$ is non-zero. For surjectivity of the third map, by $2$-connectedness of $\graf$, any graph morphism of $\stargraph{3}$ to $\graf$ can be extended to an embedding of some lollipop graph. By Proposition~\ref{prop:H1 generators} and the $Q$-relation it suffices to show that classes of the form $e\gamma$ lie in the image, where $e$ is an edge and $\gamma$ a loop class.

If $e$ does not lie in the representing loop of $\gamma$, then by $2$-connectedness of $\graf$ we may write $e\gamma=e\gamma'+e\gamma''$ where $e$ lies within the representing loops of both $\gamma'$ and $\gamma''$. By the $O$-relation, any two choices for $e$ within $\gamma'$ represent the same class. This reduces us to classes $e\gamma$ where $\gamma$ passes through both $\graphfont{G}_1$ and $\graphfont{G}_{-1}$ and $e$ lies in $\graphfont{G}_1$. Then $e\gamma$ can be rewritten as $e\gamma_{1}+e\gamma_{-1}$, where $\gamma_i$ is a loop class in $\graf_i$ for $i\in\{-1,1\}$. The first of these lies in the image of $H_1(B_2(\graf_1))$, and the second lies in the image of $H_1(B_2(\graf_{-1}))$ by the connectivity property of $\graphfont{G}_1$.

It remains to show exactness and splitting, for which we will use the retracts of Lemma~\ref{lemma: xy decomposition retract}. A generator of $\intrinsic{\graf}$ is a product of edges, half-edges, and vertices of $\graphfont{G}_1$ and $\graphfont{G}_{-1}$ (which intersect only at the vertices $x$ and $y$), and so, by applying these retracts on each of these two pieces, we obtain retracts $\pi_i:\intrinsic{\graf}\to \intrinsic{\graf_i}$ and $\rho_i:\intrinsic{\graf_i}\to \intrinsic{\cyclegraph{}}$ of the maps induced by the respective inclusions. 
We only have well-definition of these retracts because of the ``boundary conditions'' on generators involving $x$ and $y$ in the Lemma~\ref{lemma: xy decomposition retract}. Moreover, because the maps of Lemma~\ref{lemma: xy decomposition retract} are \emph{retracts}, the following diagram commutes:
\[
\begin{tikzcd}
	\intrinsic{\graf_i}\ar{r}{\rho_i}\ar[d]&\intrinsic{\cyclegraph{}}\ar[d]\\
	\intrinsic{\graf}\ar{r}[swap]{\pi_{-i}}&\intrinsic{\graf_{-i}}.
\end{tikzcd}
\] 
The map $\rho_1$ splits the sequence.  For exactness, consider the composite
\[H_1(B_2(\graf_1))\oplus H_1(B_2(\graf_{-1}))\to H_1(B_2(\graf))\to H_1(B_2(\graf_1))\oplus H_1(B_2(\graf_{-1})),\] where the second map is $(\pi_1, -\pi_{-1})$. Any map in the kernel of the first map is a fortiori in the kernel of the composition. Now writing $\iota_*$ for any map on homology induced by an inclusion, and using the commuting diagram above, this composition is given by \[(\beta_1,\beta_{-1})\mapsto(\beta_1-\iota_*(\rho_{-1})_*(\beta_{-1}),\beta_{-1}-\iota_*(\rho_1)_*(\beta_1)).\] This expression vanishes if and only if $\beta_1=\iota_*(\rho_1)_*(\beta_1)$, in which case $(\beta_1,\beta_{-1})$ is the image of $(\rho_1)_*(\beta_1)$, showing exactness.
\end{proof}

\bibliographystyle{amsalpha}

\begin{thebibliography}{HKRS14}

\bibitem[Abr00]{Abrams:CSBGG}
A.~Abrams, \emph{Configuration spaces and braid groups of graphs}, Ph.D.
  thesis, UC Berkeley, 2000.

\bibitem[ADCK]{AnDrummond-ColeKnudsen:ESHGBG}
B.~H. An, G.~C. Drummond-Cole, and B.~Knudsen, \emph{Edge stabilization in the
  homology of graph braid groups}, arXiv:1806.05585.

\bibitem[AF15]{AyalaFrancis:FHTM}
D.~Ayala and J.~Francis, \emph{Factorization homology of topological
  manifolds}, J. Topology \textbf{8} (2015), no.~4, 1045--1084.

\bibitem[AFT17]{AyalaFrancisTanaka:FHSS}
D.~Ayala, J.~Francis, and H.~Tanaka, \emph{Factorization homology of stratified
  spaces}, Selecta Math. (N.S.) \textbf{23} (2017), no.~1, 293--362.

\bibitem[AP17]{AnPark:OSBGC}
B.~H. An and H.~W. Park, \emph{On the structure of braid groups on complexes},
  Topology Appl. \textbf{226} (2017), no.~1, 86--119.

\bibitem[Arn69]{Arnold:CRGDB}
V.~Arnold, \emph{The cohomology ring of the group of dyed braids}, Mat. Zametki
  \textbf{5} (1969), 227--231.

\bibitem[BC88]{BodigheimerCohen:RCCSS}
C.-F. B{\"o}digheimer and F.~Cohen, \emph{Rational cohomology of configuration
  spaces of surfaces}, Algebraic Topology and Transformation Groups, Lecture
  Notes in Math., vol. 1361, Springer, 1988, pp.~7--13.

\bibitem[BCT89]{BodigheimerCohenTaylor:OHCS}
C.-F. B{\"o}digheimer, F.~Cohen, and L.~Taylor, \emph{On the homology of
  configuration spaces}, Topology \textbf{28} (1989), 111--123.

\bibitem[Bel04]{Bellingeri:OPSBG}
P.~Bellingeri, \emph{On presentations of surface braid groups}, J. Algebra
  \textbf{274} (2004), no.~2, 543--563.

\bibitem[BM08]{BondyMurty:GT}
A.~Bondy and M.~R. Murty, \emph{Graph theory}, Grad. Texts in Math., vol. 244,
  Springer-Verlag, 2008.

\bibitem[B{\"o}d87]{Bodigheimer:SSMS}
C.-F. B{\"o}digheimer, \emph{Stable splittings of mapping spaces}, Algebraic
  topology (H.~Miller and D.~Ravenel, eds.), Lecture Notes in Math., vol. 1286,
  Springer, 1987, pp.~174--187.

\bibitem[CE80]{CunninghamEdmonds:CDT}
W.~H. Cunningham and J.~Edmonds, \emph{A combinatorial decomposition theory},
  Canad. J. Math. \textbf{32} (1980), 734--765.

\bibitem[CLM76]{CohenLadaMay:HILS}
F.~Cohen, T.~Lada, and J.~P. May, \emph{The homology of iterated loop spaces},
  Lecture Notes in Math., no. 533, Springer, 1976.

\bibitem[DI04]{DuggerIsaksen:THA1R}
D.~Dugger and D.~Isaksen, \emph{Topological hypercovers and
  $\mathbb{A}^1$-realizations}, Math. Z. \textbf{246} (2004), no.~4, 667--689.

\bibitem[Dug17]{Dugger:PHC}
D.~Dugger, \emph{A primer on homotopy colimits}, preprint:
  \url{http://pages.uoregon.edu/ddugger/hocolim.pdf}, 2008--2017.

\bibitem[Far03]{Farber:TCMP}
M.~Farber, \emph{Topological complexity of motion planning}, Discrete Comput.
  Geom. \textbf{29} (2003), 211--221.

\bibitem[Far05]{Farber:CFMPG}
\bysame, \emph{Collision free motion planning on graphs}, Algorithmic
  Foundations of Robotics VI (M.~Erdmann, M.~Overmars, D.~Hsu, and F.~van~der
  Stappen, eds.), Springer Tracts Adv. Robot., vol.~17, Springer, 2005,
  pp.~123--138.

\bibitem[Far06]{Farley:HTBG}
D.~Farley, \emph{Homology of tree braid groups}, preprint:
  \url{http://www.users.miamioh.edu/farleyds/grghom.pdf}, 2006.

\bibitem[FN62]{FadellNeuwirth:CS}
E.~Fadell and L.~Neuwirth, \emph{Configuration spaces}, Math. Scand.
  \textbf{10} (1962), 111--118.

\bibitem[For98]{Forman:MTCC}
R.~Forman, \emph{Morse theory for cell complexes}, Adv. Math. \textbf{134}
  (1998), no.~1, 90--145.

\bibitem[FS05]{FarleySabalka:DMTGBG}
D.~Farley and L.~Sabalka, \emph{Discrete {M}orse theory and graph braid
  groups}, Alg. Geom. Topol. \textbf{5} (2005), 1075--1109.

\bibitem[FS08]{FarleySabalka:OCRTBG}
\bysame, \emph{On the cohomology rings of tree braid groups}, J. Pure Appl.
  Algebra \textbf{212} (2008), no.~1, 53--71.

\bibitem[FS12]{FarleySabalka:PGBG}
\bysame, \emph{Presentations of graph braid groups}, Forum Math. \textbf{24}
  (2012), 827--860.

\bibitem[FT00]{FelixThomas:RBNCS}
Y.~F\'elix and J.-C. Thomas, \emph{Rational {B}etti numbers of configuration
  spaces}, Topology Appl. \textbf{102} (2000), 139--149.

\bibitem[Gal01]{Gal:ECCSC}
\'{S}. Gal, \emph{Euler characteristic of the configuration space of a
  complex}, Colloq. Math. \textbf{89} (2001), no.~1, 61--67.

\bibitem[Ghr02]{Ghrist:CSBGGR}
R.~Ghrist, \emph{Configuration spaces and braid groups on graphs in robotics},
  Knots, braids, and mapping class groups---papers dedicated to {J}oan {S}.
  {B}irman (New York, 1998), AMS/IP Stud. Adv. Math., vol.~24, Amer. Math.
  Soc., 2002, pp.~29--40.

\bibitem[Gin13]{Ginot:NFAFHA}
G.~Ginot, \emph{Notes on factorization algebras, factorization homology and
  applications}, preprint: \url{https://arxiv.org/abs/1307.5213}, 2013.

\bibitem[GJ09]{GoerssJardine:SHT}
P.~Goerss and J.~F. Jardine, \emph{Simplicial homotopy theory}, Modern
  Birkh\"{a}user Classics, Birkh\"{a}user, 2009.

\bibitem[GK98]{GhristKoditschek:SCRPVDOG}
R.~Ghrist and D.~Koditschek, \emph{Safe cooperative robot patterns via dynamics
  on graphs}, Robotics Research: The Eighth International Symposium (Y.~Shirai
  and S.~Hirose, eds.), Springer, 1998, pp.~81--92.

\bibitem[Hir02]{Hirschhorn:MCL}
P.~Hirschhorn, \emph{Model categories and their localizations}, Math. Surveys
  Monogr., vol.~99, Amer. Math. Soc., Providence, Rhode Island, 2002.

\bibitem[HKRS14]{HarrisonKeatingRobbinsSawicki:NPQSG}
J.~M. Harrison, J.~P. Keating, J.~M. Robbins, and A.~Sawicki, \emph{n-particle
  quantum statistics on graphs}, Comm. Math. Phys. \textbf{330} (2014),
  1293--1326.

\bibitem[IMW16]{IvanovMikhailovWu:ONCHGS}
S.~Ivanov, R.~Mikhailov, and J.~Wu, \emph{On nontriviality of certain homotopy
  groups of spheres}, Homology Homotopy Appl. \textbf{18} (2016), no.~2,
  337--344.

\bibitem[KKP12]{KimKoPark:GBGRAAG}
J.~H. Kim, H.~K Ko, and H.~W. Park, \emph{Graph braid groups and right-angled
  {A}rtin groups}, Trans. Amer. Math. Soc. \textbf{364} (2012), 309--360.

\bibitem[Knu17]{Knudsen:BNSCSVFH}
B.~Knudsen, \emph{Betti numbers and stability for configuration spaces via
  factorization homology}, Alg. Geom. Topol. \textbf{17} (2017), no.~5,
  3137--3187.

\bibitem[Kon92]{Kontsevich:ITMSCMAF}
M.~Kontsevich, \emph{Intersection theory on the moduli space of curves and
  matrix {A}iry function}, Comm. Math. Phys. \textbf{147} (1992), no.~1, 1--23.

\bibitem[KP12]{KoPark:CGBG}
K.~H. Ko and H.~W. Park, \emph{Characteristics of graph braid groups}, Discrete
  Comput. Geom. \textbf{48} (2012), no.~4, 915--963.

\bibitem[KS06]{KashiwaraSchapira:CS}
M.~Kashiwara and P.~Schapira, \emph{Categories and sheaves}, Grundlehren Math.
  Wiss., vol. 332, Springer, 2006.

\bibitem[LS05]{LongoniSalvatore:CSANHI}
R.~Longoni and P.~Salvatore, \emph{Configuration spaces are not homotopy
  invariant}, Topology \textbf{44} (2005), no.~2, 375--380.

\bibitem[Lur09]{Lurie:HTT}
J.~Lurie, \emph{Higher topos theory}, Ann. of Math. Stud., Princeton University
  Press, 2009.

\bibitem[L{\"{u}}t14]{Luetgehetmann:CSG}
Daniel L{\"{u}}tgehetmann, \emph{Configuration spaces of graphs}, Master's
  thesis, Freie Universit\"{a}t Berlin, 2014.

\bibitem[LW69]{LundellWeingram:TCWC}
A.~Lundell and S.~Weingram, \emph{The topology of {C}{W} complexes}, The
  University Series in Higher Mathematics, vol.~4, Springer, 1969.

\bibitem[McD75]{McDuff:CSPNP}
D.~McDuff, \emph{Configuration spaces of positive and negative particles},
  Topology \textbf{14} (1975), 91--107.

\bibitem[Men27]{Menger:SAK}
K.~Menger, \emph{Zur allgemeinen {K}urventheorie}, Fund. Math. \textbf{10}
  (1927), no.~1, 96--115.

\bibitem[MS17]{MaciazekSawicki:HGP1CG}
T.~Maci{\k{a}}\.{z}ek and A.~Sawicki, \emph{Homology groups for particles on
  one-connected graphs}, J. Math. Phys. \textbf{58} (2017), no.~6.

\bibitem[Ram18]{Ramos:SPHTBG}
E.~Ramos, \emph{Stability phenomena in the homology of tree braid groups}, Alg.
  Geom. Topol. \textbf{18} (2018), no.~4, 2305--2337.

\bibitem[Rie14]{Riehl:CHT}
E.~Riehl, \emph{Categorical homotopy theory}, New Mathematical Monographs,
  vol.~24, Cambridge University Press, 2014.

\bibitem[Sab09]{Sabalka:ORIPTBG}
L.~Sabalka, \emph{On rigidity and the isomorphism problem for tree braid
  groups}, Groups Geom. Dyn. \textbf{3} (2009), no.~3, 469--523.

\bibitem[Seg73]{Segal:CSILS}
G.~Segal, \emph{Configuration-spaces and iterated loop-spaces}, Invent. Math.
  \textbf{21} (1973), 213--221.

\bibitem[Smi67]{Smith:HAEMSS}
L.~Smith, \emph{Homological algebra and the {E}ilenberg--{M}oore spectral
  sequence}, Trans. Amer. Math. Soc. \textbf{129} (1967), no.~1, 58--93.

\bibitem[{\'{S}}wi01]{Swiatkowski:EHDCSG}
Jacek {\'{S}}wi\k{a}tkowski, \emph{Estimates for homological dimension of
  configuration spaces of graphs}, Colloq. Math. \textbf{89} (2001), no.~1,
  69--79.

\bibitem[Tot96]{Totaro:CSAV}
B.~Totaro, \emph{Configuration spaces of algebraic varieties}, Topology
  \textbf{35} (1996), no.~4, 1057--1067.

\bibitem[Tut66]{Tutte:CG}
W.~T. Tutte, \emph{Connectivity in graphs}, Mathematical Expositions, vol.~15,
  Toronto University Press, 1966.

\bibitem[Wag37]{Wagner:UEEK}
K.~Wagner, \emph{{\"U}ber eine {E}igenschaft der ebenen {K}omplexe}, Math. Ann.
  \textbf{114} (1937), no.~1, 570--590.

\bibitem[WG]{Wiltshire-Gordon:MCSSC}
J.~D. Wiltshire-Gordon, \emph{Models for configuration space in a simplicial
  complex}, to appear in Colloq. Math.

\end{thebibliography}
\providecommand{\bysame}{\leavevmode\hbox to3em{\hrulefill}\thinspace}
\providecommand{\MR}{\relax\ifhmode\unskip\space\fi MR }
\providecommand{\MRhref}[2]{%
  \href{http://www.ams.org/mathscinet-getitem?mr=#1}{#2}
}
\providecommand{\href}[2]{#2}

\Addresses
\end{document}